\documentclass[a4paper, 12pt, reqno]{amsart} 
\usepackage{epstopdf}
\usepackage{amssymb}
\usepackage{amsmath, amsthm, amscd, amsfonts}
\usepackage{graphicx, float}
\usepackage{mathrsfs}
\usepackage{bbm}
\usepackage[left= 2.4 cm, right= 2.4 cm, top= 2.5 cm, bottom=2.5 cm, foot= 1 cm]{geometry}
\usepackage{abstract}

\newcommand{\mb}{\mathbb}

\newcommand{\mc}{\mathcal}

\newcommand{\ov}{\overline}
\newcommand{\ora}{\overrightarrow}

\makeatletter 
\@addtoreset{equation}{section}
\makeatother  

\begin{document}
\title{Moduli space of twisted holomorphic maps with Lagrangian boundary condition: compactness}
\author{Guangbo Xu}

\newtheorem{thm}{Theorem}[section]
\newtheorem{lemma}[thm]{Lemma}
\newtheorem{cor}[thm]{Corollary}
\newtheorem{prop}[thm]{Proposition}

\theoremstyle{definition}
\newtheorem{defn}[thm]{Definition}

\theoremstyle{remark}
\newtheorem{rem}[thm]{Remark}

\maketitle

\begin{abstract}

Let $(X, \omega)$ be a compact symplectic manifold and $L$ be a Lagrangian submanifold. Suppose $(X, L)$ has a Hamiltonian $S^1$ action with moment map $\mu$. Take an invariant $\omega$-compatible almost complex structure, we consider tuples $(C, P, A, \varphi)$ where $C$ is a smooth bordered Riemann surface of fixed topological type, $P\to C$ is an $S^1$-principal bundle, $A$ is a connection on $P$ and $\varphi$ is a section of $P\times_{S^1} X$ satisfying
$$\ov\partial_A \varphi=0,\ \iota_\nu F_A+ \mu(\varphi)=c$$
with boundary condition $\varphi(\partial C) \subset P \times_{S^1}  L$. Here $F_A$ is the curvature of $A$ and $\nu$ is a volume form on $C$ and $c\in i{\mb R}$ is a constant. 

We compactify the moduli space of isomorphism classes of such objects with energy $\leq K$, where the energy is defined to be the Yang-Mills-Higgs functional
$$\left\| F_A \right\|_{L^2}^2+ \left\| d_A\varphi \right\|_{L^2}^2+ \left\| \mu(\varphi)-c \right\|_{L^2}^2.$$
This generalizes the compactness theorem of Mundet-Tian \cite{Mundet_Tian_2009} in the case of closed Riemann surfaces. 
\end{abstract}

\tableofcontents

\section{Introduction} 

The study of twisted holomorphic maps connects two topics in geometry: gauge theory over Riemann surfaces and Gromov-Witten theory. On the gauge-theoretical side, the objects are solutions of the vortex equations over Riemann surfaces(see \cite{Jaffe_Taubes}, \cite{Bradlow_1},\cite{Bradlow_2}, \cite{Garcia-Prada_1}, \cite{Garcia-Prada_2}, \cite{Garcia-Prada_3}); the basic object in Gromov-Witten theory is the pseudoholomorphic curves in almost K\"ahler manifolds. Independently, Ignasi Mundet\cite{Mundet_2003} and Cieliebak-Gaio-Salamon \cite{Cieliebak_Gaio_Salamon_2000} discovered the objects combining the two theories, which is called the twisted holomorphic maps, or gauged holomorphic maps. It has its physical counterpart called the gauged $\sigma$-model.

We briefly describe such objects. The precise definition will be given in Section 2. Let $(X, \omega)$ be a symplectic manifold endowed with a Hamiltonian $S^1$-action. Let $\mu: X\to i{\mb R}\simeq \left( {\rm Lie} S^1\right)^*$ be a moment map. Let $c\in i{\mb R}$ be a constant. In addition let's fix an invariant $\omega$-compactible almost complex structure $I$ on $X$. Let $\Sigma$ be a Riemann surface and $P\to \Sigma$ is an $S^1$-principal bundle. We consider pairs $(A, \varphi)$ where $A$ is a connection on $P$ and $\varphi: \Sigma\to P\times_{S^1} X$ is a section such that
\begin{align}\label{11}
\left\{ \begin{array}{ccc} \ov\partial_A \varphi & = & 0 , \\
 \iota_\nu F_A+ \mu(\varphi)& = & c. \end{array}\right. 
\end{align}
Here $\nu$ is a volume form on $\Sigma$, $F_A$ is the curvature 2-form of $A$ and $\mu(\varphi)$ is well-defined on $\Sigma$. The energy of such solutions is defined as
$$\left\| F_A \right\|_{L^2}^2 + \left\| \mu(\varphi) -c \right\|_{L^2}^2 + \left\| d_A \varphi \right\|_{L^2}^2.$$
And it turns out to be equal to the pairing between the equivariant symplectic form $[\omega- (\mu-c) ]\in H^2_{S^1}( X)$ and the fundamental class $[(P, \phi) ] \in H_2^{S^1}(X)$, if $(P, A,\phi)$ is a solution to the above equation.

We consider the moduli space of such objects(with energy bounded by a constant). As in Gromov-Witten theory, the moduli space is in general noncompact. The incompactness of the moduli space is caused by the bubbling-off phenomenon. Then, as in \cite{Mundet_2003}, we can use ``stable'' objects to compactify the moduli space and use the compactified moduli to define invariants
\begin{align}\label{invariant}
\Psi_{g, n}\left( \alpha_1, \ldots, \alpha_n| A\right)\in {\mb Q},\ A\in H_2^{S^1}(X),\  \alpha_i\in H^*_{S^1}(X).
\end{align}

However, one shall consider the case when the conformal structure of the curve(with marked points) varies in the Deligne-Munford moduli. This is necessary, in particular, if we want to study the fusion rule of the invariants and define quantum cup product on $H^*_{S^1}(X)$. In this situation, there is a new phenomenon absent in Gromov-Witten theory. Namely, when a node on the curve is forming, the two sides of the node don't always connect in the image of the section, but might be disjoint and connected by a gradient line of the function $h=-i\mu$, when the connection has a ``critical holonomy'' around the node. To define the objects in the compactification(i.e. stable maps) and to prove the compactness theorem, one needs to study delicately the behavior of twisted holomorphic maps from conformally long cylinders. The details was given in \cite{Mundet_Tian_2009}. 

In this paper we consider twisted holomorphic maps from bordered Riemann surfaces, and prove the compactness theorem of the corresponding moduli space. One impose suitable boundary conditions so that the system (\ref{11}) is elliptic. The boundary condition for $\varphi$ is that it should maps boundary of the Riemann surface to a Lagrangian submanifold $L$. Since there is an $S^1$-ambiguity of the value of $\varphi$, $L$ should be $S^1$-invariant. The boundary condition for the connection $A$ is essentially a gauge fixing condition(e.g. Coulomb gauge), so doesn't appear explicitly.

Compared with the case of closed Riemann surface, there are two more types of degenerations we have to consider. Namely, the shrinking of a boundary circle and shrinking of an arc. In the first case, we have to study behaviors of twisted holomorphic maps on annuli $\Delta(1, r)=\left\{r\leq |z|\leq 1 \right\}$ as $r\to 0$; this is called a type-2 boundary node in our context, and this is essentially the same as in the case of closed Riemann surfacc. In the second case, we have to study behaviors of twisted holomorphic maps on long strips.

\subsection{Notations} 

We view the group $S^1$ as the set of complex numbers with modulus 1, and sometimes identify $S^1$ with ${\mb R}/2\pi {\mb Z}$ via the map ${\mb R}/2\pi {\mb Z}\ni \theta\mapsto e^{i\theta}\in S^1$.

We will identify ${\rm Lie}S^1\simeq i{\mb R}\simeq \left( {\rm Lie}S^1 \right)^*$. The pairing between $i{\mb R}$ and ${\rm Lie}S^1\simeq i{\mb R}$ is given by the $\langle ia, ib\rangle =ab$($a, b\in {\mb R}$).

Throughout this paper, $(X, \omega)$ will be a compact symplectic manifold. We will fix an effective Hamiltonian $S^1$ action on $X$ with a moment map $\mu:X\to i{\mb R}$. For any $\lambda\in i{\mb R}$, we define
$$X^\lambda=\left\{ x\in X \left| e^{2\pi \lambda}\cdot x =x  \right. \right\}.$$ 
We denote by ${\mc X}\in \Gamma(TX)$ the infinitesimal action of $i$. Define the real valued function $h=-i\mu=\langle i, \mu\rangle$. The gradient of $h$ with respect to $g$ is equal to $I{\mc X}$.

We also take an $S^1$-invariant almost complex structure $I$ on $X$, compatible with $\omega$ in the sense that $g:=\omega(\cdot, I\cdot)$ is a Riemannian metric on $X$. Unless otherwise mentioned, any statement involving a metric on $X$ will implicitly refer to the metric $g$. Let $d$ be the distance function on $X$ defined by the metric $g$. Define the pseudodistance
$${\rm dist}_{S^1}\left( x, y \right):=\inf_{\theta\in S^1} d\left( x, \theta\cdot y \right).$$
If $A, B\subset X$ are subsets, define $${\rm diam}_{S^1} A:=\sup_{x, y\in A} {\rm dist}_{S^1} d(x, y)$$ and
$${\rm dist}_{S^1}:= \inf_{x\in A, y\in B} {\rm dist}_{S^1}(x, y).$$
We denote by $d_{\rm Hauss}$ the Hausdorff distance between subsets: $$ d_{\rm Hauss}(A, B):=\sup_{x\in A} d(x, B)+\sup_{y\in B} d(y, A).$$

We assume that $L\subset X$ is an embedded compact Lagrangian submanifold, which is invariant under the $S^1$ action. So ${\mc X}|_L\subset \Gamma(TL)$. Moreover, $\mu|_L$ is locally constant.

A {\it curve} will be a 1-dimensional complex manifold, which may have at worst nodal singularities, not necessarily compact, may or may not have boundary.

There is a positive number $\epsilon_X= \epsilon_{X, I, \omega, L}$ such that for any nontrivial $I$-holomorphic sphere or nontrivial $I$-holomorphic disk with boundary in $L$, its energy is at least $\epsilon_X$.

\subsection{Conventions}

We will use the letter $P$ to denote an $S^1$ principal bundle over some base $B$. The letter $A$ will denote a connection, or a connection 1-form on $P$. A section of the associated bundle $P\times_{S^1} X$ will be viewed either as a map from $B$ to $P\times_{S^1} X$, or as an equivariant map from $P$ to $X$. In the first case, we will use the letter $\varphi$ and in the second case we use $\Phi$. The correspondence $\Phi\leftrightarrow \varphi$ will be used implicitly. If we trivialize $P$, then the covariant derivative of $A$ can be written as $d+\alpha$, where $\alpha$ is an $i{\mb R}$-valued 1-form on $B$, and the section can be viewed as a map $\phi: B\to X$. The correspondence $A\leftrightarrow \alpha$ and $\varphi\leftrightarrow \phi$ will also be used everywhere in this paper.

The letter $\Sigma$ will denote a {\rm smooth} Riemann surface, with or without boundary. The letter $C$ will denote a more general curve, possibly with nodal singularities. ${\bf z}$(resp. ${\bf x}$) will denote an {\bf interior}(resp. {\bf boundary}) marked point set, and its elements will be denoted by $z^1, \ldots $(resp. $ x^1, \ldots $). ${\bf w}$ will denote the set of nodal points of $C$.

\subsection{Organization of the paper}

In section 2, we recall the foundational set up and necessary results proved in \cite{Mundet_Tian_2009}. In section 3, we consider the analogous set up in the case of bordered Riemann surfaces, and prove that the boundary singularities essentially can be removed. In section 4 and 5, we will prove several results which treat two types of boundary degeneration respectively. In section 6, we give the definition of stable maps in this case. In section 7, we show that in our moduli space, there are only finitely many different topological type of the underlying curve, if the energy is bounded by some constant. In section 8 we define the topology of the moduli space and in the final section we prove the compactness theorem. 

\subsection{Acknowledgements}

I want first to thank Professor Gang Tian, for introducing the me into this field and for constant help and encouragement. I want to thank Ignasi Mundet i Riera for helpful discussions and kindly answering my questions. I also want to thank Mohommad Farajzadeh Tehrani, for general discussions on symplectic geometry. During the preparation of this paper, Mundet and Tian shared their drafts \cite{Mundet_Tian_Draft} with the author. I also thank them for their generosity. I want to thank Conan Wu from whom I learned drawing pictures using Adobe Illustrator CS.

\section{Preliminaries and review of results in \cite{Mundet_Tian_2009}}

\subsection{Twisted holomorphic maps}

Let $C$ be a complex curve (not necessarily compact). Let $I_C$ denote the complex structure on $C$ and $g_C$ be a conformal metric. Let $P$ be an $S^1$ principal bundle over $C$. Let $Y:=P\times_{S^1} X$ and $\pi: Y\to C$ be the natural projection. Let $T^{vert}Y\subset TY$ be the subbundle of the tangent bundle of $Y$ of all vertical tangent vectors. If we choose an $S^1$-invariant almost complex structure $I_X$(resp. an $S^1$-invariant Riemannian metric $g_X$) on $X$, then $I_X$(resp. $g_X$) induces a complex structure(resp. a metric) on the bundle $T^{vert}Y$.

Any connection $A$ on $P$ induces a splitting $TY=\pi^* TC\oplus T^{vert}Y$. With this one obtain an almost complex structure $I_X(A)$(resp. a  metric $g(A)$) on $TY$ as the direct sum of $\pi^* I_C$(resp. $\pi^* g_C$) and the almost complex structure(resp. the metric) on $T^{vert}Y$.

Let $\varphi: C\to Y$ be a section of the bundle $Y\to C$, which is equivalent to a anti-equivariant map $\Phi: P\to X$. In the following we will consider the pair $(A, \varphi)$ where $A$ is a connection on $P$ and $\varphi$ is a section.

\subsubsection{Covariant derivatives and $\ov\partial$-operators}

The {\bf covariant derivative} $d_A\varphi$ of $\varphi$ with respect to $A$ is given by
$$ d_A\varphi =\pi_A\circ d\varphi \in \Omega^1 \left( C, \varphi^*T^{vert}Y \right), $$
where $\pi_A: TY\to T^{vert}Y$ is the projection induced by the splitting given by $A$.

Since $TC$ and $T^{vert}Y$ are both complex vector bundles, we can define the operator $\ov\partial_A$ by taking the $(0, 1)$-part of $d_A$. More precisely, for any section $\varphi: C\to Y$,
 $$\ov\partial_A \varphi ={1\over 2} \left( d_A\varphi + I_X \circ d_A \varphi \circ I_C \right).$$

If $\ov\partial_A \varphi =0$, we say that $\varphi $ is {\bf $A$-holomorphic }.

\subsubsection{Twisted holomorphic pairs}

Let $P=C\times S^1$ be the trivial $S^1$-bundle and let $\alpha\in \Omega^1(C, i{\mb R})$. Then $\alpha$ induces a connection $A_\alpha=d+\alpha$ on $P$. The associated bundle $Y=P\times_{S^1} X$ is trivial and a section of $Y$ is identified with a map $\phi: C\to X$. Then the pair $(\alpha, \phi)$ can be viewed as the local representation of a pair $(A_\alpha, \varphi)$  under this trivialization. We say that the pair $(\alpha, \phi)$ is a {\bf twisted holomorphic pair}, if the corresponding section of $\varphi: C\to Y$ is $A_\alpha$-holomorphic in the above sense.

More precisely, the covariant derivative $d_{A_\alpha}\varphi$ is(under the given trivialization)
$$ d_\alpha \phi=d\phi- i\alpha{\mc X}(\phi)\in \Omega^1(C, \phi^* TX) $$
and
$$ \ov\partial_\alpha \phi=\ov\partial_I \phi-{1\over 2} \left(i\alpha {\mc X}(\phi)+i (\alpha\circ I_C)(I_X{\mc X})(\phi)\right),$$
where ${\mc X}\in \Gamma(TX)$ is the infinitesimal action of $S^1$ on $X$.

\subsubsection{Vortex equation}

Let $(A, \varphi)$ be a pair where $A$ is a connection on $P$ and $\varphi: C\to Y$ be a section. Let $c\in i{\mb R}$ be a constant. Let $\nu$ be a volume form on $C$. The {\bf vortex equation} with respect to $\nu$ on the pair $(A, \phi)$ is
\begin{align}\label{vortex}
\iota_\nu F_A+\mu(\varphi)=c.
\end{align}
Here since both $\varphi$ and $\mu$ are $S^1$-equivariant, $\mu(\varphi): P\to i{\mb R}$ defines a function on $C$, and the equation makes sense.

\subsubsection{Gauge transformations}

Suppose $h: C\to S^1$ is a $C^1$-map. It can be viewed as a gauge transformation on any $S^1$-principal bundle $P$ over $C$, and its action on the pair $(A, \varphi)$ is given by
$$ h^*(A, \varphi)=\left( A-d\log h, h(\varphi) \right).$$

Under a local trivialization with respect to which $(A, \varphi)$ corresponds a pair $(\alpha, \phi)$, the gauge transformation can be expressed as
$$ h^*(\alpha, \phi)=\left( \alpha- d\log h, h\phi \right).$$

It is easy to check that both the holomorphicity equation $\ov\partial_A\varphi=0$ and the vortex equation (\ref{vortex}) are invariant under gauge transformations.

\subsection{Meromorphic connections}

\subsubsection{Meromorphic connections on the punctured disc}

Let $P\to {\mb D}^*$ be a principal $S^1$-bundle over the punctured unit disc. Then $P$ is trivial but the set of homotopy classes of trivializations $T(P)$ has a natural structure on ${\mb Z}$-torsor\footnote{Recall that if $\Gamma$ is a group, a $\Gamma$-torsor is a set $T$ with a free left action of $\Gamma$.}, given by the gauge transformations. Let $A$ be a connection on $P$. We say that $A$ is a {\bf meromorphic connection} if its curvature $F_A$ is bounded with respect to the standard metric on ${\mb D}$. Let $\epsilon>0$ and let $\gamma_\epsilon\in S^1$ be the holonomy of $A$ around the circle of radius $\epsilon$ centered at the origin.

\begin{lemma}\label{lemma21}
\begin{enumerate}
\item The limit ${\rm Hol}(A, 0):=\lim_{\epsilon\to 0}\gamma_\epsilon$ exists.

\item Any $\tau\in T(P)$ has a representative with respect to which the covariant derivative associated to $A$ can be written in polar coordinates $(r, \theta)$ as
$$ d_A=d+\alpha+\lambda d\theta, $$
where $\alpha$ is a smooth form with values in $i{\mb R}$ on ${\mb D}^*$ which extends continuously to ${\mb D}$. $\lambda\in i{\mb R}$ satisfies ${\rm Hol}(A, 0)=e^{2\pi \lambda}$. Moreover, the number $\lambda$ depends on $A$ and $\tau$ but not on the representative of $\tau$.

\end{enumerate}
\end{lemma}

\begin{proof}
\begin{enumerate}
\item By Stokes' formula
$$\gamma_\epsilon\cdot \gamma_{\epsilon'}^{-1}=\exp \int_{ A(\epsilon, \epsilon')} F_A$$
where $A(\epsilon, \epsilon')$ is the annulus with radius $\epsilon, \epsilon'$. By the boundedness of $F_A$, we derive that $\lim_{\epsilon\to 0} \gamma_\epsilon$ exists.

\item Take an arbitrary representative of $\tau$, then $A$ is viewed as a 1-form. Let $$\lambda={1\over 2\pi}\lim_{\epsilon\to 0} \int_{C_\epsilon} A.$$
Its independence of representatives of $\tau$ is obvious. 

On the other hand, by Poincar\'e lemma, there exists a continuous 1-form $\alpha$ on ${\mb D}$, such that $d\alpha=F_A$ in the weak sense. Indeed, suppose $F_A= F(x, y) dx dy= F(x, y) r dr d\theta$ with $F$ a bounded function, just define 
$$ \alpha=\left\{\begin{aligned} 0,\ {\rm origin}\\
\left(\int_0^r F r dr\right) d\theta, \ {\mb D}^*\end{aligned}\right.$$
Then $\alpha$ is continuous at the origin. Indeed,
$$\left(\int_0^r Fr dr\right) d\theta={1\over r^2} \left(\int_0^r F r dr\right) (ydx-xdy)$$
and $$\lim_{r\to 0^+} {1\over r^2} \left(\int_0^r F r dr\right)\to 0.$$

Then consider the 1-form $A-\alpha-\lambda d\theta$, we have
$$\int_{C_\epsilon} (A-\alpha-\lambda d\theta)=0$$
for all $\epsilon>0$. So there exists a smooth function $h: {\mb D}^*\to i{\mb R}$ such that $$dh=A-\alpha-\lambda d\theta.$$
The gauge transformation $g=\exp h$ has winding number 0 and hence induces another representative for the same $\tau$, with repect to which $d_A= d+\alpha+\lambda d \theta$. 

\end{enumerate} 
\end{proof}

\subsubsection{Meromorphic connections on marked smooth closed Riemann surfaces}

Let $(\Sigma, {\bf z})$ be a marked smooth closed Riemann surface. Let $P\to \Sigma\setminus {\bf z}$ be an $S^1$ principal bundle and $A$ is a connection on $P$. We say that $A$ is a {\bf meromorphic connection} if its curvature $F_A$ is bounded with respect to any smooth metric on $\Sigma$.

\subsection{Critical residues}

Let $F\subset X$ be the fixed point set of the $S^1$-action. Each connected component $F'\subset F$ is an embedded submanifold of $X$, and the $S^1$ action on $X$ induces an action on the normal bundle $N_{F'}\to F'$, which splits in weights as $N=\oplus_{\chi\in {\mb Z}} N_{\chi}$. Define
$$ {\rm weight}(F'):=\left\{\chi\in {\mb Z}: N_{\chi}\neq 0 \right\}, $$
and
$$ {\rm weight}(X):=\bigcup_{F'\subset F} {\rm weight}(F')\subset {\mb Z}.$$

The set of critical residues is equal to
$$ \Lambda_{cr}=\left\{ \lambda\in i{\mb R}| \exists\ w\in {\rm weight}(X)\ s.t.\ w\lambda \in i{\mb Z} \right\}. $$

\subsection{Twisted holomorphic maps from cylinders}

The key point in proving the compactness theorem in \cite{Mundet_Tian_2009} is to study the behavior of the twisted holomorphic maps when a circle in the underlying curves shrink to a node in different cases, or equivalently, to understand the behavior of twisted holomorphic maps from long cylinders. We briefly recall the theorems in \cite[Section 4]{Mundet_Tian_2009} on this.

For any positive real number $N$, set $C_N=(-N-2, N+2)\times S^1$ and denote by $(t, \theta)$ the coordinate of points of $C_N$. Let $Z_N=[-N, N]\times S^1$. Consider both $C_N$ and $Z_N$ the standard flat metric $g=dt^2+d\theta^2$ and the conformal structure $\partial_t=i\partial_\theta$.

\subsubsection{Cylinders with noncritical residues}

The following theorem describes how the maps on long cylinders will behave, if the energy density is small and the holonomy around the loop $\{t\}\times S^1$ is away from critical holonomies. The result is very similar to the case of usual pseudoholomorphic maps.

\begin{thm}\label{thm22}
For any noncritical $\lambda\in i{\mb R}\setminus \Lambda_{cr} $, there exist positive real numbers $K_1$, $\sigma_1$, $\epsilon_1$, depending continuously on $\lambda$, with the following properties. Let $N>0$ be a real number and $(\alpha, \phi): C_N\to X$ be a twisted holomorphic pair satisfying $\left\| \alpha-\lambda d\theta \right\|_{L^\infty(C_N)}<\epsilon_1 $, $\left\| d\alpha \right\|_{L^\infty(C_N)}<\epsilon_1 $ and $\left\| d_\alpha\phi \right\|_{L^\infty(C_N)}<\epsilon_1 $. Then for any $(t, \theta)\in Z_N$ we have
$$ \left| d_\alpha \phi(t, \theta) \right| \leq K_1  e^{\sigma_1 (|t|-N)} \left\| d_\alpha\phi \right\|_{L^\infty(Z_N)}. $$
In particular this implies
$$
{\rm diam}_{S^1}(\phi(Z_N)) < K_1 \left\| d_\alpha\phi \right\|_{L^\infty(Z_N)}.
$$
\end{thm}
Note that the constants are independent of the cylinder length $N$.

\subsubsection{Limit orbits of twisted holomorphic pairs on the punctured disk}
Consider the semi-infinite cylinder $C_+={\mb R}_{>0}\times S^1$ with coordinates $(t, \theta)$ and with the flat metric $dg^2=dt^2+ d\theta^2$ and conformal structure $\partial_t=i\partial_\theta$.

\begin{thm}\label{thm23}
Suppose that $(\alpha, \phi): C_+\to X$ is a smooth twisted holomorphic pair, and that $\|d\alpha\|_{L^\infty}<\infty$ and $\|d_\alpha\phi\|_{L^2}<\infty$. Then

(1) For any $t\in {\mb R}_{>0}$ let ${\rm Hol}(t)$ be the holonomy of $\alpha$ around the circle $\{t\}\times S^1$. As $t\to\infty$, ${\rm Hol}(t)$ converges to some ${\rm Hol}(\infty)\in S^1$.

(2) There is an $S^1$-orbit ${\mc O}\subset X$ such that $d(\phi(t, \theta), {\mc O})\to 0$ as $t\to \infty$.

(3) The points in the orbit ${\mc O}$ are all stabilized by ${\rm Hol}(\infty)$.

(4) We have $\left| d_\alpha\phi(t, \theta) \right|\leq K e^{-\sigma t}$ for some positive constants $K, \sigma$.
\end{thm}

\subsubsection{Connections in balanced temporal gauge}

Let $P$ be a principal circle bundle over a cylinder $C=(p, q)\times S^1$ with $p<q$. Let $A$ be a connection on $P$. We say that a given trivialization of $P$ puts $A$ in {\bf temporal gauge} if, with respect to this trivialization, the covariant derivative $d_A$ takes the form $d+ad\theta$, where $a$ is an $i{\mb R}$-valued function. Let $\tau=(p+q)/2$. We say that the trivialization is in {\bf balanced temporal gauge} if the restriction of $a$ to the middle circle $\{\tau\}\times S^1$ is equal to some constant $\lambda\in i{\mb R}$.

Given any connection $A$ on $P$, there exist trivializations of $P$ which put $A$ in balanced temporal gauge, and any two such trivializations differ by a constant gauge transformation and a gauge transformation of the form $(t, \theta)\mapsto \theta^k$(the latter transformation changes $\lambda$ to $\lambda+ik$).

\subsubsection{Cylinders with small energy and nearly critical residue}

As before, for any positive real number $N$ we denote $C_N=(-N-2, N+2)\times S^1$ and $Z_N=[-N, N]\times S^1$ with the standard product metric and the conformal structure. We denote by $(t, \theta)$ the coordinates of points in $C_N$. Let $\nu=fdt\wedge d\theta$ be a volume form on $C_N$, and let $\eta$ be some positive real number. We say that $\nu$ is {\bf exponentially $\eta$-bounded} if $|f|+|\nabla f|<\eta e^{|t|-N}$.

\begin{thm}\label{thm24} Fix some $c\in i{\mb R}$. There exist positive real numbers $\epsilon_2, \eta_0, \sigma_2, K_2$ with the following property. Let $N>0$ be a real number and let $(\alpha, \phi): C_N\to X$ be a twisted holomorphic pair. Suppose that $\alpha$ is in balanced temporal gauge and that there exists some critical residue $\lambda\in i{\mb R}$ such that
$$
\left\| \alpha-\lambda d\theta \right\|_{L^\infty}<\epsilon_2, \ and\ \left\| d_\alpha\phi \right\|_{L^\infty}<\epsilon_2.
$$
Then there exist maps $\psi: [-N, N]\to X^\lambda$ and $\phi_0: Z_N\to TX$ with $\phi_0(t, \theta)\in T_{e^{-\lambda \theta}\psi(t)} X$, such that
$$
\phi(t, \theta)=e^{-\lambda \theta}\exp_{\psi(t)} \left(e^{\lambda \theta} \phi_0(t, \theta) \right)\  and\  \int e^{\lambda\theta} \phi_0(t, \theta)d\theta=0
$$
where $\exp_x: T_x X\to X$ denotes the exponential map on $X$. Furthermore,

(1) If $\nu$ is an exponentially $\eta$-bounded volume form on $C_N$ for some $\eta<\eta_0$ and the vortex equation $\iota_\nu d\alpha+\mu(\phi)=c$ is satisfied, then for any $(t, \theta)\in Z_N$
$$
\left| \phi_0(t, \theta) \right|<K_2 e^{\sigma_2(|t|-N)}.
$$

(2) If $d\alpha=0$, then for any $(t, \theta)\in Z_N$
$$
\left| \phi_0(t, \theta) \right|<K_2 e^{\sigma_2(|t|-N)}.
$$
\end{thm}

\begin{rem} We explain the expression appeared in this statement. Suppose $\lambda=i p/q$ where $p,q$ are relatively prime integers with $q>0$. Then $e^{\lambda \theta}$ is well-defined up to $e^{2k\pi i/ q}$. Since $\psi(t)\in X^\lambda$, which is the fixed point set of $\left\{ e^{2k\pi i/q} \right\}$, and the exponential map is equivariant, we have $e^{-2\pi i/q} \exp_{\psi(t)}\left( e^{2\pi i/q} v \right)=\exp_{\psi(t)}(v)$. The integral $\int e^{\lambda \theta} \phi_0(t, \theta) d\theta$ is understood as integrating over $[0, 2\pi q]$.
\end{rem}

\begin{thm}\label{thm25}
The same $K_2, \sigma_2$ satisfy following property. Let $N>0$ be a real number, $\lambda\in \Lambda_{cr}$ and $\left( \alpha, \phi \right): C_N\to X$ be a twisted holomorphic pair satisfying the hypothesis of last theorem for $\epsilon_2$. Assume furthermore that $\alpha$ is in balanced temporal gauge, so $\alpha= ad\theta$ and $a|_{\{0\}\times S^1}=\lambda_0\in i{\mb R}$ and $|\lambda_0-\lambda|<\epsilon$.

(1) Suppose $\nu$ is an exponentially $\eta$-bounded volume form on $C_N$ and that the vortex equation $\iota_\nu d\alpha+\mu(\phi)=c$ is satisfied. Then for any $t\in [-N, N]$,
$$
\left| \psi'(t)+i(\lambda_0-\lambda) \nabla h(\psi(t)) \right|< K_2 e^{\sigma_2(|t|-N)}.
$$

(2) Suppose that $d\alpha=0$. Then for any $t\in [-N, N]$,
$$
\left| \psi'(t)+i(\lambda_0-\lambda) \nabla h(\psi(t)) \right|< K_2 e^{\sigma_2(|t|-N)}.
$$
\end{thm}

\subsubsection{Chains of gradient segments}

A {\bf monotone chain of gradient segments} in $X$ is a compact subset ${\mc T}\subset X$ which can be written as ${\mc T}=\rho(T)$, where $T\subset {\mb R}$ is a compact interval(possibly consisting of a unique point) and $\rho: T\to X$ is a continuous map such that
\begin{itemize}
\item $\rho^{-1}(F)$ is a finite set;

\item $\rho$ is smooth at $T\setminus \rho^{-1}(F)$;

\item for any $t\in T\setminus \rho^{-1}(F)$, we have $\rho'(t)=-s\nabla h(\rho(t))$ for some positive real number $s$ depending on $t$.
\end{itemize}

Define the {\bf beginning} of ${\mc T}$ to be $b({\mc T}):=\rho(\inf T)$ and the {\bf end} of ${\mc T}$ to be $e({\mc T}):=\rho(\sup T)$. We say that ${\mc T}$ is {\bf degenerate} if $b({\mc T})=e({\mc T})$. Let ${\mc T}(X)$ be the set of chains of gradient segments, with the topology induced by the Hausdorff distance between subsets. The space ${\mc T}(X)$ is compact and carries a continuous action of $S^1$ induce by the action on $X$. Define the set of {\bf infinite monotone chains of gradient segments} ${\mc T}^\infty(X)\subset {\mc T}(X)$ as the set consisting of those ${\mc T}$ such that both the beginning $b({\mc T})$ and the end $e({\mc T})$ belong to the fixed point set $F$.

\begin{thm}\label{thm26}
Let $\left\{ \psi_u: S_u\to X \right\}_{u\in {\mb N}}$ be a sequence of maps, where each $S_u$ is a closed interval in ${\mb R}$ and $\psi_u$ is a smooth map. Suppose that the length of $S_u$ tends to infinity, and that there is a sequence of positive real numbers $l_u\to 0$ and positive constants $K, \sigma$ with the property that for any $u$ and $t\in S_u$, we have
$$
\left| \psi_u'(t)+ l_u\nabla h(\psi_u(t)) \right|\leq K e^{-\sigma d(t, \partial S_u)}.$$
For any $\tau>0$ let $S_u^\tau=\left\{ t\in S_u| d(t, \partial S_u)\geq \tau \right\}$. Then there exists a monotone chain of gradient segments ${\mc T}\in {\mc T}(X)$ and an infinite sequence $\{u_i\}\subset {\mb N}$ such that $$
\lim_{\tau\to\infty} \left(\limsup_{i\to \infty} d_{Hauss} (\psi_{u_i}(S_{u_i}^\tau), {\mc T})\right)=0,$$ $$
\lim_{\tau\to \infty} \left(\limsup_{i\to\infty} d(\psi_{u_i} (\inf S_{u_i}^\tau), b(\mc T))\right)=\lim_{\tau\to\infty}\left(\limsup_{i\to\infty} d(\psi_{u_i} (\sup S_{u_i}^\tau), e(\mc T))\right)=0.$$
\end{thm}

\section{Twisted holomorphic pairs from bordered curves with Lagrangian boundary condition}

\begin{defn} A bordered Riemann surface is a compact 2-manifold with nonempty boundary equipped with a complex structure.
\end{defn}

\begin{rem} A bordered Riemann surface $\Sigma$ is canonically oriented by the complex structure. Its boundary $\partial\Sigma$ is the disjoint union of its (finitely many) connected componets $B_i$, with induced orientation. 
\end{rem}

A bordered Riemann surface is topologically a sphere with $g\geq 0$ handles and with $h>0$ discs removed. Such a bordered Riemann surface is said to be of type $(g, h)$.

\begin{defn} Let $h$ be a positive integer, $g, n$ be nonnegative integers, and $\overrightarrow{m}=(m^1, \ldots, m^h)$ be an $h$-tuple of nonnegative integers. A {\bf marked bordered Riemann surface of type $(g, h)$ with $(n, \overrightarrow{m})$ marked points} is an $(h+3)$-tuple
$$\left( \Sigma, {\bf B}; {\bf z}; {\bf x}^1, \ldots, {\bf x}^h \right)$$
where
\begin{itemize}
\item $\Sigma$ is a bordered Riemann surface with type $(g, h)$;

\item ${\bf B}=(B^1, \ldots, B^h)$, where the $B^i$'s are connected components of $\partial \Sigma$ oriented by the complex structure;

\item ${\bf z}=(z_1, \ldots, z_n)$ is an $n$-tuple of distinct points in $\Sigma^0$, where are called {\bf interior marked points};

\item ${\bf x^i}=(x^i_1, \ldots, x^i_{m^i})$ is an $m^i$-tuple of distinct points on the circle $B_i$. All the $x^i_j$'s are call {\bf boundary marked points}.
\end{itemize}
\end{defn}

\begin{rem}
The moduli space which we will compactify is the one of isomorphism classes of twisted holomorphic maps from bordered Riemann surfaces of given type $(g, h)$ with $(n, \ora{m})$ marked points. But in proving the compactness theorem, in most cases, we won't distinguish boundary marked points on different boundary components. So we simplify the notations: $\Sigma$ will be a smooth Riemann surface with boundary $\partial \Sigma$. ${\bf z}=\{z_1, \ldots, z_k\}\subset \Sigma\setminus \partial \Sigma$ is the set of {\bf interior} marked points and ${\bf x}=\{x_1, \ldots, x_l\}\subset \partial \Sigma$ is the set of {\bf boundary} marked points. 
\end{rem}

Let's fix a conformal metric on $\Sigma$ with volume form $\nu$. Recall that $L\subset X$ is an $S^1$-invariant Lagrangian submanifold.

\begin{defn}
Let $P\to \Sigma\setminus {\bf x}\cup {\bf z}$ be a principal $S^1$-bundle and $Y=P\times_{S^1} X$. Let $\varphi$ be a section of $Y$. We say that $\varphi$ maps the boundary $\partial \Sigma$ into $L$ if $\varphi\left(\partial\Sigma\setminus {\bf x} \right)\subset P\times_{S^1} L$. Let $A$ be a connection on $P$. We say that the tuple $(P, A, \varphi)$ is holomorphic if $\ov\partial_A \varphi=0$ wherever $A$ and $\varphi$ is defined. If both conditions holds, we call such a tuple $(P, A, \varphi)$ a holomorphic tuple with boudnary mapped into $L$.

We say that two tuples $(P, A, \varphi)$ and $(P', A', \varphi')$ over $(\Sigma, {\bf x}\cup {\bf z})$ are isomorphic, if there is a bundle isomorphism $g: P\to P'$ such that $g^*A'=A$ and $g^*\varphi'=\varphi$.
\end{defn}

A {\bf meromorphic} connection $A$ on $P$ is a smooth connection with $F_A$ bounded with respect to the chosen metric. Because of the existence of nontrivial limit holonomy at interior marked points, one can't extend the connection over those points. However, up to gauge transformation, a holomorphic tuple $(P, A, \varphi)$ can be extended smoothly over boundary marked points, if they satisfy the vortex equation or flat connection equation. More precisely, we will prove the following theorem in this section.

\begin{thm}\label{removal} Let $(\Sigma, {\bf x}\cup {\bf z})$ be a marked bordered Riemann surface. Let $\nu$ be a smooth volume form on $\Sigma$ and $c\in i{\mb R}$. Let $(P, A, \varphi)$ a holomorphic tuple with boundary mapped into $L$. If in addition we have either $\iota_\nu F_A+\mu(\varphi)=c$ or $F_A=0$ on $\Sigma\setminus {\bf x}\cup {\bf z}$, where $c\in i{\mb R}$, then there exists a holomorphic tuple $(P', A', \varphi')$ over $\Sigma\setminus {\bf z}$ with boundary mapped in $L$ and a smooth isomorphism $g: P'|_{\Sigma\setminus {\bf x}\cup {\bf z}}\to P$ such that $g^*(A,\varphi)=(A', \varphi')|_{\Sigma\setminus {\bf x}\cup {\bf z}}$. Moreover, any two such extensions $(P', A', \varphi')$ and $(P'', A'', \varphi'')$ are isomorphic over $\Sigma\setminus {\bf z}$.
\end{thm}

The rest of the section is devoted to proving this theorem. In Subsection \ref{subsection31}, we show that we can extend the bundle and the connection; in Subsection \ref{subsection32}, we apply a removal of singularity theorem to extend the section; in Subsection \ref{subsection33} we prove a regularity theorem to show that the extension is smooth and in Subsection \ref{subsection34} the uniqueness of the extension. 

After the theorem is proved, we will assume every twisted holomorphic map on a bordered Riemann surface is defined over all boundary marked points.

\subsection{Extension of the bundle and the connection}\label{subsection31}

We first consider the model case. Let ${\mb D}\subset {\mb C}$ be the closed unit disc and ${\mb D}^*$ be the punctured unit disc. Let ${\mb D}_+= {\mb D}\cap\ov{\mb H}$ be the closed half disc and ${\mb D}_+^*:={\mb D}^*\cap \ov{\mb H}$ be the punctured half disc. All of these spaces are endowed with the standard conformal structure. Then ${\mb D}_+^*$ is conformal to $[0, \infty)\times [0,\pi]$.

Let $P\to {\mb D}_+^*$ be the trivial $S^1$-bundle and $A$ be a meromorphic connection on $P$. Then $P$ extends uniquely to the trivial bundle over $B$, which is, in this subsection, denoted by $\widetilde{P}$.

Under the cylindrical coordinate $(t, \theta)\in {\mb D}_+^*$, write $d_A=d+\alpha=d+ \alpha_t dt+\alpha_\theta d\theta$. Then there exists a (smooth) gauge transformation $g_1: {\mb D}_+^*\to S^1$ such that $d_{g_1^*A}=d+ \alpha_t' dt$. Then define a continous 1-form on $\widetilde{\alpha}$ on ${\mb D}^*\simeq [0, \infty)\times S^1$: $$
\widetilde{\alpha}=\widetilde{\alpha}_t(t, \theta)dt=\left\{ \begin{array}{crr}
        & \alpha_t'(t, \theta) dt, &\ {\rm if}\ \theta\in [0, \pi],\\
        & \alpha_t'(t, 2\pi-\theta) dt, &\ {\rm if}\ \theta\in [\pi, 2\pi].
                          \end{array} \right. $$

Then $\widetilde{\alpha}$ gives a continuous connection on the trivial bundle $\widetilde{P}={\mb D}^*\times S^1$. Note that the residue of this connection at the origin is zero, hence by the proof of Lemma \ref{lemma21} there is a continuous gauge transformation $g_2: {\mb D}^*\to S^1$, which is of winding number 0, and whose restriction to each half disk is smooth, such that $\widetilde{\alpha}- d\log g_2$ extends to a continuous 1-form on ${\mb D}$. Set $g=g_2\cdot g_1$, then the covariant derivative of $g^*A$ is written as $d_{g^*A}= d+ \alpha_0$ with $\alpha_0$ a continuous 1-form on ${\mb D}_+$ which is smooth on ${\mb D}_+^*$.

Now let $P\to \Sigma\setminus {\bf x}\cup {\bf z}$ with a meromorphic connection $A$. Then $P$ extends uniquely to a bundle $\widetilde{P}\to \Sigma \setminus {\bf z}$. We can do as above to find gauge transformations near each boundary marked point and glue them together to get a global gauge transformation $g: \Sigma\setminus {\bf x}\cup {\bf z}\to S^1$, then extend the connection $g^* A$ to a continuous one on $\widetilde{P}$.

\subsection{Extension of the section}\label{subsection32}

Now let $(P, A, \varphi)$ be a holomorphic tuple with $P$ and $A$ already extended over the boundary marked points as above, and $\varphi$ maps $\partial \Sigma$ into $L$.

Near each boundary marked point $x$, we can identify $(P, A, \varphi)$ with a twisted holomorphic pair $(\alpha, \phi)$ on a punctured half disk ${\mb D}_+^*$. The 1-form $\alpha$ induces a {\it continuous} almost complex structure $I_\alpha$ on ${\mb D}_+\times X$. Let $\widetilde{\phi}: {\mb D}_+^*\to {\mb D}_+\times X$ be given by $\widetilde{\phi}(x)=(x, \phi(x))$. Then $\widetilde{\phi}$ is $I_\alpha$-holomorphic and $\widetilde{\phi}({\mb D}^*\cap {\mb R})\subset ({\mb D}\cap {\mb R})\times L$. We want to apply the theorem of Ivashkovich-Shevchishin \cite{Ivashkovich_Shevchishin_2002} on the removal of boundary singularity for continuous almost complex structures with Lagrangian boundary condition.

To apply the theorem, we need the following definition:

\begin{defn} Let $(Y, h, J)$ be a Riemannian manifold equipped with a continuous almost complex structure. $W\subset Y$ is a totally real submanifold with respect to $J$. $A\subset W$ is a subset. We say that $W$ is {\bf $h$-uniformly totally real along $A$ with a lower angle $\alpha>0$} if for any $w\in A$ and $\xi\in T_w W$, $\xi\neq 0$, the angle between $J\xi$ and $T_w W$ is no less than $\alpha$.
\end{defn}

\begin{thm}\label{thm35} Let $(Y, h, J)$ be a Riemannian manifold equipped with a continuous almost complex structure $J$. $W\subset Y$ is a totally real submanifold with respect to $J$ with a lower angle $\alpha>0$. $A\subset W$ is a subset. Let $u: {\mb D}_+^*\to Y$ be a $J$-holomorphic map with $u({\mb D}^*_+\cap {\mb R})\subset A$. Suppose that $J$ is $h$-uniformly continuous on $u({\mb D}_+^*)$ and the closure of $u({\mb D}_+^*)$ is $h$-complete; $W$ is $h$-uniformly totally real along $A$; the energy of $u$ is finite. Then $u$ extends to the origin $0\in {\mb D}_+$ as a $W^{1, p}$ map for some $p>2$.
\end{thm}

Take $g$ be a Riemannian metric on $X$ such that $TL\bot I(TL)$ and $h$ to be the product metric on ${\mb D}_+\times X$. In order to apply the above theorem for $Y={\mb D}_+\times X$, $W=({\mb D}_+\cap {\mb R})\times L$, we only need to check that $L$ is $h$-uniformly totally real along $\phi({\mb D}_+^*\cap {\mb R})$, with respect to $I_\alpha$. Then for $(x, y)\in  ({\mb D}_+^*\cap {\mb R})\times L$ and $(v, Y)\in T_{(x, y)} ({\mb D}_+^*\times X)$, $$
I_\alpha (v, Y)=\left( jv, IY+i(\alpha(jv){\mc X}-\alpha(v) I{\mc X}) \right).$$

Suppose $|\alpha|<M$ on ${\mb D}_+$ since it is continuous. Let $(s, t)$ be the Euclidean coordinate on ${\mb D}$ with $s$ the real coordinate. $\alpha_s=\alpha(\partial_s)$, $\alpha_t=\alpha(\partial_t)$. Take two vectors $\xi_i=(a_i \partial_s, Y_i)$, $Y_i\in T_w L$, $i=1, 2$. If $a_1=0$, then $I \xi_1\bot \xi_2$. If $a_1\neq 0$, then
\begin{multline*}
\left| \cos \angle_h (I_\alpha \xi_1, \xi_2) \right|={\left|\langle I_\alpha \xi_1, \xi_2\rangle\right| \over \|I_\alpha \xi_1\|\cdot\|\xi_2\| }
={\left|\left\langle (a_1\partial_t, IY_1+ ia_1(\alpha_t {\mc X}-\alpha_s I{\mc X})) ,( a_2\partial_s, Y_2)\right\rangle \right|\over \left(\|a_1 \partial_t\|^2+\|IY_1+a_1 i(\alpha_t {\mc X}-\alpha_s I{\mc X})\|^2\right)^{1/2}\cdot\|\xi_2\|}\\
={|a_1 \alpha_t |\cdot |\langle {\mc X}, Y_2\rangle |\over (|a_1|^2+\|a_1\alpha_t {\mc X}\|^2+\|Y_1-ia_1\alpha_s {\mc X}\|^2)^{1/2}\cdot \|\xi_2\|}
\leq {|a_1 \alpha_t |\cdot \|{\mc X}\|\cdot \|Y_2\|\over (|a_1|^2+\|a_1\alpha_t {\mc X}\|^2)^{1/2}\cdot \|\xi_2\|}\\
\leq {|\alpha_t| \|{\mc X}\|\over (1+|\alpha_t|^2 \|{\mc X}\|^2)^{1/2} }\leq \left( {M^2 \|{\mc X}\|_{L^\infty}^2\over 1+ M^2\|{\mc X}\|_{L^\infty}^2}\right)^{1/2}<1.
\end{multline*}

Thus we have proved $L$ is $h$-uniformly totally real along $\phi({\mb D}^*_+)\cap L$. Then we can apply Theorem \ref{thm35} to this case and $\phi$ extends to a map of class $W^{1, p}$. Hence the tuple $(P, A, \Phi)$ is extended to a $W^{1, p}$ section on $\Sigma\setminus {\bf z}$.

\subsection{Regularity for twisted holomorphic pairs}\label{subsection33}

Now let $\Sigma$ be a bordered Riemann surface with a smooth volume form $\nu$, ${\bf z}\subset \Sigma$ is a subset of interior marked points. Let $P$ be a principal $S^1$-bundle over $\Sigma\setminus {\bf z}$ and $Y=P\times_{S^1} X$.
\begin{thm} Suppose $A$ is a connection on $P$ of class $W^{1, p}$ and $\varphi\in W^{1, p}$ is a $A$-holomorphic section of $Y$ and $\varphi$ maps $\partial \Sigma$ into $L$. If $(A, \varphi)$ satisfies either the vortex equation $\iota_\nu F_A+\mu(\varphi)=c$ or $F_A=0$, both in the $L^p$ sense, then there exists a gauge transformation $g$ of class $W^{2, p}$ such that $g^*(A, \varphi)$ is smooth.
\end{thm}

\begin{proof}
This is a simple generalization of the theorem in \cite{Cieliebak_Gaio_Mundet_Salamon_2002} to the case of Riemann surface with boundary. The method is essentially the same, i.e. using the local slice theorem(see \cite{Wehrheim_Uhlenbeck}) and elliptic bootstrapping.
\end{proof}



Thus, let $(P, A, \varphi)$ be a holomorphic tuple over $(\Sigma, {\bf x}\cup {\bf z})$ which satisfies the boundary condition and either the vortex equation or $F_A=0$. {\it A priori } $A$ is a continuous connection on $P$ and $\varphi$ is of class $W^{1, p}$. Then the vortex equation or $F_A=0$ implies that $F_A$ is of class $W^{1, p}$ on $\Sigma\setminus {\bf z}$. Then by Poincar\'e lemma, we can find a $W^{2, p}$ gauge transformation $h$ such that $h^*A$ is actually of class $W^{1, p}$. Then we can apply the above theorem to get regularity, i.e. there exists a $W^{2, p}$ gauge tranformation such that $g^* (A, \varphi)$ is smooth. But both $(A, \varphi)$ and $g^* (A, \varphi)$ are smooth on $\Sigma \setminus {\bf x}\cup {\bf z}$, which implies that $g$ is smooth on $\Sigma\setminus {\bf x}\cup {\bf z}$. Hence, every tuple on $\Sigma \setminus {\bf z}\cup {\bf x}$ is gauge equivalent to a smooth tuple on $\Sigma \setminus {\bf z}$ by a smooth gauge transformation.

\subsection{Uniqueness}\label{subsection34}

Suppose we have two tuples $(P', A', \varphi')$, $(P'', A'', \varphi'')$ over $\Sigma\setminus {\bf z}$ which are both the extension of the original $(P, A, \varphi)$. Then there exists a smooth isomorphism $g: P'|_{\Sigma\setminus {\bf x}\cup {\bf z}}\to P''|_{\Sigma\setminus {\bf x}\cup {\bf z}}$ such that on $\Sigma\setminus {\bf x}\cup {\bf z}$, $g^*(A'', \varphi'')=(A', \varphi')$. So for each $x_j\in {\bf x}$, take a local trivialization of $P'$, $P''$ near $x_j$, with respect to which $d_{A'}=d+\alpha'$, $d_{A''}=d+\alpha''$, $g$ is  an $S^1$-valued function. One has, $$
\alpha'-d\log g=\alpha''.$$
But $d(\alpha'-\alpha'')=0$, in the weak sense (hence in the usual sense), since $F_{A'}=F_{A''}$. So there exists a smooth function $h: \Sigma\setminus {\bf z}\to i{\mb R}$ such that $\alpha'-\alpha''=dh$. So $\log g=h+C$ and $g$ extends smoothly over $x_j$, which is still denoted by $g$. So $g^*(P', A')=(P'', A'')$. And $g^*\varphi'=\varphi''$ by the uniqueness of the removal of singularity Theorem \ref{thm35}.

\section{Analogues of Theorem \ref{thm22}, \ref{thm24} and \ref{thm25} for half-cylinders}

One type of boundary degenerations of bordered curves is the shrinking of a boundary circle to a point. Conformally, this is equivalent to the process that the lengths of boundary cylinders tends to infinity. In the case of twisted holomorphic maps with boundary in $L$, such degenerations are almost the same as that treated in \cite{Mundet_Tian_2009}.

For any positive real number $N$, set $C_N^+=[0, N+2)\times S^1$, $Z_N^+=[0, N]\times S^1$. Let $P$ be a principal bundle over $C_N^+$, and let $A$ be a connection on $P$. We say that a given trivialization puts $A$ in {\bf temporal gauge} if with respect to this trivialization, the covariant derivative $d_A$ takes the form $d+ ad\theta$, where $a: C_N^+\to i{\mb R}$ is a function. We say that the trivialization is in {\bf balanced temporal gauge} if the restriction of $a$ to the boundary $\{0\}\times S^1$ is equal to some constant $\lambda\in i{\mb R}$.

We give $C_N^+$ and $Z_N^+$ the standard product metrics and conformal structures. We denote by $(t, \theta)$ the coordinate of points in $C_N^+$. Let $\nu=fdt\wedge d\theta$ be a volume form on $C_N^+$, and let $\eta$ be a positive real number. We say that $\nu$ is {\bf exponentially $\eta$-bounded} if $|f|+|\nabla f|<\eta e^{t-N}$.

We have the following analogue for Theorem \ref{thm22}, \ref{thm24} and \ref{thm25}.

\begin{thm}\label{thm40}
For any noncritical $\lambda\in i{\mb R}\setminus \Lambda_{cr}$, there exist positive real numbers $K_3, \sigma_3, \epsilon_3$ depending continuously on $\lambda$ with the following properties. Let $N>0$ be any real number and let $(\alpha, \phi): C_N^+\to X$ be a twisted holomorphic pair satisfying $\left\| \alpha-\lambda d\theta \right\|_{L^\infty (C_N^+)} <\epsilon_3$, $\left\| d\alpha \right\|_{L^\infty(C_N^+)}<\epsilon_3$, $\left\| d_\alpha \phi \right\|_{L^\infty(C_N)}<\epsilon_3$. Then for any $(t, \theta)\in Z_N^+$, we have
$$\left| d_\alpha \phi(t, \theta)\right| \leq K_3 e^{\sigma_3(t-N)} \left\| d_\alpha\phi \right\|_{L^\infty(Z_N^+)}.$$
In particular this implies that
$${\rm diam}_{S^1} \left( \phi(Z_N^+)\right) < K_3 \left\| d_\alpha \phi \right\|_{L^\infty (Z_N^+)}.$$
\end{thm}

\begin{thm}\label{thm4.1}
Fix some number $c\in i{\mb R}$. There exists positive real numbers $\epsilon_4, \eta_0, \sigma_4, K_4$ with the following property. Let $N>0$ be a real number and let $(\alpha, \phi): C_N^+\to X$ be a twisted holomorphic pair. Suppose that $\alpha$ is in balanced temporal gauge and that there exists some critical value $\lambda\in \Lambda_{cr}$ such that
$$\left\| \alpha-\lambda d\theta \right\|_{L^\infty}<\epsilon_4,\ \left\|d_\alpha\phi \right\|_{L^\infty}<\epsilon_4.$$
Then there exist maps $\psi: [0, N]\to X^{\lambda}$ and $\phi_0: Z_N^+\to TX$ satisfying, for any $(t, \theta)\in Z_N^+$, $\phi_0(t, \theta)\in T_{e^{-\lambda\theta}\psi(t)}X$,
$$\phi(t, \theta)=e^{-\lambda\theta}\exp_{\psi(t)}\left(e^{\lambda\theta} \phi_0(t, \theta)\right),\ and\ \int e^{\lambda\theta} \phi_0(t,\theta)=0. $$
where $\exp$ is the exponential map of the metric on $X$. Furthermore,

\begin{enumerate}

\item If $\nu$ is an exponentially $\eta$-bounded volume form on $C_N^+$ for some $\eta<\eta_0$, and the vortex equation $\iota_\nu d\alpha+\mu(\phi)=c$ is satisfied, then for any $(t, \theta)\in Z_N^+$
$$\left| \phi_0(t, \theta) \right|<K_4 e^{\sigma_4(t-N)}.$$

\item If $d\alpha=0$, then for any $(t, \theta)\in Z_N^+$,
$$\left| \phi_0(t, \theta) \right|<K_4 e^{\sigma_4(t-N)}.$$
\end{enumerate}
\end{thm}

\begin{thm}\label{thm4.2}
The same $K_4$, $\sigma_4$ satisfy the following property. Let $N>0$ be a real number and let $(\alpha, \phi): C_N^+\to X$ be a twisted holomorphic pair satisfying the hypothesis of the above theorem for the $\epsilon_4$. Assume furthermore that $\alpha$ is in balanced temporal gauge, so that $\alpha=ad\theta$ and $a|_{0\times S^1}$ is equal to some $\lambda_0\in i{\mb R}$ with $|\lambda_0-\lambda|<\epsilon_4$.

\begin{enumerate}

\item Suppose $\nu$ is an exponentially $\eta$-bounded volume form on $C_N^+$ and that the vortex equation $\iota_\nu d\alpha+\mu(\phi)=c$ is satisfied. Then for any $t\in [0, N]$,
$$ \left| \psi'(t)+i(\lambda_0-\lambda) \nabla h(\psi(t)) \right|<K_4 e^{\sigma_4(t-N)}.$$

\item Suppose that $d\alpha=0$. Then for any $t\in [0, N]$,
$$ \left| \psi'(t)+i(\lambda_0-\lambda)\nabla h(\psi(t)) \right|< K_4 e^{\sigma_4(t-N)}.$$
\end{enumerate}
\end{thm}

These theorems can be proved the same way as proving \cite[Theorem 4.1, Theorem 4.3, Theorem 4.4]{Mundet_Tian_2009}. We omit the proof.

\section{Reflection of twisted pairs from strips with Lagrangian boundary condition}

Another typical degeneration of bordered curves is the shrinking of an arc to a point, which is conformally equivalent to the forming of long strips. In this section we treat this type of degenerations in twisted holomorphic maps with boundary in $L$. We will see that in the limit, the two sides of the boundary node actually connect.

Indeed, if there is a global anti-holomorphic involution on $(X, I)$ with fixed point set equal to $L$, then we can use the reflectio principle to reduce the problem to the cylinder case. Even though in the general case there is no global involution, the image of the strip lies in a small neighborhood of the Lagrangian $L$ and there is still a local involution which allows us to use the reflection principle. Moreover, the holonomy along the central circle of the cylinder will be always trivial because the two half circles cancel with each other by the reflection. So by Theorem \ref{thm24} and \ref{thm25}, the gradient flow lines will be degenerate and the two limit orbits of the long cylinders will coincide in the limit. 

\subsection{A tubular neighborhood of the Lagrangian submanifold and local charts}

Let $N_L$ be the normal bundle of $L$, regarded as the subbundle of normal vectors of $TX$. Let $N_{L, \delta}\subset N_L$ be the subset consisting of all normal vectors with length less than $\delta$. It has a natural reflection map $\sigma: N_{L, \delta}\to N_{L, \delta}$ given by $\sigma(x, v)=(x, -v)$. By basic differential geometry we know that for $\delta$ small enough, the exponential map $\exp: N_{L, \delta}\to X$ given by $(x, v)\mapsto \exp_x v$ is a diffeomorphism onto a tubular neighborhood of $L$ in $X$, and ${\rm dist}_g(\exp_x v, L)=|v|$. In this section we fix a small $\delta_0$, and identify $N_{L, \delta_0}$ with a tubular neighborhood of $L$ by the exponential map. Since the $S^1$ action preserves the Riemannian metric, $N_L\to N$ is an equivariant bundle, and the infinitesimal action ${\mc X}$ is invariant under the reflection $\sigma$.

The Riemannian metric gives the Levi-Civita connection, which restricts to a connection on the bundle $\pi: N_L\to L$. Then this connection induces the ``horizontal-vertical'' splitting $TN_L\simeq \pi^* TL\oplus \pi^* N_L$. The almost complex structure $I$ gives a bundle isomorphism $I: TL\to N_L$. By the splitting this isomorphism gives an almost complex structure on $N_{L, \delta_0}$, which is denoted by $\widetilde{I_0}$. Then $\widetilde{I_0}$ may be different from the original $I$. But they coincide on $L$(or the 0-section). Moreover, there exists a $K_0>0$, depending only on $X$, $I$ and $g$, such that for any $(x, v)\in N_{L, \delta_0}$, 
\begin{align}\label{k}
\left\| I(x, v)-\widetilde{I_0}(x, v) \right\|_g \leq K_0 |v|=K_0 {\rm dist}_g(\exp_x v, L).
\end{align}

For any $y\in L$ and a small $\delta$ with $0<\delta<\delta_0$, denoted by $B_L(y, \delta)$ the $\delta$-ball in $L$ centered at $y$. Take a normal coordinate on $B_L(y, \delta)$ to be $(x^1, \ldots, x^n)$ and $\partial_i$ be the vector field $\partial/\partial x^i$. Define
$${\mb C}^n \simeq  {\mb R}^n\oplus i{\mb R}^n\ni \left((x^1, \ldots, x^n), (v^1, \ldots, v^n) \right)\mapsto \left(  (x^1, \ldots, x^n), \sum_i  v^i (I\partial_i)\right)\in N_L. $$
This map is a diffeomorphism on a small neighborhood of the origin. Then for $\delta$ small enouth, we define $f_{y,\delta}$ to be the inverse of the above map restricted on $\pi^{-1}(B_L(y, \delta))\cap N_{L, \delta}$. Let ${\mb C}^n$ be endowed with the standard complex structure $I_0$ and the standard metric $g_0$. One has that
$$f_{y, \delta}^{-1}\left({\mb R}^n \right)=B_L(y, \delta),\ \left. \left( f_{y, \delta}^{-1} \right)^* \widetilde{I}_0\right|_{{\mb R}^n}=I_0|_{{\mb R}^n},\ f_{y, \delta}\circ \sigma=\sigma_0\circ f_{y, \delta},$$ where $\sigma_0$ is the conjugation of ${\mb C}^n$. Since $f_{y, \delta}$ is isometric at $y$, one has that for any $\epsilon$ with $0<\epsilon<1/2$, there exists $\delta=\delta(\epsilon, y)>0$ such that
$$
\left\| g_0-\left( f_{y, \delta}^{-1} \right)^* g \right\|_{L^\infty}<\epsilon,\ {\rm and}\ 1/2\leq \left\|df \right\|\leq 2, \label{x:2}
$$
with respect to the metrics $g$ on $N_{L, \delta}$. By taking $\delta(\epsilon, y)$ small one can also assume that the image of $f$ is contained in the unit ball of ${\mb C}^n$.

\begin{rem}
Since $\left.\left(f_{y, \delta}^{-1} \right)^*\left( \widetilde{I}_0 \right)\right|_{{\mb R}^n}=I_0|_{{\mb R}^n}$, by smoothness, for any $y\in L$ and $\delta<\delta_0$, there exists $K_{y, \delta}>0$ such that on the image of $f_{y, \delta}$
$$
\left\| (f_{y, \delta}^{-1})^* \widetilde{I}_0-I_0 \right\|\leq K_{y, \delta} |{\bf v}|, \ \left\|D \left( (f_{y, \delta}^{-1})^* g \right) \right\|\leq K_{y, \delta}\label{x: K}
$$
with respect to the standard metric $g_0$ on ${\mb C}^n$. Here ${\bf v}$ denotes the vector $(v^1, \ldots, v^n)$.
\end{rem}

\subsection{Reflection}

Let $S_N^+=(-N-2, N+2)\times [0, \pi]$ and $S_N^-=(-N-2, N+2)\times [\pi, 2\pi]$ be the strips. Let $C_N=S_N^+\cup S_N^-$ be the cylinder glued from the two parts. Denote by $\tau$ the natural reflection of $C_N$, given by $\tau(t, \theta)=(t, 2\pi-\theta)$. We say that a map $f: C_N\to N_{L, \delta}$(resp. a differential form $\alpha$ on $C_N$) is {\bf reflection-symmetric} if $f\circ \tau=\sigma\circ f$(resp. $\tau^*\alpha=\alpha$).

In this section, we will do the following operation on each twisted holomorphic pair $(\alpha, \phi)$ defined on $S_N^+$ with $\left\| d_\alpha\phi \right\|_{L^\infty}$ small enough. First, take a gauge transformation $g_1$ on $S_N^+$ such that $g_1^*\alpha$ only has $dt$ component. Second, ``reflect'' the pair to get a twisted pair $(\tilde\alpha, \tilde\phi)$ defined on the cylinder $C_N$. Third, use another gauge transformation $g_2$ on $C_N$ such that $g_2^*\alpha$ is in balanced temporal gauge and such that $g_2^*\phi$ is reflection-symmetric. The resulting pair will NOT be holomorphic with respect to the original almost complex structure, but holomorphic with respect to a family of almost complex structures on $N_{L, \delta_0}$ which we will construct.

\begin{enumerate}
\item {\bf The first gauge transformation.} Let $(\alpha, \phi): S_N^+\to X$ be a twisted holomorphic pair with $\phi(\partial S_N^+)\subset L$ and $\|d_\alpha\phi\|_{L^\infty}$ small enough. Then $\phi(S_N^+)\subset N_{L, \delta_0}$.

Suppose $\alpha=\alpha_t dt+\alpha_\theta d\theta$. Then there exists a gauge transformation $g_1: S_N^+\to S^1$ such that $g_1^*\alpha=a dt$. Set $(\alpha', \phi')=g_1^*(\alpha, \phi)$.

\item {\bf Reflect the pair $(\alpha', \phi')$.} We define a pair $(\tilde{\alpha}, \tilde{\phi})$ on $C_N$ as follows:
\begin{align*}
\widetilde{\alpha}=\left\{ \begin{array}{cc}
         a(t, \theta) dt,\ &\ {\rm if}\ \theta\in [0, \pi]; \\
         a(t, 2\pi-\theta) dt,\ &\ {\rm if}\ \theta\in [\pi, 2\pi],
                                   \end{array} \right. \ \ \ \widetilde{\phi}(t, \theta)=\left\{ \begin{array}{cc}
\phi(t, \theta), \ &\ {\rm if} \ \theta\in [0, \pi];\\
\sigma(\phi(t, 2\pi-\theta)),\ &\ {\rm if}\ \theta\in [\pi, 2\pi].
\end{array}\right.
\end{align*}
Write $\widetilde{\alpha}=\widetilde{\alpha}_t dt$, then $\widetilde{\alpha}_t$ is continuous.

\item {\bf Construct a family of almost complex structures on $N_{L, \delta_0}$.} Choose $\delta_1$ to be smaller than the injective radius of $g$. Take a cut-off function $\rho: [0, \infty)\to {\mb R}$ such that $\rho\geq 0$, $\rho(t)=1$ for $t\in [0, \delta_1/2]$ and $\rho(t)=0$ for $t\geq \delta_1$. Now we define a {\it continuous} family of almost complex structures on $N_{L, \delta_0}$, parametrized by points in $C_N$. For each $z\in S_N^+$, set $y=\pi(\phi'(z))\in L$.  $I(\phi'(z))$ is an almost complex structure on the vector space $T_{\phi'(z)} N_L\simeq T_y L\oplus N_{L, y}$, where the splitting is given by the connection on $N_L$. Trivialize $TL$, $N_L$ on $B_L(y, \delta_1)$ by parallel transport along radial directions with respect to the Levi-Civita connection, then $I(\phi'(z))$ defines a complex structure of the bundle $(TL\oplus N_L)|_{B_L(y, \delta_1)}$. By the splitting $TN_L\simeq \pi^* TL\oplus \pi^*N_L$, this defines an almost complex structure on $N_{L, \delta_0}\cap \pi^{-1}\left( B(y, \delta_1) \right)$, which is denoted by $I^*_z$. Recall that we have the other almost complex structure $\widetilde{I_0}$ on $N_{L, \delta_0}$. Take $\rho_z: L\to {\mb R}$ to be $\rho_z(y')=\rho(d(y, y'))$. Define 
$$I_z=\left( 1-\rho_z I^*_z \widetilde{I_0} \right) \widetilde{I_0} \left( 1-\rho_z I^*_z \widetilde{I_0} \right)^{-1}$$
(it easy to check that $\left( 1-\rho_z I^*_z \widetilde{I_0} \right) $ is invertible). $I_z$	coincides with $I^*_z$ in $\pi^{-1}(B(y, \delta_1/2))\cap N_{L, \delta_0}$ and coincides with $\widetilde{I_0}$ on $N_{L, \delta_0}-\pi^{-1}(B(y, \delta_1))\cap N_{L, \delta_0}$. Moreover, if $z\in \partial S_N^+$ then $\phi'(z)\in L$ and it is easy to see that $I_z=\widetilde{I_0}$. Then for $z\in S_N^-$, set $I_z=-d\sigma\circ I_{\tau(z)} \circ d\sigma$. And we see that $I_z$ is a continuous family of almost complex structures on $N_{L,\delta_0}$.

It is easy to check that the pair $\left( \widetilde{\alpha}, \widetilde{\phi} \right)$ is $I_z$-holomorphic, i.e. for any $z\in C_N$, we have
$$
\partial_t \widetilde{\phi}(z)- i\widetilde{\alpha}_t(z) {\mc X}\left( \widetilde{\phi}(z) \right)=I_z \partial_\theta \widetilde{\phi}(z).$$

Also, by (\ref{k}), one easily obtain that for any $z\in C_N$,
\begin{align}\label{kk}
\left\| I_z-\widetilde{I_0} \right\|_{L^\infty(g)}\leq K_0 {\rm dist}_g \left( \widetilde\phi(z), L \right).
\end{align}

\item {\bf Construct the second gauge transformation.} Set $b(t, \theta)=\int_0^t \widetilde{\alpha}_t(s, \theta) ds$, $g_2(t, \theta)=\exp (b(t, \theta))$. Since $\tilde{\alpha}$ is reflection-symmetric, so is $g_2$, i.e. $g_2(t, \theta)= g_2(t, 2\pi -\theta)$. So
$$ g_2^*\tilde\alpha=-\left(\int_0^t {\partial \widetilde{\alpha}_t\over \partial\theta} dt\right) d\theta $$
is in balanced temporal gauge, $\left. g_2^*\widetilde{\alpha}\right|_{\{0\}\times S^1}=0$ and $g_2^*\widetilde\phi$ is reflection-symmetric. Now set $\widetilde{I_z}= g_2(z)_* \circ I_z\circ (g_2(z)^{-1})_*$. Then $\widetilde{I_z}$ and $g^*\widetilde\phi$ satisfy (\ref{kk}) and the pair $g_2^*\left( \widetilde\alpha, \widetilde\phi \right)$ is $\widetilde{I_z}$-holomorphic.
\end{enumerate}

With notations revised, we summarize that, from the original $I$-holomorphic pair $(\alpha, \phi): S_N^+\to X$, we get a pair $\left( \widetilde{\alpha}, \widetilde{\phi} \right) : C_N\to N_{L, \delta_0}$ and a continuous family of almost complex structures $I_z$ on $N_{L, \delta_0}$ with the following properties:

\begin{itemize}

\item $\widetilde{\alpha}$ is in balanced temporal gauge, $\widetilde{\alpha}|_{\{0\}\times S^1}=0$;

\item $\widetilde{\phi}$ is reflection-symmetric;

\item $\left( \widetilde{\alpha}, \widetilde{\phi} \right) $ is $I_z$-holomorphic;

\item $I_z$ satisfies (\ref{kk}).
\end{itemize}

We want to point out that the pair $\left( \widetilde\alpha, \widetilde\phi \right)$ we constructed is not necessarily smooth, but smooth on both $S_N^+$ and $S_N^-$.

Moreover, if the original pair satisfies the vortex equation, then the ``upper half'' of the new pair $\left( \widetilde{\alpha}, \widetilde{\phi} \right)$ also satisfies the vortex equation since we only did gauge transformations on the upper half. Define a function $\widetilde\mu: C_N\to i{\mb R}$ as follows: $\widetilde\mu(z)=\mu(\phi(z))$ if $z\in S_N^+$ and $\widetilde\mu(z)=2c-\mu(\phi(\tau(z)))$ if $z\in S_N^-$. We can also construct a continuous volume form $\widetilde{\nu}$ on $C_N$ by reflection. The new pair satisfies the equation
$$
\iota_{\widetilde\nu} d\widetilde{\alpha}+\widetilde\mu=c.
$$
Similarly, if the original pair is flat on $S_N^+$(i.e. $d\alpha=0$), then $d\widetilde{\alpha}=0$ on $C_N$. In either case, we have the estimate
$$
\left\|\iota_{\widetilde{\nu}} d\widetilde{\alpha} \right\|_{L^\infty}\leq C.
$$

The family $I_z$ depends on the original pair $(\alpha, \phi)$. We will prove theorems which state that for a given family $I_z$, certain conditions hold for all $I_z$-holomorphic pairs in balanced temporal gauge, including in particular, the pair $\left( \tilde{\alpha}, \tilde{\phi}\right)$.

\subsection{Analogue of Theorem \ref{thm24}}

Let $K_0$ be the one in (\ref{k}).
\begin{thm}\label{thm: x1}
Fix some $c\in i{\mb R}$. There exist positive real numbers $\epsilon$, $\eta_0$, $\sigma$, $K$ with the following property. Let $N>0$ be a real number and let $I_z$ be a continuous family of almost complex structures on $N_{L, \delta_0}$, parametrized by $z\in C_N$. Let $(\alpha, \phi): C_N\to N_{L, \delta_0}$ be a pair such that 
\begin{itemize}

\item $\phi$ is reflection-symmetric;

\item $\alpha$ is in balanced temporal gauge and $\alpha|_{\{0\}\times S^1}=0$;
\item $\|\alpha\|_{L^\infty}<\epsilon, \ \|d_\alpha\phi\|_{L^\infty}<\epsilon$.
\end{itemize}
Then there exists maps $\psi: [-N, N]\to N_{L, \delta_0}$ and $\phi_0: Z_N\to TX$ satisfying that for any $(t, \theta)\in Z_N$, $\phi_0(t, \theta)\in T_{\psi(t)} X$,
$$
\phi(t, \theta)=\exp_{\psi(t)} \left( \phi_0(t, \theta) \right), \int_0^{2\pi} \phi_0(t, \theta) d\theta=0.
$$

Moreover, if $(\alpha, \phi)$ is $I_z$-holomorphic and $\left\| I_z-\widetilde{I_0} \right\| \leq K_0 {\rm dist}(\phi(z), L)$, and that
\begin{align}\label{curv}
|d\alpha|\leq \eta_0 e^{|t|-N},
\end{align}
then for any $(t, \theta)\in Z_N$,
$$|\phi_0(t, \theta)|\leq K e^{\sigma(|t|-N)}.$$
\end{thm}

\begin{rem} In Theorem 4.3 of \cite{Mundet_Tian_2009}, they distinguish two cases, $d\alpha=0$ or $(\alpha, \phi)$ satisfies the vortex equation. But in order to prove the estimate, we only need (\ref{curv}).
\end{rem}

\begin{proof} The existence of $\psi$ and $\phi_0$ follows by applying the implicit function theorem. In the following we prove the rest of the theorem.

First we prove a lemma in a local chart. Let $C=(-2, 3)\times S^1$ and $Z=[-1, 2]\times S^1$, $Z_I=[-1,0]\times S^1$, $Z_{II}= [0, 1]\times S^1$, $Z_{III}= [1, 2]\times S^1$. Let $X$ be a smooth vector field on ${\mb C}^n$ such that $\sigma_* X=X$, where $\sigma$ is the conjugation of ${\mb C}^n$.
\begin{lemma}\label{lemma: x2} For any $r>0$ and $K>0$ there exists an $\epsilon>0$ with the following property. Let $I_z$ be a continuous family of linear complex structures on ${\mb C}^n$ parametrized by points on $C$. Let $g$ be a Riemannian metric on ${\mb C}^n$ and ${\mc X}$ be a smooth vector field on $B(0, r)$, satisfying
\begin{itemize}

\item $\|I_z-I_0\|_{L^\infty}<\epsilon$, $\|g_0-g\|_{L^\infty}<\epsilon$, $\|Dg\|_{L^\infty}\leq K$, $\|{\mc X}\|_{L^\infty}\leq K$. 

\end{itemize}
Let $(\alpha, \phi): C\to {\mb C}^n$ be a pair, which satisfies
\begin{itemize}

\item $\alpha$ in balanced temporal gauge with $\alpha|_{\{0\}\times S^1}=0$;

\item $\partial_t\phi(z)=I_z \left(  \partial_\theta \phi(z)-i\alpha_\theta (z){\mc X}(\phi(z)) \right)$;

\item $\phi$ is reflection-symmetric, i.e., $\phi\circ \tau=\sigma\circ \phi$;

\item $\left\| I_z-I_0 \right\|_{L^\infty}\leq K  \left| {\rm Im} \phi(z) \right|$; 

\item $\left\| \alpha \right\|_{L^\infty}<\epsilon$, $\left\| d_{\alpha}\phi \right\|_{L^\infty}<\epsilon$, $\phi(C)\subset B(0, 2r)$. 
\end{itemize}
Let $\psi$, $\phi_0$ be defined as in Theorem \ref{thm: x1} and suppose that $\|d\alpha\|_{L^2}<\epsilon \|\phi_0\|_{L^2}$. Then in terms of $g$, we have
$$
\left\| \phi_0 \right\|_{L^2(Z_{II})}^2\leq {1\over 3}\left( \left\| \phi_0 \right\|_{L^2(Z_I)}^2+\left\| \phi_0 \right\|_{L^2(Z_{III})}^2\right).
$$
\end{lemma}

We will use the following two lemmas proved in \cite{Mundet_Tian_2009}.

\begin{lemma}[Lemma 11.4,\cite{Mundet_Tian_2009}]\label{lemma: 11.4} Let $\phi: C\to {\mb C}^n$ be a holomorphic map and let $\phi_{av}(t, \theta):=\phi(t, \theta)-{1\over 2\pi} \int_0^{2\pi} \phi(t, \nu) d\nu$. Then
$$
\left\| \phi_{av} \right\|_{L^2(Z_{II})}^2\leq {1\over 5} \left( \left\| \phi_{av} \right\|_{L^2(Z_I)}^2+ \left\| \phi_{av} \right\|_{L^2(Z_{III})}^2\right)
$$
where the norms are computed with respect to the standard metric on ${\mb C}^n$.
\end{lemma}

\begin{lemma}[Lemma 11.5, \cite{Mundet_Tian_2009}]\label{lemma: 11.5}
For any $K>0$, there exist constants $\epsilon>0$ with the following property: let $g_0$ be the standard metric on ${\mb C}^n$ and $g$ be another Riemannian metric on ${\mb C}^n$ satisfying $K^{-1} g_0\leq g\leq Kg_0$ and $\|Dg\|_{L^\infty({\mb C}^n, g_0)}\leq K$. Let $x\in {\mb C}^n$ and let $\gamma_0: S^1\to T_x{\mb C}^n$ be a smooth map satisfying $\int \gamma_0=0$ and $\sup |\gamma_0|<\epsilon$. Define $\gamma(\theta)=\exp_x^g \gamma_0(\theta)$ and
$$\gamma_{av}(\theta)=\gamma(\theta)-{1\over 2\pi} \int_0^{2\pi} \gamma(\nu)d\nu.$$
Then $0.9\|\gamma_0\|_{L^2}\leq \|\gamma_{av}\|_{L^2}\leq 1.1 \|\gamma_0\|_{L^2}$ and $0.9\sup|\gamma_0|\leq \sup|\gamma_{av}|\leq 1.1\sup|\gamma_0|$.
\end{lemma}

\begin{rem}\label{remark57} It is easy to see that for any $M>0$, there exists $K_M>0$ such that, if we replace the condition $\sup|\gamma_0|<\epsilon$ by $\sup|\gamma_0|\leq M$, then we have $\sup|\gamma_{av}|\leq K_M\sup|\gamma_0|$.
\end{rem}

\begin{proof}[Proof of Lemma \ref{lemma: x2}] Suppose the lemma is false, then there exists a sequence $\epsilon_u\to 0$, and a sequence of $I_{u, z}$, $g_u$, $\alpha_u$, $\phi_u$, ${\mc X}_u$ which satisfy the following conditions:
\begin{itemize}

\item $\left\| I_{u, z}-I_0 \right\|_{L^\infty}<\epsilon_u$, $ \left\| g_u-g_0 \right\|_{L^\infty}<\epsilon_u$, $\left\|Dg_u \right\|_{L^\infty}\leq K$, $ \left\| {\mc X}_u \right\|_{L^\infty}\leq K$;

\item $\alpha_u$ is in balanced temporal gauge with $\alpha_u|_{\{0\}\times S^1}=0$;

\item $\partial_t \phi_u(z)=I_{u, z} \left( \partial_\theta\phi_u(z)-i\alpha_{u, \theta}(z){\mc X}_u(\phi_u(z)) \right)$; 

\item $\phi_u$ is reflection-symmetric; 

\item $\left\| I_{u, z}-I_0 \right\|_{L^\infty}\leq K  \left| {\rm Im}\phi_u(z) \right|$; 

\item $\left\| \alpha_u \right\|_{L^\infty}<\epsilon_u$, $\left\|d_{\alpha_u}\phi_u \right\|_{L^\infty}<\epsilon_u$, $\phi_u(C)\subset B(0, 2r)$; 

\item $\left\|d\alpha_u \right\|_{L^2}<\epsilon_u \left\| \phi_{u, 0} \right\|_{L^2}$;

\item $\left\| \phi_{u, 0} \right\|_{L^2(Z_{II})}^2> {1\over 3}\left( \left\|\phi_{u, 0} \right\|_{L^2(Z_I)}^2+ \left\| \phi_{u, 0} \right\|_{L^2(Z_{III})}^2\right)$.

\end{itemize}

Then $\left\| d\phi_u \right\|_{L^\infty}\leq \left\| d_{\alpha_u}\phi_u \right\|_{L^\infty}+\left\| \alpha_u \right\|_{L^\infty} \left\| {\mc X}_u \right\|_{L^\infty}\to 0$. This implies that $\|\phi_{u, 0}\|_{L^\infty}\to 0$. Hence by Lemma \ref{lemma: 11.5}, for big enough $u$, one has
\begin{align}\label{53}
\left\| \phi_{u, av} \right\|_{L^2(Z_{II})}^2>{1\over 4} \left( \left\| \phi_{u, av} \right\|_{L^2(Z_I)}^2+ \left\| \phi_{u, av} \right\|_{L^2(Z_{III})}^2\right);
\end{align}
and \begin{align}\label{54}
{\|d\alpha_u\|_{L^2}\over \|\phi_{u, av}\|_{L^2}}\to 0.
\end{align}

Let $x_u=\int_{S^1} \phi_u(0, \theta) d\theta$. By the assumption that $\phi$ is reflection-symmetric, one has that $x_u\in {\mb R}^n$. Set $\xi_u=\phi_u-x_u$. Then $\|d\xi_u\|_{L^2}>0$, for otherwise one would get $\xi_{u, av}=\phi_{u, av}=0$, which contradicts with (\ref{53}).

{\bf Claim.} There exists a constant $M>0$ such that for all $u$, $\left\| \xi_u \right\|_{L^2}\leq M \left\| d\xi_u \right\|_{L^2}$, $\left\| \phi_{u, av} \right\|_{L^2} \leq M \left\| d\phi_{u, av} \right\|_{L^2}$ and  $\left\| \alpha_u \right\|_{L^2}\leq M \left\| d\alpha_u \right\|_{L^2}$, where the norms are taken with respect to the metric $g_0$.

\begin{proof}[Proof of the claim]
Using the Fourier expansion $\xi_u=\sum_k a_{u, k}(t) e^{ik\theta}$. The condition $\int_{S^1}\xi_u(0, \theta) d\theta=0$ implies that $a_{u, 0}(0)=0$. While $d\xi_u= dt \sum_k a_{u, k}'(t) e^{ik\theta}+ d\theta\sum_k a_{u, k}(t) ik e^{ik\theta}$. Hence
$$\int_{-2}^3 \left| a_{u, 0}(t) \right|^2 dt=\int_{-2}^3\left|\int_0^t a_{u, 0}'(s) ds\right|^2 dt\leq \int_{-2}^3  |t|\left(\int_0^t \left| a_{u, 0}'(s) \right|^2 ds\right) dt\leq 100 \int_{-2}^3 \left| a_{u, 0}'(t) \right|^2 dt.$$
So
\begin{multline*}
\|\xi_u\|_{L^2}^2 =\sum_{k\neq 0} \int_{-2}^3 \left| a_{u, k}(t) \right|^2 dt+\int_{-2}^3 \left| a_{u, 0}(t) \right|^2 dt\\
\leq \sum_{k\neq 0} \int_{-2}^3 \left| a_{u, k}(t) ik \right|^2 dt+100 \int_{-2}^3 \left| a_{u, 0}'(t) \right|^2 dt
\leq 100 \left\| d\xi_u \right\|_{L^2}^2
\end{multline*}
Since $\int_{S^1} \phi_{u, av}(0, \theta)=0$ and $\int_{S^1}\alpha_u(0, \theta)=0$, using the same method, we obtain the other two estimates. 
\end{proof}

{\bf Claim.} We have $\left\| \ov\partial_{I_0}\xi_u \right\|_{L^2}/ \left\| d\xi_u \right\|_{L^2}\to 0$ as $u$ goes to infinity.

\begin{proof}[Proof of the claim] We have $\partial_t \phi_u(z)= I_{u, z} \left( \partial_\theta \phi_u- i\alpha_u {\mc X}_u(\phi_u(z)) \right)$. Hence
$$\ov\partial_{I_0} \xi_u=\partial_t \xi_u-I_0\partial_\theta \xi_u=(I_{u, z}-I_0)\partial_\theta \xi_u-i\alpha_u I_{u, z} {\mc X}_u(\phi_u(z)),$$
and
\begin{multline*}
\left\| \ov\partial_{I_0} \xi_u\right\|_{L^2} \leq \left\| I_{u, z}- I_0\right\|_{L^2} \left\| \partial_\theta \xi_u\right\|_{L^\infty} + \left\| i\alpha_u\right\|_{L^2} \left\| I_{u, z} {\mc X}_u\right\|_{L^\infty} \\
\leq K\left\| {\rm Im} \phi_u(z)\right\|_{L^2} \left\| \partial_\theta\phi_u\right\|_{L^\infty}+ 2KM \left\|d\alpha_u\right\|_{L^2}
\leq K\left\| {\rm Im} \xi_u\right\|_{L^2} \left\| d\phi_u\right\|_{L^\infty} + 2KM \epsilon_u \left\|\phi_{u, 0}\right\|_{L^2}  \\
\leq K \left\| \xi_u\right\|_{L^2}\left\| d\phi_u \right\|_{L^\infty}  + 3KM \epsilon_u \left\| \phi_{u, av}\right\|_{L^2}
\leq KM \epsilon_u \left\| d\xi_u\right\|_{L^2} + 2KM^2 \epsilon \left\| d\phi_{u, av}\right\|_{L^2}\\
 \leq \epsilon_u \left( KM \left\| d\xi_u\right\|_{L^2} + 2KM^2 \left\| d\xi_{u, av}\right\|_{L^2} \right).
\end{multline*}
It is easy to see that $\left\| d\xi_u\right\|_{L^2}$ can contral $\left\| d\xi_{u, av}\right\|_{L^2}$. Hence the claim is proved.
\end{proof}

Now define $\xi_u'=\xi_u/\|d\xi_u\|_{L^2}$. Since $\int_{S^1} \xi_u'(0,\theta)d\theta=0$ and $\|d\xi_u'\|_{L^2}=1$, the $L^2$-norm of $\xi_u'$ is uniformly bounded. Hence we have a uniform $W^{1, 2}$-bound of $\xi_u'$. Then since the inclusion $W^{1, 2}\subset L^2$ is compact, there exists a subsequence, which is still denoted by $\xi_u'$, converging in $L^2$ to some $\xi\in L^2(C, {\mb C}^n)$. By the above claim, $\xi$ is a weak solution of $\ov\partial \xi=0$ and hence holomorphic in the interior of $C$. By Garding's inequality $\|\xi_u'-\xi\|_{W^{1, 2}}\leq {\rm const.} (\|\ov\partial(\xi_u'-\xi)\|_{L^2}+ \|\xi_u'-\xi\|_{L^2})$, $\xi_u'$ converges to $\xi$ in $W^{1, 2}$ and hence $\|d\xi\|_{L^2}=1$ and hence not constant.

Since $\xi_u'\to \xi$ in $L^2$, it follows that $\xi_{u, av}'\to \xi_{av}$ in $L^2$. Since $\xi_{u, av}'$ is a rescaling of $\xi_{u, av}$, $\xi_{u, av}=\phi_{u, av}$ and $g_u\to g_0$, inequality (\ref{53}) implies that
$$\|\xi_{av}\|_{L^2(Z_{II})}^2\geq {1\over 4}\left(\|\xi_{av}\|_{L^2(Z_I)}^2+\|\xi_{av}\|_{L^2(Z_{III})}^2\right).
$$
But this will contradicts Lemma \ref{lemma: 11.4} since $\xi$ is non-constant.
\end{proof}

Then we have the corollary in $N_{L, \delta_0}$ instead of ${\mb C}^n$.

\begin{cor}
There exists an $\epsilon>0$ with the following property. Let $I_z$ be a continuous family of almost complex structures on $N_{L, \delta_0}$ parametrized by $z\in C$. Let $(\alpha, \phi): C\to N_{L, \delta_0}$ be an $I_z$-holomorphic pair, which satisfies: 
\begin{itemize}

\item $\alpha$ is in balanced temporal gauge with $\alpha|_{\{0\}\times S^1}=0$;

\item $\phi$ is reflection-symmetric; 

\item $ \left\| I_z-\widetilde{I_0} \right\|_{L^\infty}\leq K_0 {\rm dist}_g (\phi(z), L)$;

\item $\|\alpha\|_{L^\infty}<\epsilon$, $\|d_\alpha\phi\|_{L^\infty}<\epsilon$.
\end{itemize}
Let $\psi$, $\phi_0$ be defined as in Theorem \ref{thm: x1} and suppose that $\|d\alpha\|_{L^2}<\epsilon\|\phi_0\|_{L^2}$. Then we have
$$\|\phi_0\|_{L^2(Z_{II})}^2\leq {1\over 3}\left(\|\phi_0\|_{L^2(Z_I)}^2+\|\phi_0\|_{L^2(Z_{III})}^2\right).$$
\end{cor}

\begin{proof} First, take $\delta$ small enough such that for each $y\in L$, the image of $f_{y, \delta}$ lies in the unit ball. Recall the definition of $K_{y, \delta}$ in (\ref{x: K}). Set $K_y'=3( K_0+ K_{y, \delta}+\|{\mc X}\|_{L^\infty})$ and apply Lemma \ref{lemma: x2} for $r=1$ and $K=K_y'$, one gets an $\epsilon_{y, \delta}$. Then for each $y\in L$, take $\delta(y, \epsilon)<\delta$ such that (\ref{x:2}) holds for the $\epsilon=\epsilon_{y, \delta}$.

Then $\{B_L(y, \delta(y))\}_{y\in L}$ gives an open cover of $L$ and it has a finite subcover $\{B_L(y_i, \delta_i)\}$. Now, there exists $\epsilon'>0$ such that for any pair $(\alpha, \phi)$ with $\|\alpha\|_{L^\infty}<\epsilon'$, $\|d_\alpha\phi\|_{L^\infty}<\epsilon'$, the image must lie in one of the $B_L(y_i ,\delta_i)$'s. Take $\epsilon={\rm min}(\epsilon', \epsilon_{y_i, \delta_i})$. Now we check that this $\epsilon$ satisfies the required condition. Suppose $(\alpha, \phi)$ is a pair and $I_z$ is a family of almost complex structures which satisfy the hypothesis. Then the image $\phi$ is contained in some $B_L(y_i, \delta_i)$ and hence it descends to a problem in ${\mb C}^n$. But by our definition of $K_{y_i}'$, $\epsilon$ and $\delta(y)$, the pair $(\alpha, \phi)$ and $I_z$ satisfies the hypothesis of Lemma \ref{lemma: x2} and hence one has the desired inequality.
\end{proof}

Then, use the same method as in \cite{Mundet_Tian_2009}, we prove Theorem \ref{thm: x1}.
\end{proof}

\subsection{Analogue of Theorem \ref{thm25}}

Now let's consider the following case. Let $I_z$ be a continuous family of almost complex structures on $N_{L, \delta_0}$, parametrized by $z\in C_N$, and let $(\alpha, \phi): C_N\to N_{L, \delta_0}$ be an $I_z$-holomorphic pair. Let $\epsilon$ be that in the previous theorem and $\|\alpha\|_{L^\infty}<\epsilon$, $\|d_\alpha\phi\|_{L^\infty}<\epsilon$. This implies the existence of $\psi$ and $\phi_0$. Furthermore, by implicit function theorem, for any $(t, \theta)\in C_N$,
$$\left| \phi_0(t, \theta)\right|,  \left| \nabla\phi_0(t, \theta) \right|\leq {\rm const.} \sup_{\nu\in S^1} \left|\nabla \phi(t, \nu) \right|.$$

\begin{thm}\label{thm56}
There exist positive constants $K, \sigma$ with the following property. Let $N>0$ be a real number and $(\alpha, \phi): C_N\to N_{L, \delta_0}$ be an $I_z$-holomorphic pair satisfying the hypothesis of the previous theorem. If we have
$$|d\alpha|\leq \eta_0 e^{|t|-N},$$
then for any $t\in [-N, N]$
$$|\psi'(t)|\leq K e^{\sigma(|t|-N)}.$$
\end{thm}

We first prove a lemma.

\begin{lemma}
Let $K, M>0$ be real numbers, let $I_z$ be a continuous family of complex structures on ${\mb C}^n$, $g$ be a Riemannian metric on ${\mb C}^n$ and ${\mc X}$ be a symmetric vector field on ${\mb C}^n$. Let $V\subset {\mb C}^n$ be a compact subset. Let $\phi: Z\to {\mb C}^n$ be a reflection-symmetric map satisfying: $\phi(Z)\subset V$; for each $z\in Z$, $\|I_z-I_0\|\leq K |{\rm Im} \phi(z)|$ and
$$\partial_t\phi(z)=I_z\left( \partial_\theta\phi(z)-i \alpha{\mc X}(\phi(z)) \right)$$
for some function $\alpha: Z\to i{\mb R}$. Suppose that $\|\nabla\phi\|_{L^\infty}$ is small enough so that $\psi: [-1, 2]\to {\mb C}^n$ and $\phi_0: Z\to {\mb C}^n$ can be defined as in Theorem \ref{thm: x1}, and that
\begin{align}\label{528}
\left| \phi_0 \right|, \left| \partial_\theta\phi_0 \right|, \left| \partial_t\phi_0 \right|\leq M.
\end{align}
Then for any $t\in (0, 1)$ one has
\begin{align}\label{529}
\left| \psi'(t) \right| \leq K'\left( \sup_{Z_t} |\phi_0|+\sup_{Z_t} |\alpha| \right)
\end{align}
for some constant $K'$ independent of $\phi$.
\end{lemma}

\begin{proof} The proof is almost the same as that of \cite[Lemma 12.1]{Mundet_Tian_2009}. We have $\psi([0, 1])\subset \tilde{V}=\left\{ \exp_x^g v| x\in V, |v|\leq M \right\}$. Let $B\subset {\mb C}^n$ be the closed ball of radius $M$, let $E\times \tilde{V}\to {\mb C}^n$ be the exponential map $E(x, v)=\exp_x^g v$. Let $F=D_x E$ and $G=D_v E$. Since the domain is compact, the second derivatives of $E$ are uniformly bounded. Hence there exists $K_1$ such that for any $(x, v)\in \tilde{V}\times B$,
\begin{align}\label{530}
|F(x, v)-{\rm Id}|\leq K_1 |v|,\ |G(x, v)-{\rm Id}|\leq K_1|v|.
\end{align}
Take partial derivatives with respect to $t$ in the equation $\exp_\psi \phi_0=\phi$, one obtains
$$F(\psi, \phi_0)\psi'+ G(\psi, \phi_0)\partial_t \phi_0=\partial_t \phi.$$
Writing $F={\rm Id}+ F-{\rm Id}$ and $G={\rm Id}+ G-{\rm Id}$, and integrating over $S^1$ of the above equation, making use of the fact that $\int \phi_0(t, \theta) d\theta=0$ and (\ref{528}), (\ref{530}), it yields
\begin{align}\label{532}
\left|\psi'(t)-{1\over 2\pi}\int \partial_t \phi(t, \theta)d\theta\right|\leq K_2\sup_{Z_t}|\phi_0|.
\end{align}

On the other hand,
\begin{align}\label{533}
\partial_t\phi(z)=I_z \left(\partial_\theta\phi(z)-i\alpha {\mc X}(\phi(z))\right)=(I_z-I_0)\left(\partial_\theta\phi(z)\right)+I_0\partial_\theta\phi(z)- i\alpha I_z{\mc X}(\phi(z)).
\end{align}
We also have
$$\partial_\theta \phi= G(\psi, \phi_0) \partial_\theta \phi_0 \Rightarrow \left| \partial_\theta \phi \right|\leq  K_3 \sup \left| \partial_\theta \phi_0 \right|\leq K_3 M.$$

Integrating (\ref{533}) over $S^1$, using the above estimate and see there exist $K_4$ and $K_5$ such that
\begin{multline*}
 \left|\int_{S^1}  \partial_t \phi(t, \theta)d\theta\right|\leq K_4\left(\sup_{Z_t} |{\rm Im}(\phi)|+\sup_{Z_t}|\alpha|\right)\\
\leq K_4\left(\sup_{Z_t}|\phi_{av}|+\sup_{Z_t}|\alpha|\right)\leq K_5\left(\sup_{Z_t}|\phi_0|+\sup_{Z_t}|\alpha|\right).
\end{multline*}
Here $|{\rm Im} \phi|\leq |\phi_{av}|$ uses the fact that $\phi$ is reflection-symmetric; the last inequality uses the estimate in the Remark \ref{remark57}. Hence combining (\ref{532}) and above, one obtains (\ref{529}).
\end{proof}

\begin{proof}[Proof of theorem \ref{thm56}] It follows word by word from the proof of \cite[Theorem 4.4]{Mundet_Tian_2009}. 
\end{proof}

\begin{cor}\label{cor58}
Let $c\in i{\mb R}$. Then there exist real numbers $\epsilon_5, \sigma_5, \eta_0, K_5>0$ such that for any positive $N>0$ and any $I$-holomorphic pair $(\alpha, \phi): C_N^+\to X$, which satisfies
\begin{align*}
|d\alpha|\leq \eta_0 e^{|t|- N},\  \left\| d_\alpha \phi \right\|_{L^\infty}\leq \epsilon_5,
\end{align*}
one has
$${\rm diam}_{S^1} \phi\left( S_{N/2}^+ \right)\leq K_5 e^{-\sigma_5 N}.$$
\end{cor}

\section{Definition of $(c, L)$-stable twisted holomorphic maps}

\subsection{Prestable bordered Riemann surfaces}

\begin{defn}
Let $(x, y)$ be the coordinate on ${\mb C}^2$, and $A(x, y)=(\ov{x}, \ov{y})$ be the complex conjugation. A {\bf node} on a bordered Riemann surface is a singularity isomorphic to one of the following:
\begin{enumerate}
\item $(0, 0)\in \{xy=0\}$ (interior node);

\item $(0, 0)\in \left\{ x^2+y^2=0 \right\}/A$ (boundary node of type 2);

\item $(0, 0)\in \left\{ x^2-y^2=0 \right\}/A$ (boundary node of type 1).

\end{enumerate}
A {\bf nodal bordered Riemann surface} is a singular bordered Riemann surface whose singularities are nodes. A {\bf prestable bordered Riemann surface} is either a smooth bordered Riemann surface or a nodal bordered Riemann surface.
\end{defn}

The {\bf complex double} of a prestable bordered Riemann surface $C$ is a prestable Riemann surface without boundary, which is intuitively, gluing $C$ and $\ov{C}$(same curve with opposite complex structure) along the corresponding boundary components.

Let $g$ be a nonnegative integer and $h$ a positive integer. A {\it smooth} bordered Riemann surface $C$ is said to be {\bf of type $(g, h)$} if $C$ is topologically a sphere attached with $g$ handles and $h$ disks removed. A nodal bordered Riemann surface is of type $(g, h)$ if it is a degeneration of a smooth bordered Riemann surface of the same type. The reader may refer to \cite{Melissa_Liu_Thesis} for the precise definition.

Let $C$ be a prestable bordered Riemann surface, let ${\bf z}\subset C$(resp. ${\bf x}\subset \partial C$) be a finite subset of interior(resp. boundary) points which is disjoint from nodes. Then the tuple $(C, {\bf x}\cup {\bf z})$ is called a {\bf marked nodal bordered Riemann surface}. There is a unique curve $\hat{C}$ which is the disjoint union of finitely many components which are either smooth marked bordered Riemann surface or marked smooth Riemann surface without boundary, together with a holomorphic map $\hat\pi: \hat{C} \to C$. This map is called the {\bf normalization} of $C$. $\hat\pi$ is 1-1 away from interior nodes and boundary nodes of type 1; The preimage of every interior node or boundary node of type 1 is two points.

\subsection{Meromorphic connections on prestable marked bordered Riemann surfaces}

Now let $(C, {\bf x}\cup {\bf z})$ is a pretable marked bordered Riemann surface. Let $\hat\pi: \hat{C} \to C$ be the normalization map. Denote by ${\bf w}$ the set of nodes of $C$, with ${\bf w}^{int}$, ${\bf w}^{b1}$, ${\bf w}^{b2}$ the subsets of different types of nodes, respectively. Given an $S^1$-principal bundle $P$ over $\hat{C}\setminus \hat\pi^{-1} \left( {\bf z}\cup {\bf w}^{int}\cup {\bf w}^{b2} \right)$.

\begin{defn}A meromorphic connection on $P$ is a connection $A$ on $P$ such that $F_A$ is bounded for any smooth metric on $\hat{C}$ and that for every $w\in {\bf w}^{int}$ with preimages $y, y'\in \hat{C} $, we have
$${\rm Hol}(A, y)={\rm Hol}(A, y')^{-1}.$$
\end{defn}

\subsection{Gluing data at interior nodes and boundary nodes of type 1}

Now let $(C, {\bf x}\cup {\bf z})$ be a smooth marked bordered Riemann surface, let $P\to {C}\setminus {\bf z} $ be an $S^1$-principal bundle and let $A$ be a meromorphic connection on $P$. Let $z\in {\bf z}$ and let $S_z$ be the quotient space of $T_z C-\{0\}$ by the action of ${\mb R}_+ $ given by scalar multiplication. Choose a conformal metric on $C$ and let $S_\epsilon\subset C$ be the circle of radius $\epsilon$ centered at $z$, and denote $(P_\epsilon, A_\epsilon)$ be the restriction of $(P, A)$ to $S_\epsilon$. Let $P_z\to S_z$ be the $S^1$-principal bundle whose fibre over the class of $s\in T_z C-\{0\}$ is the set of $A$-covariantly constant sections of $P$ defined over the ray $\{\exp ts| t\in (0, \delta)\}$ for some small $\delta$. Parallel transport in radial directions gives the canonical isomorphism of bundles $\iota_\epsilon: P_z\to P_\epsilon$ for small enough $\epsilon$ and the connections $\iota_\epsilon^*A_\epsilon$ converges to a connection $A_z$ on $P_z$. The pair $(P_z, A_z)$ is independent of the conformal metric on $C$. We call $(P_z, A_z)$ the {\bf limit pair} of $(P, A)$ at $z$.

Now let $(C, {\bf x}\cup {\bf z})$ be a marked nodal bordered Riemann sufrace and let $\hat\pi: \hat{C}\to C$ be the normalization. If $w\in {\bf w}^{int}$ with preimages $y, y'$. Then the circles $S_y, S_{y'}$ have natural structures of $S^1$-torsors. Define $G_w(C)$ to be the set of all isomorphisms of $S^1$-torsors $$g_C: S_y\to \ov{S_{y'}},$$ where $\ov{S_{y'}}$ is $S_{y'}$ with the inverse $S^1$-torsor structure.

Suppose $P$ is an $S^1$-principal bundle over $\hat{C}\setminus \hat\pi^{-1}\left({\bf z}\cup {\bf w}^{int}\cup {\bf w}^{b2} \right)$, and $A$ is a meromorphic connection on $P$. 

Let $(P_y, A_y)$ and $(P_{y'}, A_{y'})$ be the limit pairs of $(P, A)$ at $y, y'$ respectively. Then define the set of {\bf gluing data} at $w\in {\bf w}^{int}$ to be the set of pairs $(g_C, g_P)$, where $g_C\in G_w(C)$ and $g_P: P_y\to g_C^* P_{y'}$ is an isomorphism of $S^1$-bundles such that $g_P^* g_C^* A_{y'}= A_y$.

Now suppose $w' \in {\bf w}^{b1}$ is a boundary node of type 1 with preimages $y, y'$. Note that $P$ is defined over $y$ and $y'$. Then we define the set of gluing data at $w'$ to be ${\rm Iso}(P_{y}, P_{y'})$, the set of isomorphisms between the two fibres.

\subsection{Stable marked bordered curves and the moduli space}

Let $(C, {\bf x}\cup {\bf z})$ be a prestable marked bordered Riemann surface. We say that a point in $C$ is exceptional if it is either a node or a marked point. Let $\hat\pi: \hat{C}\to C$ be the normalization. The exceptional points in $\hat{C}$ is by definition the preimages of the exceptional points in $C$. An irreducible component $D$ of $C$ is called unstable, if $\pi^{-1}(D)$ satisfies one of the following conditions: 
\begin{itemize}
\item it is a sphere with 2 or less exceptional points; 

\item it is a torus with no exceptional points; 

\item it is a disc with 1 interior exceptional point and no exceptional points on the boundary; 

\item it is a disc with no interior exceptional points and 2 or less exceptional points on the boundary; 

\item it is a ``half-torus'' with no exceptional points, i.e. its {\it complex double} is a 2-torus without exceptional points.
\end{itemize}
If no component of $\hat{C}$ is unstable, then we say that $C$ is {\bf stable}. 

We say that $(g, h, n, {\ora m})$ is in the {\it stable domain}, if a smooth bordered Riemann surface of genus $g$ and $h$ boundary components, with $(n, {\ora m})$-marked points is stable. For the moduli space ${\mc M}_{(g, h), (n, \ora{m})}$ of smooth bordered Riemann surfaces of genus $g$ and $h$ boundary components, with $(n, \ora{m})$-marked points, and its compactification $\ov{\mc M}_{(g, h), (n, \ora{m})}$ of stable curves, the reader may refer to \cite{Melissa_Liu_Thesis}. From now on we fix a $(g, h, n, \ora{m})$ in the stable domain.

\subsection{Collapsing and stablization maps}\label{section65}

Let $(C, {\bf x}\cup {\bf z})$ be a prestable marked bordered Riemann surface. A {\bf collapsing map} is a quotient map $\pi^{co}: C\to C'$, where $(C', {\bf x}'\cup {\bf z}')$ is also a prestable marked bordered Riemann surface, and $\pi^{co}$ is quotienting some of the unstable components of $C$. The marked point set ${\bf x}'\cup {\bf z}'$ is naturally induced from ${\bf x}\cup {\bf z}$, such that $\pi^{co}: {\bf x}\to {\bf x}'$, $\pi^{co}: {\bf z}\to {\bf z'}$ are bijections.

Now for a given collapsing map $\pi^{co}: C\to C'$. Then each node of $C'$ falls in one of the following cases:

\begin{enumerate}
\item If $w$ is an interior node, one can find sequences of points of $\hat{C}$, $(y_0, \ldots, y_{d-1})$ and $(y_1', \ldots, y_d')$, which satisfy the following properties:

\begin{itemize}

\item the points $\{y_0, \ldots, y_{d-1}\}$ and $\{y_1', \ldots, y_d'\}$ are all different;

\item $\hat{\pi}(y_0)$ and $\hat{\pi}(y_d')$ belong to componens on which $\pi^{co}$ is a local bijection;

\item each pair $\{\hat{\pi} (y_j), \hat{\pi} (y_j')\}$ is contained in a sphere bubble $O_j\subset C$ for each $1\leq j\leq d-1$;

\item for each $j$ we have $\hat{\pi} (y_j)=\hat{\pi} (y_{j+1}')$;

\item we have $\pi^{co} \left( \hat{\pi}(y_j) \right) =\pi^{co}\left(  \hat{\pi}(y_j') \right) =w$ for each $j$.

\end{itemize}

We call the sequences $(y_0, \ldots, y_{d-1})$, $(y_1', \ldots, y_d')$ an {\bf ordered list of connecting exceptional points over $w$}.

\item If $w'$ is a boundeary node of type 1, one can find sequences of points of $\partial \hat{C}$, $(y_0, \ldots, y_{d-1})$ and $(y_1', \ldots, y_d')$, which satisfies the following properties:

\begin{itemize}

\item the points $\{y_0, \ldots, y_{d-1}\}$ and $\{y_1', \ldots, y_d'\}$ are all different;

\item $\hat\pi(y_0)$ and $\hat\pi(y_d')$ belong to components on which $\pi^{co}$ is a local bijection;

\item each pair $\{\hat{\pi} (y_j), \hat{\pi} (y_j')\}$ is contained in a disk bubble $D_j\subset C$ for each $1\leq j\leq d-1$;

\item for each $j$ we have $\hat{\pi} (y_j)=\hat{\pi} (y_{j+1}')$;

\item we have $\pi^{co} \left(\hat{\pi} (y_j)\right)=\pi^{co} \left( \hat{\pi} (y_j')\right) =w' $ for each $j$.

\end{itemize}
We call the sequences $(y_0, \ldots, y_{d-1})$, $(y_1', \ldots, y_d')$ an {\bf ordered list of connecting exceptional points over $w'$}.

\item If $w''$ is a boundary node of type 2, one of the following two cases must hold.

{\bf Case 1.} One can find a sequence of points of $\hat{C} $, $(y_0, \ldots, y_{d-1}$) and $(y_1', \ldots, y_{d-1}')$ which satisfy the following properties:

    \begin{itemize}
    \item the points $\{y_0, \ldots, y_{d-1}\}$ and $\{y_1', \ldots, y_{d-1}'\}$ are all different;

    \item $\hat\pi(y_0)$ belong to a component of $C$ on which $\pi^{co}$ is a local bijection, and $\hat{\pi} (y_{d-1})$ is a boundary node of type 2;

    \item each pair $\{\hat{\pi} (y_j), \hat{\pi} (y_j)\}$ is contained in a sphere bubble $B_j\subset C$ for each $1\leq j\leq d-1$;

    \item for each $j$ with $0\leq j\leq j-2$ we have $\hat{\pi} (y_j)=\hat{\pi} (y_{j+1}')$;

    \item we have $\pi^{co}\left(\hat{\pi} (y_j)\right)=\pi^{co} \left( \hat{\pi} (y_j')\right)=w''$ for each $j$.
\end{itemize}

{\bf Case 2.} One can find a sequence of points of $\hat{C} $, $(y_0, \ldots, y_{d-1})$ and $(y_1', \ldots, y_d')$ which satisfy the following properties:
\begin{itemize}

    \item the points $\{y_0,\ldots, y_{d-1}\}$ and $\{y_1', \ldots, y_d'\}$ are all different;
    
    \item $\hat\pi(y_0)$ belong to a component of $C$ on which $\pi^{co}$ is a local bijection;
    
    \item each pair $\left\{ \hat{\pi}  (y_j) ,\hat{\pi}  (y_j') \right\}$ is contained in a sphere bubble $O_j$ for $1\leq j\leq d-1$ and $y_d'$ is contained in a disk bubble $O_d$;
    \item for each $j$ with $0\leq j\leq j-1$ we have $\hat{\pi} \left( y_j \right) =\hat{\pi} \left( y_{j+1}' \right)$;
    \item we have $\pi^{co} \left(\hat{\pi} (y_j) \right)=\pi^{co}\left( \hat{\pi} (y_j') \right)=w''$ for each $j$.
\end{itemize}
In either case, we call the corresponding list of points an {\bf ordered list of connecting exceptional points over $w''$}. The difference is that in the first case a disk is collapsed and in the second case there is no disk.
\end{enumerate}

In particular, the collapsing map collapses all the unstable components of $C$ is called the {\bf stablization map}, and denoted by $\pi^{st}: (C, {\bf x}\cup {\bf z})\to (C^{st}, {\bf x}^{st}\cup {\bf z}^{st})$. In this case, the components of $C$ that are contracted by $\pi^{st} $ are called {\bf bubble components}(since there is no unstable torus component), and others are called {\bf principal components}. A bubble component is either a sphere or a disk, which are called respectively a {\bf sphere bubble} or a {\bf disk bubble}. The points $y_j$ or $y_j'$ appeared above are called {\bf connecting exceptional points}. The {\bf connecting nodes} are the images of the connecting exceptional points under the map $p$. The nodes which are not connecting are called {\bf tree nodes}. We call the bubbles $O_i$ or $D_j$ the {\bf connecting bubbles} over $w$. We call {\bf tree bubbles} those bubbles which are not connecting.

\subsection{Volume forms}

Let $g, n$ be nonnegative integers with $2g+n\geq 3$ and let $\ov{\mc M}_{g, n}$ be the moduli space of stable curves and let $f: \ov{\mc M}_{g, n+1}\to \ov{\mc M}_{g, n}$ be the forgetful map. Then we can choose a smooth(in the orbifold sense) Hermitian metric $h$ on the universal curve $\ov{\mc M}_{g, n+1}$, which gives a conformal metric $h_{(C, {\bf x})}$ on each stable curve $(C, {\bf x})$ with $[(C, {\bf x})]\in \ov{\mc M}_{g, n}$ and the metric is invariant under the automorphism group of $(C, {\bf x})$. We call $h_{(C, {\bf x})}$ the {\bf canonical metric} on $(C, {\bf x})$ and its volume form the {\bf canonical volume form} on $(C, {\bf x})$. From now on we fix such an $h$ for each pair $(g, n)$.

Now suppose $(C, {\bf x}\cup {\bf z})$ be a marked stable bordered Riemann surface. Then it has a {\bf complex double}, which is a stable curve $\left( C_{\mb C}, {\bf z}_{\mb C}=\left\{ z_i^\pm, x_j\right\} \right)$ with an anti-holomorphic involution $\tau$ such that $\tau(z_i^+)=\tau(z_i^-)$ and $\tau(x_j)=x_j$. Then the metric ${1\over 2} \left( h_{\left(C_{\mb C}, {\bf z}_{\mb C}\right)}+\tau^* h_{\left(C_{\mb C}, {\bf z}_{\mb C}\right)} \right) $ is a metric which descends to a metric on the original bordered curve $C$. We denoted it by $h_{(C, {\bf x}\cup {\bf z})}$ and call it the {\bf canonical metric} on $(C, {\bf x}\cup {\bf z})$ and its volume form the {\bf canonical volume form} on $(C, {\bf x}\cup {\bf z})$. 

In the following, for any prestable curve $C$ with stablization $C^{st}$, we will use metrics on each irreducible components of $C$ as: if $D$ is a principal component, then use the metric induced from the canonical metric $h_{C^{st}}$; if $D$ is a bubble component, then use an arbitrary smooth metric.

\subsection{Definition of $(c, L)$-STHM's} 

Pick an element $c\in i{\mb R}$, and $(g, h, n, {\overrightarrow m})$ in the stable domain. A {\bf $(c, L)$-stable twisted holomorphic map} from a bordered curve of genus $g$, with $h$ boundary components and $(n, {\overrightarrow m})$ marked points is a tuple
$${\mc C}= \left((C, {\bf x}\cup{\bf z}), (P, A, \varphi), G, \{{\mc T}_w\}\right)$$
where

\begin{enumerate}

\item $(C, {\bf x}\cup {\bf z})$ is a marked prestable bordered Riemann surface and the isomorphism class of its stabilization belongs to $\ov{\mc M}_{(g, h), (n, {\overrightarrow m})}$.

\item $P$ is an $S^1$-principal bundle over $\hat{C}\setminus \hat\pi^{-1}\left( {\bf z}\cup {\bf w}^{int}\cup {\bf w}^{b2} \right)$, where $\hat\pi: \hat{C}\to C$ is the normalization map.

\item $A$ is a meromorphic connection on $P$.

\item $\varphi$ is a smooth section of the bundle $P\times_{S^1} M$, such that $\varphi(x)\in P\times_{S^1} L$ for any $x$ in a boundary circle(but not boundary nodes of type 2) of $C$.

\item $G=\left\{ G_w \right\}$ is a choice of gluing data at each node $w\in {\bf w}^{int} \cup {\bf w}^{b1}$.

\item For each interior connecting node $z\in C$ with preimage $\left\{ y_j, y_j' \right\}$ in $C'$, ${\mc T}_z$ is a pair of sections ${\mc T}_y: S_y\to P_y\times_{S^1} {\mc T}(X)$ and ${\mc T}_{y'}: S_{y'}\to P_{y'}\times_{S^1} {\mc T}(X)$ satisfying ${\mc T}_y={\mc T}_{y'}\circ g_z$.

\item For each boundary node $w''\in C$ of type 2, ${\mc T}_{w''}$ is a section ${\mc T}_{w''}: S_{w''}\to P_{w''} \times_{S^1} {\mc T}(X)$.
\end{enumerate}
The tuple ${\mc C}$ must satisfy the following conditions.

\begin{enumerate}

\item {\bf The section is holomorphic.} The section $\varphi$ satisfies the equation
$$\ov\partial_A \varphi=0.$$

\item {\bf Vortex equation.} For each {\it principal component} $D\subset C$, let $\nu_D$ (resp. $A_D$, $\phi_D$) be the restriction of $s^* v_{(C^{st}, {\bf x}^{st}\cup {\bf z}^{st})}$(resp. $A$, $\phi$), where $s: (C, {\bf x}\cup {\bf z})\to (C^{st}, {\bf x}^{st}\cup {\bf z}^{st})$ is the stabilization map; then
$$\iota_{\nu_D} F_{A_D}+\mu(\varphi_D)=c.$$

\item {\bf Flatness on bubbles.} The restriction of $A$ to each bubble component is flat.

\item {\bf Critical holonomy at interior marked points.} For any interior marked point $z\in {\bf z}$, the holonomy ${\rm Hol}(A, z)$ is critical.

\item {\bf Finite energy.} $\left\| d_A\varphi \right\|_{L^2(C)}<\infty$.

Let $y\in \hat{C}$ be the preimage of an interior node or a boundary node of type 2 of $C$ in the normalization. We denote by $S_\epsilon\subset \hat{C}$ the circle of radius $\epsilon$ centered at $y$, and by $\iota_\epsilon: P_y\to P|_{S_\epsilon}$ be the natural isomorphism given by parallel transport with respect to $A$. Then the section $\varphi_\epsilon=\iota_\epsilon^*\varphi$ converges in $C^0$-topology, as $\epsilon\to 0$, to a section $\varphi_y$ of $P_y\times_{S^1} X$ satisfying $d_{A_y} \varphi_y=0$. We call $(P_y, A_y, \varphi_y)$ the {\bf limit triple} of $(P, A, \varphi)$ at $y$.

\item {\bf Stabilizers of monotone chains of gradient segments.}

$\bullet$ Take any interior connecting node $w\in {\bf w}$, let $y, y'$ be its preimage in the normalization, and let ${\mc T}_y: P_y\to {\mc T}(X)$, ${\mc T}_{y'}: P_{y'}\to {\mc T}(X)$ the equivariant maps with values in the set of chains of gradient segments. Then ${\mc T}_y$ and ${\mc T}_{y'}$ are covariantly constant, i.e., take ${\mc T}_y$ for example, if we identify $S_y$ with $S^1$ and trivialize $P_y\to S_y$ such that $d_{A_y}=d+\lambda d\theta$, then the corresponding map ${\mc T}_y: S_y\to {\mc T}(X)$ satisfies ${\mc T}_{y}(\theta)=e^{\lambda\theta} {\mc T}_y(0)$. Here $\lambda$ is the residue at $y$ with respect to the trivialization. This implies that the images of ${\mc T}_y$ and ${\mc T}_{y'}$ are contained in $X^\lambda$. Another consequence is that if $\lambda$ is not critical, then the maps ${\mc T}_y$, ${\mc T}_{y'}$ are degenerate.

$\bullet$ For any boundary node of type 2 $w''\in {\bf w}^{b2}$, let ${\mc T}_{w''}: P_{w''}\to {\mc T}(X)$. Then ${\mc T}_{w''}$ is covariantly constant.

\item  {\bf Matching condition and monotonicity at interior connecting nodes.}

Let $w\in C^{st}$ be an interior node. Pick an ordered list of connecting exceptional points over $w$, $(y_0, \ldots, y_{d-1})$, $(y_1', \ldots, y_d')$. For each $j$, denote by $(P_j, A_j, \varphi_j)$ and $(P_j', A_j', \varphi_j')$ the limit triples of $(P, A, \varphi)$ at the points $y_j$ and $y_j'$ respectively. Let $\Phi_j: P_j\to X$ and $\Phi_j': P_j'\to X$ be the equivariant maps corresponding to $\varphi_j$ and $\varphi_j'$ respectively. Suppose that $y_{j-1}$ and $y_j'$ are the preimages of a connecting node $z_j\in C$. Then each ${\mc T}_{z_j}$ is a pair of maps ${\mc T}_{j-1}: P_j\to {\mc T}(X)$ and ${\mc T}_j': P_j\to {\mc T}(X)$. Recall that there are maps $b, e: {\mc T}(X)\to X$ given by the beginning and end of gradient segments. Finally, let $g_j: P_j\to P_{j+1}'$ be the isomorphism given by the gluing data at $z_j$. Then the matching condition requires that either for all $j$,  $\Phi_{j-1}=b\circ {\mc T}_j'\circ g_{j-1}$ and $\Phi_j'=e\circ {\mc T}_j'$ or the same equations hold with $b$ and $e$ switched for every $j$(see Figure \ref{figure1}).
\begin{figure}[htbp]
\centering
\includegraphics[scale=0.70]{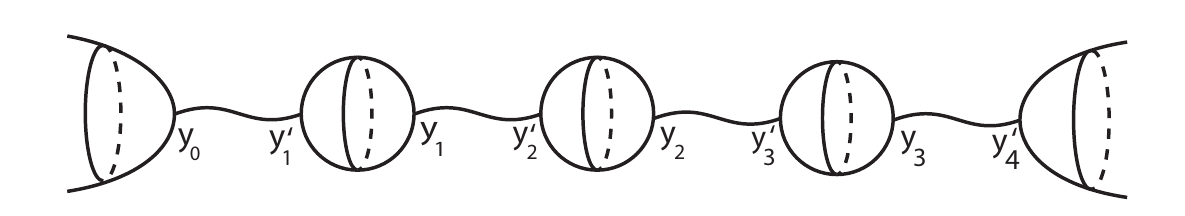}
\caption{ Chains of gradient segments and connecting bubbles at an interior node with $d=4$}
\label{figure1}
\end{figure}

\item {\bf Matching condition and monotonicity at boundary connecting nodes of type 2.}

Let $w''\in C^{st}$ be a boundary node of type 2. There are two possibilities on the connecting bubbles over $w''$ in $C$. The first is a sequence of bubbles $O_1, \ldots, O_d$ with $O_1, \ldots, O_{d-1}$ are sphere bubbles, and $O_d$ is a disk bubble. Then we have an ordered list of connecting exceptional points, $(y_0, \ldots, y_{d-1})$, $(y_1', \ldots, y_d')$. With the notations used in the above condition, we should also have $\Phi_{j-}=b\circ {\mc T}_j'\circ g_{j-1}$ and $\Phi_j'=e\circ {\mc T}_j'$ or the same equations with $b$ and $e$ switched. The second possibility is a sequence of sphere bubbles $O_1, \ldots, O_{d-1}$ with $O_{d-1}$ has a boundary node of type 2. Then using the above notation, we have either, for $1\leq j\leq d-1$, we have $\Phi_{j-1}=b\circ {\mc T}_j'\circ g_{j-1}$ and $\Phi_j'=e\circ {\mc T}_j'$, and $\Phi_{d-1}=b\circ {\mc T}_{d-1}$, with the {\bf boundary condition} $e\circ {\mc T}_{d-1}\in L$; or the same equations with $b$ and $e$ switched(see Figure \ref{figure2}).
\begin{figure}[htbp]
\centering
\includegraphics[scale=0.60]{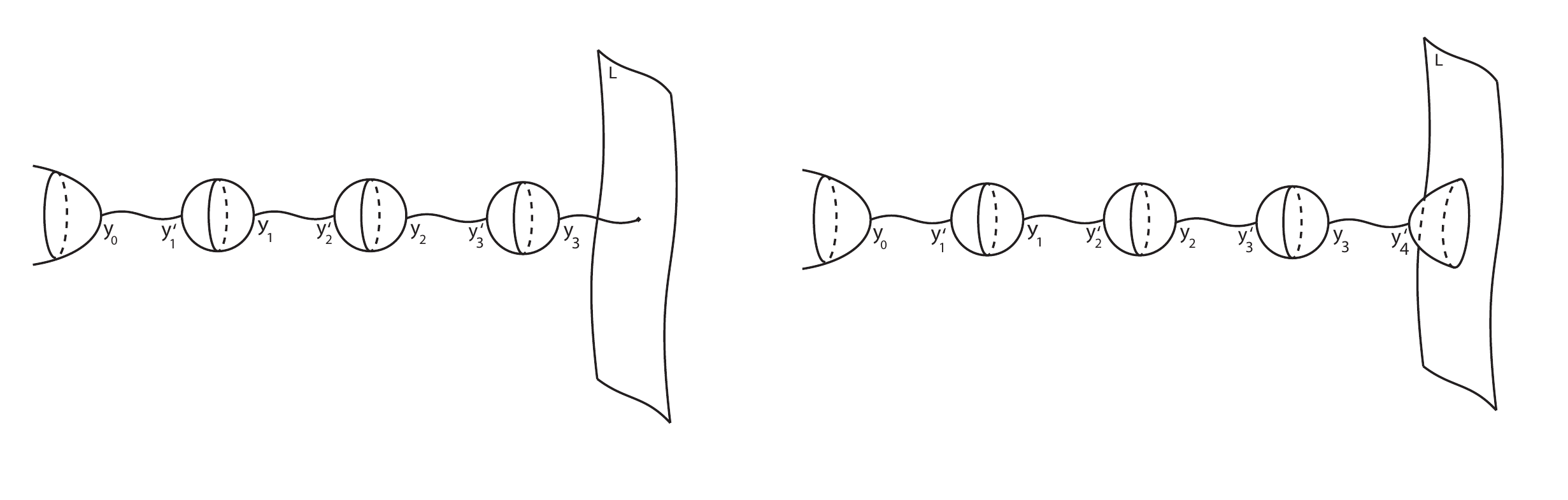}
\caption{ Chains of gradient segments and connecting bubbles at a type-2 boundary node in two cases respectively}
\label{figure2}
\end{figure}

\item {\bf Matching conditions at interior tree nodes.}

Let $w\in C$ be an interior tree node and let $y, y'\in \hat{C}$ be its preimages in the normalization. Let $(P_y, A_y, \phi_y)$ and $(P_{y'}, A_{y'}, \phi_{y'})$ be the limit triples, and let $g_w: P_y\to P_{y'}$ be the isomorphism of bundles given by the gluing data at $w$. Denote by $\Phi_y: P_y\to X$ and $\Phi_y': P_y'\to X$ be the corresponding equivariant maps respectively, then we have $\Phi_y=\Phi_y'\circ g_w$. This implies that the limit orbits of $\varphi$ at $y$ and $y'$ are the same.

\item {\bf Matching conditions at boundary nodes of type 1}

Let $w' \in C$ be a boundary node of type 1 with $y, y'$ the preimages in the normalization. Let $g_{w'}: P_y\to P_{y'}$ be isomorphism of fibres given by the gluing data at $w'$. Let $\Phi_y: P_y\to L$ and $\Phi_y': P_{y'}\to L$ be the corresponding equivariant maps. Then $\Phi_y=\Phi_{y'}\circ g_{w'} $.

\item {\bf Stability condition.}

If a bubble component $D\subset C$ is unstable, then the restriction of $d_A \varphi$ to $D$ is not identically zero.

\end{enumerate}

\begin{rem} The last stability condition in the definition doesn't automatically guarantee that the automorphism group of a $(c, L)$-STHM is finite. We have to put a restriction on the pair $(X, L)$ and the value $c$ to have that. However, this is not related to the compactness issue.
\end{rem}

\subsection{Isomorphisms of $(c, L)$-STHM's}

Take two $(c, L)$-STHM's $${\mc C}_i=\left( ( C_i, {\bf x}_i\cup {\bf z}_i), (P_i, A_i, \varphi_i), G_i, \{{\mc T}_{w, i}\}\right), i=1, 2.$$ We say that they are {\bf isomorphic}, if there is an isomorphism $f: (C_1, {\bf x}_1\cup {\bf z}_1)\to (C_2, {\bf x}_2\cup {\bf z}_2)$ of prestable curves and an isomorphism $g: P_1\to f^*P_2$ of $S^1$-principal bundles satisfying 
\begin{enumerate}

\item $g^*f^*A_2=A_1$, $g^*f^*\varphi_2=\varphi_1$ in the usual sense;

\item (see the definition of isomorphisms in \cite{Mundet_Tian_Draft})for each interior node $w_1\in {\bf w}^{int}\subset C_1$, let $w_2= f(w_1)$. Then we can lift $f$ to the normalization 
$$\hat{f}: \hat{C}_1\to \hat{C}_2.$$
Choose preimages $y_1, y_1'$ of $w_1$ and $y_2, y_2'$ of $w_2$ in their normalizations such that $\hat{f}(y_1)=y_2$, $\hat{f}(y_1')=y_2'$. Let $(P_{y_j}, A_{y_j})$(resp. $(P_{y_j'}, A_{y_j'})$) be the limit pairs of $(P_j, A_j)$ at $y_j$(resp. $y_j'$). Let $g_w: P_{y_1}\to P_{y_2}$ and $g_{w}': P_{y_1'}\to P_{y_2'}$ be the maps induced by $g$. The gluing data of ${\mc C}_j$ gives $\gamma_j: P_{y_j}\to P_{y_j'}$. 

$\bullet$ If the limit holonomy of $A_j$ at $y_j$ is {\it generic}, then we require that the following diagram commutes
$$\begin{CD}
P_{y_1} @>g_w>> P_{y_2}\\
@V\gamma_1VV  @VV \gamma_2V\\
P_{y_1'} @> g_w'>> P_{y_2'}.
\end{CD}$$

$\bullet$ If the holonomy of $A_j$ at $y_j$ is {\it critical}, then we require that there exists $t\in {\mb R}$ such that the following diagram commutes
$$
\begin{CD}
P_{y_1} @>g_w>> P_{y_2}\\
@V\gamma_1VV  @V V \gamma_2\circ \eta_{y_2}(t) V\\
P_{y_1'} @> g_w'>> P_{y_2'}.
\end{CD}$$
Here $P_{y_2}\simeq S_{y_2}\times S^1$ and $\eta_{y_2}: P_{y_2}\times {\mb R}\to P_{y_2}$ is the flow given by the limit connection on $P_{y_2}$.

\item For each boundary node of type 1 $w_1\in {\bf w}^{b1}$, let $w_2= f(w_1)$. We require that there exists $\gamma\in S^1$ in the stablizer of $\varphi_2(y_2)$ such that
$$\begin{CD}
P_{y_1} @>g_w>> P_{y_2}\\
@V\gamma_1VV  @V V \gamma_2\circ \gamma V\\
P_{y_1'} @> g_w'>> P_{y_2'}
\end{CD}
$$
commutes.

\item Using the above notations, for each pair of chains of gradient segments ${\mc T}_{y_1}$, ${\mc T}_{y_2}$, we require that
$${\mc T}_{y_1}= {\mc T}_{y_2}\circ g_w. $$

\end{enumerate}

\begin{rem}
As Ignasi Mundet told the author, this definition seems unnatural in the first glance, because the relaxation from the natural but more rigid definition allows us to have nice structure of moduli space.
\end{rem}

\subsection{Yang-Mills-Higgs functionals}

We give the definition of the Yang-Mills-Higgs functional for a $(c, L)$-STHM, which plays same role as energy in the theory of pseudoholomorphic curves.

\begin{defn}
Let ${\mc C}$ be a $(c, L)$-STHM. Its Yang-Mills-Higgs functional is defined to be
$$\mc{YMH}_c({\mc C}) :=\sum_{D\subset C} \left\| d_A\varphi\right\|_{L^2(D)}^2 + \sum_{D\subset C_p} \left( \left\|F_A \right\|_{L^2(D)}^2+ \left\| \mu(\varphi)-c\right\|_{L^2(D)}^2 \right),$$
where the first summation is taken over all components of $C$ while the second is taken over all principal components.
\end{defn}

\section{Bounding the number of bubbles}

In this subsection, we will prove the following theorem.

\begin{thm}\label{bubblebound}
Given $(g, h, n, \overrightarrow{m} )$ in the stable domain and $K>0$, there is a number $f(g, h, n, \overrightarrow{m}, K)$ such that for any $c$-stable twisted holomorphic maps with boundary ${\mc C}$ with the underlying curve $\left( C^{st}, {\bf x}^{st}\cup {\bf z}^{st} \right) \in \ov{\mc M}_{(g, h), (n, \overrightarrow{m})}$ and $\mc{YMH}_c({\mc C}) \leq K$, the number of bubbles of $C$ is no greater than $f(g, h, n, \overrightarrow{m}, K)$.
\end{thm}

In the case of pseudoholomorphic curves, the bounding of bubbles is relatively easy because the energy of nontrivial bubbles has a positive lower bound. But in the case of twisted holomorphic curves, we can have bubbles of arbitrarily small energy. Then to find a bound, we have to consider such bubbles of small energy.

\subsection{Canonical transports}\label{canonicaltransport}

\subsubsection{Canonical transports on twisted spheres}

Consider the points $x_0=[1: 0]$ and $x_1=[0:1]$ in the complex projective line ${\mb P}^1$, define the path $\gamma: [0,1]\to {\mb P}^1$ as $\gamma(t)=[1-t: t]$ and let $u_0=\gamma'(0)\in T_{x_0}{\mb P}^1$. Let $\rho: {\mb P}^1\to {\mb P}^1$ be the map defined by $\rho([x: y])=[\ov{y}: {\ov x}]$ and let $u_1=d\rho(u_0)$.

Actually, if we identify ${\mb P}^1$ with the standard 2-sphere, with $x_0$ the south pole and $x_1$ the north pole. Then $\gamma$ is a geodesic connecting $x_0$ and $x_1$ and $u_0$ is the unit tangent vector of this geodesic at $x_0$, $u_1$ is the unit tangent vector of this geodesic at $x_1$.

Suppose $B$ is a compact smooth curve of genus 0 and ${\bf y}=\{y_0, y_1\}\subset B$ be two different points. Given a nonzero tangent vector $v_0\in T_{y_0}B $, by the uniformization theorem, there exists a unique biholomorphism $f: {\mb P}^1\to B$ such that $f(x_i)=y_i$ and $df(u_0)=v_0$. We set $g_B' (v_0)=df(u_1)\in T_{y_1} B$. Then this defines a canonical map $g_B': T_{y_0}B\setminus \{0\}\to T_{y_1} B\setminus \{0\}$ and induces an isomorphism of $S^1$-torsors: $f_{y_0, y_1}: S_{y_0}\to \ov{S_{y_1}}$.

Now let $P$ be an principal $S^1$-bundle over $B\setminus{\bf y}$ and $A$ be a flat connection on $P$. Then the $A$-parallel transport along the curve $f\circ \gamma$ defines an isomorphism between the fibre of $P_{y_0}$ over the class of $v_0$ in $S_{y_0}$ and the fibre of $P_{y_1}$ over the class of $v_1$ in $S_{y_1}$. In this way we obtain an isomorhism of $S^1$ principal bundles $$g_{y_0, y_1}: P_{y_0}\to P_{y_1}$$ which lifts the map $f_{y_0, y_1}$.

Given a point $z\in B\setminus {\bf y}$ we choose a biholomorphism $f: {\mb P}^1\to B$ satisfying $f(x_i)=y_i$ and $z$ belongs to the image of $f\circ \gamma$. Then we set $f_{y_j}(z)=[df(u_j)]$. This gives a well-defined map $f_{y_j}: B\setminus {\bf y}\to S_{y_j}$ for $j=0, 1$. Finally, using the parallel transport along the curve $f\circ \gamma$ we obtain morphisms of principal bundles $g_{y_j}: P\to P_{y_j}$, $j=0, 1$, lifting the maps $f_{y_j}$.

We call the maps $g_{y_j}$ and $g_{y_0, y_1}$ the {\bf canonical transport } of $P$.

\subsubsection{Canonical transports on twisted disks}

Consider the unit disk ${\mb D}$ and $0\in {\mb D}$ the center and $1\in \partial {\mb D}$. Set $\gamma: [0,1]\to {\mb D}$ to be $\gamma(t)=t$ and $\gamma'(0)=u_0$. Then for any disk $D$ with an interior point $z_0$, for any $z\in D\setminus \{z_0\}$, there is a unique biholomorphism $f: {\mb D}\to D$ such that $f(0)=z_0$ and $z$ belongs to the image of $f\circ \gamma$. Set $g_{z_0}: D\setminus \{z_0\}\to S_{z_0}$ by $g_{z_0}(z)= [df(u_0)]$. If we have a principal $S^1$-bundle $P$ over $D\setminus \{z_0\}$ with a flat connection $A$ on $P$, then the parallel transport along $f\circ \gamma$ gives a morphism \begin{align}\label{canonicaltransport}
g_{z_0}: P\to P_{z_0}.
\end{align}
We call the map $g_{z_0}$ the {\bf canonical transport} of $P$. In particular, if with respect to some trivialization $A= d+ \lambda d\theta$. Then the trivialization gives $P\simeq {\mb D}^* \times S^1\simeq [0, \infty)\times S^1\times S^1$ and $P_{z_0}\simeq S^1\times S^1$. Then $g_{z_0}$ is just the projection $(t, \theta_1, \theta_2)\mapsto (\theta_1, \theta_2)$.

\subsection{Twisted bubbles}

Let ${\bf y}\subset S^2$ be a finite subset. We call a {\bf twisted sphere bubble} over $(S^2, {\bf y})$ a triple $(P, A, \varphi)$ consisting a principal $S^1$-bundle $P\to S^2\setminus {\bf y}$, a flat meromorphic connection $A$ on $P$ and a section $\varphi: S^2\setminus {\bf y}\to P\times_{S^1} X$ satisfying $\ov\partial_A \varphi=0$. Similarly, let ${\bf y}\subset D$ be a finite subset in the interior of the unit disk. We call a {\bf twisted disk bubble} over $(D, {\bf y})$ a triple $(P, A, \varphi)$ consisting a principal $S^1$-bundle $P\to D\setminus {\bf y}$, a flat meromorphic connection $A$ on $P$ and a section $\varphi$ of $P\times_{S^1} X$ satisfying $\varphi(\partial D)\subset P\times_{S^1} L$ and $\ov\partial_A\varphi=0$. A twisted sphere (or disk) bubble $(P, A, \varphi)$ is {\bf trivial} if the covariant derivative $d_A \varphi$ is identically zero.

For any $\epsilon>0$ smaller than half the minimal distance $\delta_{cr}$ between two different elements in $\Lambda_{cr}$, define the sets $\Lambda^{\epsilon,+}_{cr}$ and $\Lambda^{\epsilon, -}_{cr}$ as
$$\Lambda_{cr}^{\epsilon,\pm}=\Lambda_{cr}\pm (0, \epsilon)=\left\{ \lambda_{cr}\pm \delta| \lambda_{cr}\in \Lambda_{cr}, i\delta\in (0, \epsilon) \right\},$$
and let $\Lambda^\epsilon_{gen}=i{\mb R}\setminus \left( \Lambda^{\epsilon,+}_{cr}\cup \Lambda^{\epsilon,-}_{cr} \right) $.

The following is \cite[Theorem 6.2]{Mundet_Tian_2009}.
\begin{thm} \label{thm9.2}
Let $y_+, y_-\in S^2$ be two different points and let ${\bf y}=\{y_+, y_-\}$. Let $\epsilon$ be smaller than the $\epsilon_2$ given by Theorem \ref{thm24} and ${1\over 2} \delta_{cr}$. There exists some $\delta>0$ with the following property. Suppose that $(P, A, \varphi)$ is a twisted bubble on $(S^2, {\bf y})$ satisfying $\left\| d_A\varphi \right\|_{L^2}<\delta$. Take some trivialization of $P$. Then $\varphi$ corresponds to a map $\phi: S^2\setminus {\bf y}\to X$. Let $\lambda$ be the corresponding residue of $A$ at $y_+$.

{\bf Case A.} If $\lambda\in \Lambda_{cr}$ then $(P, A, \phi)$ is trivial;

{\bf Case B.} Otherwise, by Theorem 4.2 the following limits exist
$$\phi(y_\pm):=\lim_{z\to y_\pm} \phi(z)\in F.$$
Let $F_\pm$ be the connected components of $F$ containing $\phi(y_\pm)$. Then we have
\begin{enumerate}
\item if $\lambda\in \Lambda_{cr}^{\epsilon,+}$ then $h(F_+)\leq h(F_0)$ with equality only if $(P, A, \varphi)$ is trivial;

\item if $\lambda\in \Lambda_{cr}^{\epsilon,-}$ then $h(F_+)\geq h(F_0)$ with equality only if $(P, A, \varphi)$ is trivial;

\item if $\lambda\in \Lambda^\epsilon_{gen}$ then $(P, A, \varphi)$ is trivial.
\end{enumerate}
Let $(P_{\pm}, A_{\pm})$ be the limit pairs of $(P, A)$ at $y_\pm$. When Case B holds, there are covariantly constant maps ${\mc T}_{\pm}: P_\pm\to {\mc T}^\infty(X)$ such that for any $p\in P_{\pm}$ we have
$$
{\mc T}_\pm(p)=\left\{ \phi(y_+), \phi(y_-)    \right\}\cup \left\{ \phi(p)\left| q\in P, g_{y_\pm} (q)=p \right.     \right\}.$$
\end{thm}

We have an analogue of this theorem for twisted disk bubbles.
\begin{thm}\label{thm9.3}
Let $z_0\in D$ be an interior point. Let $\epsilon$ be smaller than the $\epsilon_4$ given in Theroem \ref{thm4.1} and ${1\over 2}\delta_{cr}$. Then there exists $\delta>0$ with the following property. Suppose that $(P, A, \varphi)$ is a twisted disk bubble on $(D, z_0)$ satisfying $\left\| d_A\varphi \right\|_{L^2}<\delta$. Take some local trivialization of $P$ on $D\setminus \{z_0\}$. Then $\varphi$ corresponds to some map $\phi: D\setminus \{z_0\}\to X$. Let $\lambda$ be the corresponding residue of $A$ at $z_0$.

{\bf Case A.} If $\lambda\in \Lambda_{cr}$ then $(P, A, \varphi)$ is trivial;

{\bf Case B.} Otherwise, by Theorem 4.2 in \cite{Mundet_Tian_2009}, the following exists
$$\phi(z_0)=\lim_{z\to z_0} \phi(z)\in F_k.$$ Then we have
\begin{enumerate}
\item if $\lambda\in \Lambda^{\epsilon, +}_{cr}$, then $h(F_k)\leq h(L)$, with equality only if $(P, A, \varphi)$ is trivial;

\item if $\lambda\in \Lambda^{\epsilon, -}_{cr}$, then $h(F_k)\geq h(L)$, with equality only if $(P, A, \varphi)$ is trivial;

\item if $\lambda\in \Lambda^\epsilon_{gen}$, then $(P, A, \varphi)$ is trivial.
\end{enumerate}
Let $(P_{z_0}, A_{z_0})$ be the limit pair of $(P, A)$ at $z_0$. When case B holds, there is a covariantly constant flat map ${\mc T}_{z_0}: P_{z_0}\to {\mc T}^\infty(X)$ such that for any $p\in P_{z_0}$, we have
$${\mc T}_{z_0}(p)=\left\{ \phi(z_0)  \right\}\cup \left\{   \phi(q)| q\in P,\ g_{z_0}(q)=p  \right\},
$$
where $g_{z_0}: P\to P_{z_0}$ is the canonical transport (\ref{canonicaltransport}).
\end{thm}

\begin{proof} The proof is similar to that of Theorem \ref{thm9.2} given in \cite{Mundet_Tian_2009}. Indeed, suppose first that $\lambda=ip/q\in \Lambda_{cr}$ for relatively prime integers $p, q$. Take a covering $\pi: D\to D$ of degree $q$ totally ramified at $z_0$. Then $\pi^*A$ has trivial holonomy around $z_0$. Since $\pi^*A$ is flat, we can extend $\pi^*P$ and $\pi^*A$ to a bundle $P'\to D$ with a flat connection $A'$. We can take a trivialization $P\simeq D\times S^1$ with respect to which $A'$ is trivial. Consider the induced trivialization of $\pi^*P$, we can view the section $\pi^*\phi$ as an $I$-holomorphic map $\phi: D\setminus \{z_0\}\to X$ of bounded energy. By the removal of singularity theorem, $\phi$ extends to a holomorphic disk $\widetilde{\phi}: D\to X$ satisfying $\left\|d\widetilde\phi \right\|_{L^2}=q\left \|d_A\varphi\right\|_{L^2}$. This is a standard result in Gromov-Witten theory that there is a lower bound of energy of nontrivial holomorphic disks, which depends only on $\omega$ and $I$. Since the critical residues contains finite many classes modulo integers, we can take $\delta$ small enough so that the lemma holds for all critical $\lambda$.

Now consider the case that $\lambda$ is not critical. Fix a conformal isomorphism between $D-\{z_0\}$ and the semi-infinite cylinder $C^+=[0, +\infty)\times S^1$, and take the latter the standard coordinate $(t, \theta)$ in such a way that as $t$ goes to $+\infty$, $(t, \theta)$ approaches to the point $z_0$. Then we can identify $S_{z_0}\simeq S^1$ in such a way that the map $f_{z_0}\to S_{z_0}$ is the projection $(t, \theta)\mapsto \theta$.

Fix a triple $(P, A, \varphi)$ and take a trivialization $P\times C^+\times S^1$ with respect to which $d_A=d+\lambda d\theta$, since $A$ is flat.

Suppose $\lambda\in \Lambda_{cr}^{\epsilon, +}$, so there exists $\lambda_{cr}\in \Lambda_{cr}$ such that $i(\lambda-\lambda_{cr})\in(0, \epsilon)$. By a standard argument in Gromov-Witten theory(see \cite[Proposition B.4.9]{McDuff_Salamon_2004}), there exists $\delta>0$ such that $\left\| d_\alpha\phi \right\|_{L^2}<\delta$ implies that $\left\| d_\alpha\phi \right\|_{L^\infty}<\epsilon$. Then we can apply Theorem \ref{thm4.1} and Theorem \ref{thm4.2} to the pair $(\alpha, \phi)$. Hence there exists $\psi: {\mb R}_{\geq 0}\to X^{\lambda_{cr}}$ and $\phi_0: {\mb R}_{\geq 0}\times S^1\to TX$ such that
\begin{align*}
\phi(t, \theta)=e^{-\lambda_{cr}\theta}\exp_{\psi(t)}\left( e^{\lambda_{cr}\theta}\phi_0(t, \theta) \right).
\end{align*}
Take some $t\in {\mb R}_{\geq 0}$ and $N>t$. Apply Theorem \ref{thm4.1}, \ref{thm4.2} to $[0, N]\times S^1\subset C^+$ together with $d\alpha=0$ implies that
\begin{align*}
\left| \phi_0(t, \theta) \right| \leq K_4 e^{\sigma_4 (t-N)}, \ \left|\psi'(t)+i(\lambda-\lambda_{cr})\nabla h(\psi(t)) \right|\leq K_4 e^{\sigma_4 (t-N)}.
\end{align*}
Making $N\to \infty$, we deduce that $\phi_0(t, \theta)=0$, and $\psi'(t)=-i(\lambda-\lambda_{cr})\nabla h(\psi(t))$. Since $\phi(\{0\}\times S^1)\subset L$, we have $\psi(0)\in L$ and $\psi(t)$ is a downward gradient flow line of $h=- i\mu$. So we have $h(F_k)\leq h(L)$ and equality holds only if $\phi$ is trivial. The case when $\lambda\in \Lambda_{cr}^{\epsilon, -}$ is treated the same way.

Finally suppose $\lambda\in \Lambda_{gen}^\epsilon$. By Theorem \ref{thm40}, there exists real numbers $K', \sigma', \epsilon'$ such that its statements hold for any $\lambda\in \Lambda_{gen}^\epsilon$(since we have the continuous dependence of the constants on $\lambda$). Then we have the estimate
$$\left| d_\alpha \phi\right| \leq K_3 e^{\sigma_3(t-N)}\left\| d_\alpha \phi\right\|_{L^\infty}.$$
Let $N$ go to infinity, we see that the pair $(\alpha, \phi)$ is in fact trivial. 

The map ${\mc T}_{z_0}: P_{z_0}\to {\mc T}^\infty(X)$ is given by the trajectory $\psi(t)$ i.e. ${\mc T}_{z_0}(\theta_1, \theta_2)= e^{-\lambda_{cr} \theta_2} \psi(t)$. Here $\theta_1$ is the coordinate on the base $S_{z_0}$ and $\theta_2$ is the coordinate on the fibre of $P_{z_0}$. It is easy to check that $\psi$ coincides with the one given in the statement of the theorem. In particular, $\psi(0)\in L$.
\end{proof}

\subsection{Proof of Theorem \ref{bubblebound}}

Let $\Gamma$ be the dual graph of $C$, so the set of vertices of $\Gamma$ is the set of irreducible components of $C$, and for any pair of vertices $v', v''$ corresponding to components $C', C''$(it is possible that $v'=v''$), there are as many edges connecting $v'$ and $v''$ as there are many nodes in $C$ at which $C'$ and $C''$ meet. All subgraphs $\Gamma'\subset \Gamma$ which we shall mention will be saturated, that is, if $v'$ and $v''$ are vertices of $\Gamma'$, then all edges in $\Gamma$ connecting $v'$ and $v''$ are also contained in $\Gamma'$. We call a vertex of $\Gamma$ a tree (resp. connecting) vertex if it corresponds to a tree (resp. connecting) bubble.

Denote by $\Lambda_{cr}^\sigma$ the set of sums of finitely many elements in $\Lambda_{cr}$. This is a discrete subgroup of ${\mb R}$. Then we can take $\epsilon$ small such that $\Lambda_{cr}^\sigma\subset \Lambda_{cr}\cup \Lambda_{gen}^\epsilon$.

Each sphere(resp. disk) tree vertex of $\Gamma$ belongs to a maximal connected subgraph $T\subset \Gamma$ consisting of entirely of sphere(resp. disk) tree vertices, and $T$ is a tree. We denote the collection of all trees of sphere tree vertices $T_1^s,\ldots, T_p^s$ and the collection of all trees of disk tree vertices $T_1^d, \ldots, T_q^d$. Each tree has a unique vertex, called the {\bf root}, which is connected by an edge to the complement of this tree.

We slightly generalize \cite[Lemma 6.3]{Mundet_Tian_2009} as follows.

\begin{lemma} Let $T$ be one of the trees $T_1^s, \ldots, T_p^s$.
\begin{enumerate}
\item For any sphere bubble $B\subset C$ corresponding to a vertex in some $T_i^s$, the residue of $A$ at any node in $B$ belongs to $\Lambda_{cr}^\sigma$.

\item If moreover, the root of $T_s^i$ is connected by an edge to a connecting vertex or a disk tree vertex, then the limit holonomy of $A$ at any node in $B$ is trivial.
\end{enumerate}
\end{lemma}

\begin{proof}
The same as the proof of \cite[Lemma 6.3]{Mundet_Tian_2009}
\end{proof}

\begin{proof}[Proof of Theorem \ref{bubblebound}] We first show that, the total number of sphere tree vertices is bounded. Let $\left\{T_i^s \right\}_{i=1}^m$ be the set of trees in $\Gamma$ consisting of sphere bubbles. We say that a vertex $v\in T_i^s$ is {\bf stable} if its degree {\it as a vertex of $\Gamma$} ${\rm deg}_\Gamma (v)\geq 3$. Since $T_i^s$ is a tree, it has $|T_i^s|-1$ edges and has a unique edge connecting to its complement. Hence we have
$$\sum_{v\in T_i^s} {\rm deg}_\Gamma (v)= 1+ \sum_{v\in T_i^s} {\rm deg}_{T_i^s}(v)=1+ 2\left( |T_i^s|-1\right).$$
Hence
$$2|T_i^s|= \sum_{v\in T_i^s} {\rm deg}_\Gamma(v)+1\geq 3 \# \{v\in T_i^s\ {\rm stable}\}\Rightarrow \#\{v\in T_i^s\ {\rm unstable}\} \leq {1\over 3}|T_i^s|.$$
Summing up for $i=1, \ldots, m$, one obtains that 
$$\sum_{i=1}^m \#\{v\in T_i^s\ {\rm unstable}\}\leq {1\over 3}\sum_{i=1}^m |T_i^s|.$$
Among those unstable vertices, there are at most $|{\bf x}|$ contains interior marked points, while the other vertices really correspond to the unstable bubbles in the usual sense. The energy of each of those ``real'' unstable vertices is at least $\delta$, due to the lemma. Hence we have
$$\delta\left( {1\over 3} \sum_{i=1}^m |T_i^s|-|{\bf x}|\right) \leq K.$$
This gives a upper bounded of the number of sphere tree vertices, which is denoted by $N_1$ temporarily.

The total number of connecting sphere bubbles is also bounded, according to the proof of \cite[Theorem 6.1]{Mundet_Tian_2009}. We denote this number by $N_2$. Note that this step uses the monotonicity property of Theorem \ref{thm9.2}.

Now we bounded the number of disk bubbles. Let $T_j^d\subset \Gamma$ be a tree of disk tree bubbles. Then each vertex $v\in T_j^d$ has no interior marked point, or it will be stable. Also, if $v$ has an interior node $w$, then $v$ is the root of some $T_i^s$, so the residue at $w$ belongs to $\Lambda_{cr}^\sigma$. But $\Lambda_{cr}^\sigma\subset \Lambda_{cr}\cup \Lambda_{gen}^\epsilon$. Hence by the above theorem, each vertex of $T_j^d$ contributes at least $\delta$ to the energy. Hence the total number of disk tree vertices is bounded.

The only remaining bubbles are the disk connecting bubbles, each of which belongs to two cases, which are discribed in Section \ref{section65}. One is a single connecting bubble with an interior node, but there is at most one such bubble for each boundary node of type 2 of the stablization; the other kind is a chain of connecting bubbles, but it is easy to see that, the holonomy at each interior exceptional point belonging to this chain is trivial, hence each bubble in the sequence contributes at least $\delta$ to the energy. This gives an upper bounde to the total number of bubbles of this kind.
\end{proof}

\section{Definition of the convergence}

\subsection{Deformation of prestable curves}

Let $(C, {\bf x}\cup {\bf z})$ be a prestable curve and $h$ be a conformal metric. Let $\hat\pi: \hat{C} \to C$ be the normalization map and let ${\bf y}=\hat\pi^{-1}({\bf w})=\hat\pi^{-1}({\bf w}^{int})\cup \hat\pi^{-1}({\bf w}^{b1})\cup \hat\pi^{-1}({\bf w}^{b2})=: {\bf y}^{int}\cup {\bf y}^{b1}\cup {\bf y}^{b2}$. We have an anti-holomorphic involution $\tau$ on the complex double $C_{\mb C}$, which induces, for each point $y\in \partial \hat{C}$, an anti-holomorphic involution $\tau: T_y \hat{C}\to T_y \hat{C}$. Denote by $\left( T_y \hat{C}\right)^{\mb R}$ to be the subspace of all $\tau$-invariant vectors.

For each $y\in {\bf y}^{int}$, take a neighborhood of $U_y\subset \hat{C}$ such that there exists a biholomorphism $\xi_y: D(\epsilon)\to U_y$ with the disk $D(\epsilon)\subset T_y\hat{C}$ centered at $0$ of radius $\epsilon$; for each $y\in {\bf y}^{b1}$, take a neighborhood $U_y$ small enough such that there exists a biholomorphism $\xi_y: D(\epsilon)/\tau\to U_y$, where $D(\epsilon)\subset T_y \hat{C} $ and $\tau$ the involution on $T_y \hat{C}$; for each $y\in {\bf y}$ with $\hat\pi(y)$ a boundary node of type 2, take a neighborhood $U_y$ small enough such that there exists a biholomorphism $\xi_y: D(\epsilon)\to U_y$ where $D(\epsilon)\subset T_y \hat{C}$. For each $x\in {\bf x}$ and $z\in {\bf z}$, also choose a small neighborhood $B_x\subset \hat{C}$ of $\hat\pi^{-1}(x)$ and $B_z\subset \hat{C} $ of $\hat\pi^{-1}(z)$. Assume that these neighborhoods are pairwise disjoint.

Denote by $J$ the complex structure of $\hat{C}$ and $J'$ be another complex structure which coincides with $J$ on each $\bigcup \left( U_y\cup B_x\cup B_z\right)$. Such complex structures are called $(C, {\bf x}\cup {\bf z})$-admissible. Take, for each $w\in {\bf w}^{int}$ with preimages $y, y'\in \hat{C}$, an element $\delta_w\in T_y \hat{C} \otimes_{\mb C} T_{y'} \hat{C}$ with $|\delta_z|<\epsilon^2$; for each $w' \in {\bf w}^{b1}$ of type 1 with preimages $y, y'\in \hat{C} $, an element $\delta_{w'}\in (T_y \hat{C} \otimes_{\mb C} T_{y'} \hat{C} )^{\mb R}$ with $|\delta_{w'}|<\epsilon^2$; for each $w''\in {\bf w}^{b2}$, with preimage $y\in \hat{C} $, a real number $\delta_{w''}$ with $0\leq \delta_{w''}<\epsilon$. The collection of numbers ${\bf \delta}=\{\delta_w\}$ is called a set of {\bf smoothing parameters} of $C$.

Now from the original curve with an admissible complex structure $J'$ and a set of smoothing parameters $\delta$, we can construct a new curve $C(J', h, \delta)$ as follows: 
\begin{enumerate}
\item First replace the complex structure $J$ by $J'$ on the normalization $\hat{C}$;

\item For each interior node ${\bf w}^{int}$ with $\hat\pi^{-1}(w)=\{y, y'\}$, remove the sets $\xi_y\left( D({|\delta_w|\over \epsilon}) \right)$ and $\xi_{y'}\left( D({|\delta_w|\over \epsilon}) \right)$ from $\hat{C}$, and for each pair of elements $u\in D(\epsilon)\setminus D({|\delta_w| \over \epsilon})\subset T_y \hat{C} $ and $v\in D(\epsilon)\setminus D(|\delta_w|/\epsilon)\subset T_{y'} \hat{C}$ with $u\otimes v=\delta_w$, identify their images $\xi_y(u)$ and $\xi_{y'} (v)$; 

\item For each boundary node $w'\in {\bf w}^{b1}$ with $\hat\pi^{-1}(w')=\{y, y'\}$, remove the sets $\xi_y\left(D({|\delta_{w'}|\over \epsilon})/\tau \right)$ and $\xi_{y'}\left( D({|\delta_{w'} |\over \epsilon})/\tau \right)$ from $\hat{C} $, and for each pair of elements $[u]\in \left(D(\epsilon)\setminus D({|\delta_{w'}|\over \epsilon}) \right)/\tau\subset T_y \hat{C} /\tau$ and $[v]\in \left(D(\epsilon)\setminus D({|\delta_{w'} |\over \epsilon} ) \right)/\tau\subset T_{y'} \hat{C} /\tau$ identify the images $\xi_y([u])$ and $\xi_{y'}([v])$ if there are representatives $u\in T_y \hat{C} $ and $v\in T_{y'} \hat{C} $ such that $u\otimes v=\delta_{w'}$;

\item For each $w''\in {\bf w}^{b2}$, remove the open disk $\xi_{w''}(D(\delta_{w''}))$.
\end{enumerate}

Then for every node $w\in {\bf w}$, define the ``neck'' $N_w(\delta_w)\subset C(J', \delta)$ to be the following subset in different cases: 
\begin{enumerate}

\item if $w\in {\bf w}^{int}$, then $N_w(\delta_w)$ is the image by $\xi_y$ of the annulus $D(\epsilon)\setminus D\left( {|\delta_w|\over \epsilon}\right)$. $N_w(\delta_w)$ is conformal to the cylinder
$$N_w(\delta_w)\simeq [\log |\delta_w|-\log \epsilon, \log\epsilon] \times S^1.$$

\item if $w'\in {\bf w}^{b1}$, then $N_{w'}(\delta_{w'})$ is the image by $\xi_y$ of $\left(D(\epsilon)\setminus D({ |\delta_{w'}|\over \epsilon}) \right)/\tau$. $N_{w'}(\delta_{w'})$ is conformal to the strip
$$N_{w'}(\delta_{w'})\simeq [\log| \delta_{w'}| -\log\epsilon, \log\epsilon] \times [0, 1].$$

\item if $w''\in {\bf w}^{b2}$, then $N_{w''}(\delta_{w''})$ is the annulus $\xi_{w''}(D(\epsilon) \setminus D(\delta_{w''}))$, which is conformal to
$$N_{w''} (\delta_{w''})\simeq [ \log| \delta_{w''}| , \log \epsilon]\times S^1.$$

\end{enumerate}

A set of smoothing parameters $\{\delta_w\}$ is {\it small enough} if each $|\delta_w|$ is small enough. Given a compact subset $K\subset C-{\bf w}$, if the smoothing parameters $\{\delta_w\}$ is small enough so that $\xi_y(D(|\delta_w|/\epsilon))$ and $\xi_{y'}(D(|\delta_w|/\epsilon))$ are disjoint from $K$, then there is a canonical inclusion $\iota_K: K\to C(J', \delta)$. In particular, there is a canonical inclusion of the marked points ${\bf x}\cup {\bf z}$ inside any small deformation $C(J', \delta)$.

\subsection{Definition of the convergence}\label{defnofconvergence}

Let the following be a sequence of $(c, L)$-stable twisted holomorphic maps
$$\left\{{\mc C}_u=\left((C_u, {\bf x}_u\cup {\bf z}_u), (P_u, A_u, \phi_u), G_u, \{{\mc T}_{u, w}\}\right)\right\}_{u\in {\mb N}}.$$
Let ${\mc C}=\left((C, {\bf x}\cup {\bf z}), (P, A, \phi), G, \{{\mc T}_w\}\right)$ be a $(c, L)$-stable twisted holomorphic map. We say that the isomorphism classes $[{\mc C}_u]$ converge to $[{\mc C}]$, if one can write ${\mb N}=U_1\cup \cdots U_n$ in such a way that for each $j$, the sequence $\{[C_{u'}]\}_{u'\in U_j}$ converges strictly to $[{\mc C}]$. While a sequence $\{[{\mc C}_u]\}$ {\bf converges strictly} to $[{\mc C}]$ means the following:

$\bullet$ There is a $(C, {\bf x}\cup {\bf z})$-admissible complex structure $J_u$, smoothing parameters $\delta_u=\{\delta_{u, w}\}$ and a continuous map
$$\xi_u: C_u\to C(J_u, \delta_u)=C(J_u, h, \delta_u)$$
which can be written as $\xi_u=\xi_u'\circ \pi_u^{co}$ where $\pi_u^{co}: C_u\to C_u^0$ is a collapsing map which collapses some of the unstable connecting bubbles and $\xi_u': C_u^0\to C(J_u, \delta_u)$ is a biholomorphism of nodal curves.

$\bullet$ We assume that for each node $w\in C$, either all $\delta_{u, w}$ vanish or none of them vanishes(this is why we call this notion {\it strict} convergence). In the first case we say that $w$ is an {\bf old node} and in the second case we say that $w$ is a {\bf new node}.

$\bullet$ We require that each each old interior node $w$ of $C$ and old boundary node $w''$ of type 2 of $C$, there are index sets $B_w$ and $B_{w''}$ {\it indepdent} of $u$, such that there are collections of unstable connecting bubbles $\{O_{u, b, w}\}_{b\in B_w}$ and $\{O_{u, b, w''}\}_{b\in B_{w''}}$ which are those bubbles that $\pi_u^{co}$ collapses. Also, $\pi_u^{co}$ collapses a disk at $w''$ for all $u$, or doesn't collapse a disk at $w''$ for any $u$.

$\bullet$ Given the previous data, the pair $(P_u, A_u)$ on $C_u$ naturally induces a pair, which we still denote by $(P_u, A_u)$, on the curve $C_u^0$. For any exhaustion $K_1\subset \cdots\subset K_l\subset\cdots$ of the smooth locus of $C\setminus {\bf z}$ by compact subsets, any choice of conformal metric $h$ on $C$, the following convergence properties are satisfied as $u\to \infty$. (For each $K_l$ in this sequence, the set $\iota_u(K_l)\subset C_u^0$ are identified for $u$ sufficiently large.)
\begin{enumerate}
\item {\bf Convergence of the marked curves.} The admissible complex structure $J_u$ converges to $J$ in the $C^\infty$-topology on each compact subset $K_l\subset C$; for each node $w$, $\lim_{u\to\infty} \delta_{u, w}=0$; the sequence of marked point sets ${\bf x}_u\cup {\bf z}_u$ converges to ${\bf x}\cup {\bf z}$.
	
\item {\bf Convergence of the connections and sections away from nodes and interior marked points.} For each $K_l$ and any big enough $u$(so that the inclusion $\iota_{K_l}: K_l\to C(J_u, \delta_u)$ is defined) there exists an isomorphism of principal bundles
\begin{align*}
\rho_{u, l}: P|_{K_l}\to \iota_{K_l}^*P_u
\end{align*}
such that $\rho_{u, l}^* \iota_{K_l}^*A_u$ converges to $A$ and $\rho_{u, l}^*\iota_{K_l}^*\phi_u$ converges to $\phi$ on $K_l$ in $C^\infty$. This implies that the limit holonomy of $A_u$ around each interior marked point converges to that of $A$. We can assume that if $l<l'$, then
\begin{align*}
\rho_{u, l}=\rho_{u, l'}|_{K_l}
\end{align*}
and we will write $\rho_u$ instead of $\rho_{u, l}$ in the following. We can also assume that for an old boundary node $w$ of type 1, with preimages $y, y'$ in $\hat{C}$, the isomorphisms $\rho_{u, l}$ can be extended over $y$ and $y'$, for any $u$.

\item {\bf Convergence near the interior marked points.} For each interior marked point $z_u\in {\bf z}_u$ with $\lim z_u\to z\in {\bf z}$, take $D(z_u, \epsilon)^*\subset C_u$ the punctured disk centered at $z_u$ of radius $\epsilon$ and $D(z, \epsilon)^*$, identify biholomorphically $D(z_u, \epsilon)^*\simeq {\mb R}_{\geq 0}\times S^1\simeq D(z, \epsilon)^*$. For each $\tau>0$ define $N^\tau_z=[\tau, \infty)\times S^1$, then we have
$$\lim_{\tau\to \infty}\left(\limsup_{u\to \infty} {\rm diam}_{S^1}\phi_u(N_z^\tau)\right)=0.$$

\item {\bf Convergence near the new interior nodes.} Take some new interior node $w\in {\bf w}^{int}$, define for each $\tau>0$ and big enough $u$ the truncation of the neck $N_w(\delta_{u, w})$
$$N_u^\tau(w)=\left[ \ln|\delta_{u, w}|-\ln\epsilon+\tau, \ln\epsilon-\tau \right]\times S^1\subset N_w(\delta_{u, w})= \left[ \ln|\delta_{u, w}|-\ln\epsilon, \ln\epsilon \right]\times S^1.$$

-{\it Limit behavior of the gluing data.} Let $y, y'\in \hat{C}$ be the preimages of $z$ in the normalization of $C$ and for any small enough $\epsilon$, let $S_{y, \epsilon}, S_{y', \epsilon}\subset \hat{C} $ be the circles of radius $\epsilon$ centered at $y, y'$, respectively. Let $P_{y, \epsilon}$(resp. $P_{y', \epsilon}$) be the restriction of $P$ to $S_{y, \epsilon}$(resp. $S_{y', \epsilon}$), and let $\iota_{\epsilon}: P_y\to P_{y, \epsilon}$ and $\iota_{\epsilon}': P_{y'}\to P_{y', \epsilon}$ be the isomorphisms given by radial parallel transport. Take on $S_{y, \epsilon}$ the structure of $S^1$-torsor making $\exp_{y}^{-1}: S_{y, \epsilon}\to S_y$ an isomorphism of torsors and do the same for $S_{y', \epsilon}$. We can assume that for big enough $u$ there are $\tau_{u, \epsilon}, \tau'_{u, \epsilon}\in {\mb R}$ and identifications $S_{y, \epsilon}\simeq \{\tau_{u, \epsilon}\}\times S^1\subset N_u$ and $S_{y', \epsilon}\simeq \{\tau'_{u, \epsilon}\}\times S^1\subset N_u$, the first being an isomorphism of $S^1$-torsors and the second an antiisomorphism. Using these identifications we define $P_{u, y, \epsilon}$ as the restriction of $P_u$ to $S_{y, \epsilon}$ and $P_{u, y', \epsilon}$ as the restriction of $P_u$ to $S_{y', \epsilon}$. Parallel transport along lines $[\tau_{u, \epsilon}, \tau'_{u, \epsilon}]\times \{\theta\}\subset N_u$ using the connection $A_u$ gives an isomorphism of bundles $g_{u, z ,\epsilon}: P_{u, y, \epsilon}\to P_{u, y', \epsilon}$. Hence we have a diagram which may NOT be commutative:
$$ \begin{CD}
P_y @> \iota_{\epsilon}>> P_{y,\epsilon} @> \rho_u >> P_{u, y, \epsilon} \\
@V g_z VV @. @VV g_{u, z, \epsilon}V\\
P_{y'} @> \iota_\epsilon' >> P_{y', \epsilon} @> \rho_u >> P_{u, y', \epsilon},
\end{CD} $$
where $g_z$ is the gluing data at $z$ given by $G$, and we require that the diagram becomes commutative in the limit, in the following sense. Define $\gamma_{u, z, \epsilon}: P_y\to P_{y'}$ as the composition $(\iota_\epsilon')^{-1}\circ \rho_u^{-1} \circ g_{u, z, \epsilon}\circ \rho_u \circ \iota_\epsilon$ and taking any distance function $d$ on ${\rm Iso}(P_y, P_{y'})$. Then there exists some $t\in {\mb R}$ such that 
$$\lim_{\epsilon\to 0}\left(\limsup_{u\to\infty} d\left( g_z\circ \gamma(t), \gamma_{u, z, \epsilon} \right)\right)=0.$$
Recall that $\gamma: P_y\times {\mb R}\to P_y$ is the flow given by the limit connection on $P_y$\footnote{Actually in proving the compactness, we can construct the limit curve such that the above holds for $t=0$, i.e. $\gamma(0)={\rm id}$.}.

-{\it Limit behavior of the section.} Any trivialization of $P_u$ over $N_u$ induces a trivialization of $P_{u, y, \epsilon}$, and composing with $\rho_u\circ \iota_\epsilon$ we obtain a trivialization $P_y\simeq S_y\times S^1\simeq S^1\times S^1$. Thus the limit section $\phi_y: P_y\to X$(resp. $\phi_{y'}: P_{y'}\to X$) becomes a map $\varphi_y: S^1\to X$(resp. $\varphi_{y'}: S^1\to X$). Suppose the neck is $N_u=[p_u, q_u]\times S^1$ and we assume that $y$ is on the side of $p_u$. Then we require that for any $\theta\in S^1$,
$$\lim_{\tau\to\infty} \left(\limsup_{u\to\infty} \left(d(\phi_u(\inf S_{u^\tau}, \theta), \varphi_y(\theta))+d(\phi_u(\sup S_{u, \tau}, \theta), \varphi_{y'} (\theta))\right)\right)=0.$$
Moreover, if $z$ is a connecting node, and then we have a map ${\mc T}_y: P_y\to {\mc T}(X)$, which, under the above trivialization, becomes a map $T: S_1\to {\mc T}(X)$. Recall that $N_u^\tau$ is the truncation. Then we require that
$$\lim_{\tau\to\infty}\left(\limsup_{u\to\infty} d_{{\rm Hauss}}(\phi_u(S_u^\tau\times \{\theta\}), T(\theta))\right)=0.$$
On the other hand, if $z$ is a tree node, then we require that
$$\lim_{\tau\to\infty} \left(\limsup_{u\to\infty} {\rm diam}_{S^1} (\phi_u(N_u^\tau(w)))\right)=0.$$

-{\it Limit behavior of the energy density.} We require that
$$\lim_{\tau\to\infty} \limsup_{u\to\infty} \left\| d_{A_u}\phi_u \right\|_{L^\infty(N_u^\tau(w))}=0.$$

\item {\bf Convergence near the new boundary nodes of type 1.} Take such a node $w'\in C$, and define for each $\tau$ the truncation for big enough $u$
$$N_u^\tau(w')=\left[ \ln \left({|\delta_{u, w'}| \over \epsilon}\right)+\tau, \ln\epsilon-\tau \right]\times [0, \pi]\subset N_{w'}(\delta_{u, w'})=\left[ \ln\left( { |\delta_{u, w'}| \over \epsilon}\right), \ln\epsilon \right]\times [0, \pi].$$

-{\it Limit behavior of the gluing data.} Let $y, y'$ be the preimages of $w'$ in the normalization of $C$ and for any small enough $\epsilon$, let $C_{y, \epsilon}, C_{y', \epsilon}$ be the semi-circles of radius $\epsilon$ centered at $y, y'$ respectively. Let $P_{y, \epsilon}$(resp. $P_{y', \epsilon}$) be the restriction of $P$ to $C_{y, \epsilon}$(resp. $C_{y', \epsilon}$) and let $\iota_\epsilon: P_y\times [0, \pi]\to P_{y, \epsilon}$ and $\iota_\epsilon': P_{y'}\times [0, \pi]\to P_{y', \epsilon}$ be the isomorphisms given by radial parallel transport. We can assume that for big enough $u$ there are $\tau_{u, \epsilon}, \tau_{u, \epsilon}'\in {\mb R}$ and identifications $C_{y, \epsilon}\simeq \{\tau_{u, \epsilon}\}\times [0, \pi]\subset N_u$ and $C_{y', \epsilon}\simeq \{\tau_{u, \epsilon}'\}\times [0, \pi]$, such that
$$\begin{CD}[0, \pi] @> \exp_y >> C_{y, \epsilon} \to \{\tau_{u, \epsilon}\}\times [0, \pi]  @> p_2 >> [0, \pi] \end{CD}$$
is the identity map and the corresponding map for $y'$ coincides with $r: \theta\mapsto \pi-\theta$. We define $P_{u, y, \epsilon}$ as the restriction of $P_u$ to $C_{y, \epsilon}$ and define similarly $P_{u, y', \epsilon}$. Parallel transport along lines $[\tau_{u, \epsilon}, \tau'_{u, \epsilon}]\times [0, \pi]\subset N_u$ using the connection $A_u$ gives an isomorphism of bundles $g_{u, w', \epsilon}: P_{u, y, \epsilon}\to P_{u, y', \epsilon}$. Hence we have a diagram which may {\it not} be commutative:
$$\begin{CD}
P_y\times [0, \pi] @> \iota_\epsilon>> P_{y, \epsilon} @> \rho_u>> P_{u, y, \epsilon}\\
@V g_{w'}\times r VV @. @VV g_{u, w', \epsilon}V\\
P_{y'}\times [0, \pi] @> \iota_\epsilon' >> P_{y', \epsilon} @> \rho_u >> P_{u, y', \epsilon},
\end{CD}$$
where $g_{w'}$ is the gluing data at $w$ given by $G$. Define $\gamma_{u, w', \epsilon}(\theta): P_y\to P_{y'}$ to be the restriction of $(\iota_\epsilon')^{-1}\circ \rho_u^{-1}\circ g_{u, w', \epsilon}\circ \rho_u\circ \iota_\epsilon$ to $P_y\times \{\theta\}$ and take any distance function $d$ on ${\rm Iso}(P_y, P_{y'})$, we required that
$$\lim_{\epsilon\to 0} \left(\limsup_{u\to\infty} d(g_{w'}\circ \gamma, \gamma_{u, {w'}, \epsilon}(\theta))\right)=0$$
for all $\theta\in [0, \pi]$ and some $\gamma$ in the stablizer of $\varphi(w')\in L$.

-{\it Limit behavior of the section.} We require that
\begin{align}\label{914}
\lim_{\tau\to \infty}\left(\limsup_{u\to \infty} {\rm diam}_{S^1} \phi_u\left( N_u^{\tau}(w') \right)\right)=0.
\end{align}

\item {\bf Convergence near the new boundary nodes of type 2.}
If $w''$ is a boundary node of type 2 in $C$, then $N_{w''}(\delta_{u,w''})$ is conformal to the long cylinder $N_u=[0, 2\ln \epsilon-\ln |\delta_{w''} |]\times S^1=[0, r_u]\times S^1$. Then trivialize $P_u$ over $N_u$ with respect to which $A_u$ is in balance temporal gauge, and the pair $(A_u, \varphi_u)$ can be viewed as a pair $(\alpha_u, \phi_u)$ with $\alpha_u$ a 1-form and $\phi_u: N_u\to X$. The following must hold:

-{\it Limit behavior of the connection.} Choose appropriately the balanced temporal gauge for all $u$, the restriction of $\alpha_u$ to $\{0\}\times S^1\subset N_u$ is equal to $(\lambda+\lambda_u)d\theta$, for $\lambda$ independent of $u$ and $\lambda_u\to 0$. Also, the limit holonomy of $A$ around $w$ is equal to $e^{\pm 2\pi \lambda}$.

-{\it Limit behavior of the section.} The trivialization of $P_u$ on $N_u$ which puts $A_u$ in balanced temporal gauge, induces a trivialization of $P_w=S^1\times S^1$. Then the section ${\mc T}_w$ becomes a map $T: S^1\to {\mc T}(X)$. Denote $N_u^\tau=S_u^\tau\times S^1=[0, r_u-\tau]\times S^1$. We require that for any $\theta\in S^1$,
\begin{align*}
\lim_{\tau\to \infty}\left(\limsup_{u\to\infty} d_{Hauss}\left( \varphi_u(S_u^\tau\times \{\theta\}), T(\theta) \right) \right)=0,
\end{align*}
and either
\begin{align*}
\lim_{\tau\to\infty}\left(\limsup_{u\to\infty} d(\varphi_u(0, \theta), b(T(\theta)))+d(\varphi(r_u-\tau, \theta), e(T(\theta)))\right)=0
\end{align*}
or the same formula holds switching $b$ and $e$(in one case $i\lambda_u$ is positive for large $u$ and in the other case $i\lambda_u$ is negative for large $u$).

-{\it Limit behavior of the energy density.} We have
\begin{align*}
\lim_{\tau\to\infty}\limsup_{u\to\infty} \left\| d_{A_u} \phi_u \right\|_{L^\infty(N_u^\tau)}=0.
\end{align*}

\item {\bf Convergence near the old interior nodes.} We require for each old interior node $w\in C$ the same condition as in the condition (5) or (6)(depending on whether $w$ is a connecting or tree node) in the definition of convergence of $c$-STHM in \cite[Section 7.4]{Mundet_Tian_2009}.

\item {\bf Convergence near the old boundary nodes of type 1.} If a boundary node $w'\in C$ is of type 1, we require the following conditions to hold:

-{\it Limit behavior of the gluing data.} Since the isomorphism $\rho_{u, l}$ can be extended over $y$ and $y'$, which are the preimages of $w'$ in the normalization $\hat{C}$, the gluing data of $P_u$ over $w'$ can be viewed as an element of ${\rm Iso}(P_y, P_{y'})$. We require that the sequence converges to that of the gluing data of $P$ at $w'$ modulo the stablizer of $\varphi(y)$.

-{\it Limit behavior of the connection and the section.} Let $y$ and $y'$ be the preimages of $w$ in the normalization $C'$. Then they can be identified with the the preimages of $w$ in the normalization of $C_u$, for every $u$. Take $U_y$(resp. $U_{y'}$) a small neighborhood of $y$(resp. $y'$) in the normalization $C'$. Then it can be identified with a small neighborhood of $y$(resp. $y'$) in $C_u'$. We require that the connections and the sections converges on $U_y$ and $U_{y'}$ both in $C^\infty$-topology.

\item {\bf Convergence near the old boundary nodes of type 2.} If a boundary node $w''\in C$ is of type 2, then it must be a connecting node. Let $B_{w''}=\{b_1, \ldots, b_{d-1}\}$ be the index set of vanishing bubbles over $w''$. Let $O_j= O_{u, b_j}$ and assume that $O_{d-1}$ contains a boundary component of $C_u$. $O_{d-1}$ might be sphere with a boundary node of type 2, in which case we denote $y_j, y_j'$ to be the nodes on $O_j$ with $y_j$ on the side of boundary; or $O_{d-1}$ is a disk bubble, in which case we denote $y_j, y_j'$ to be the node on $O_j$ for $j< d-1$. Then we require the following conditions to hold(note that in this case we don't have requirement on the gluing data, and the following holds either $\pi_u^{co}$ collapses a disk at $w''$ or not):

-{\it The vanishing bubbles do vanish.} For each $b_j\in B_{w''}$ the energy $\left\| d_{A_u}\phi_u \right\|_{L^2(O_j)}$ converges to 0.

-{\it Limit behavior of the connections.} The covariant derivative $d_{A_u}$ can be written as $d+\alpha_u$ where $\alpha_u$ is in temporal gauge in $N_{y_0}$ and there is some $\lambda_u\in i{\mb R}$ such that $\lim_{t\to\infty} \alpha_u|_{N_{y_0}(t, \theta)}=\lambda_u d\theta$. And $\lambda_u\to \lambda$ as $u\to\infty$, where $\lambda$ is a logarithm of the limit holonomy of $A$ around $w$. Finally, if $B_z\neq \emptyset$, then $\lambda\in \Lambda^{cr}$ and for large $u$ we have $\lambda_u\notin\Lambda_{cr}$.

-{\it Limit behavior of the sections.} We distinguish two cases.

\begin{enumerate}
\item Suppose that $\lambda_u$ is generic for big enough $u$. Then there is a collection of equivariant and covariantly constant maps(with respect to the limit connection $A_{w''}$ on $P_{w''}$)
\begin{align*}
{\mc T}_{w'', 0}, \ldots, {\mc T}_{w'', d-1}: P_y\to {\mc T}(X),
\end{align*}
such that ${\mc T}_{w''}= \cup_{j=0}^{d-1} {\mc T}_{w'', j}$ and the following holds.

We have trivialization $t_{u, w''} : P_{u, w''}\to S^1\times S^1$ such that the diagram commutes
$$\begin{CD} P_u|_{N_{w''}}  @> \tau^{A_u}   >>   P_{u, w''}\\
@V  \tau_{u, w''}^P VV     @VV  t_{u, w''} V \\
N_{w''}\times S^1   @>>>   S^1\times S^1  \end{CD}$$
where $\tau_{u, w''}$ is the given trivialization of $P_u$ and the bottom map is the projection: $N_{w''}\times S^1\simeq R_{\geq 0}\times S^1\times S^1\to S^1\times S^1$. Under the trivialization $\tau_{u, w''}^P$, we can view $\varphi_u$ on $N_{w''}$ as a map $\phi_{u, w''}: N_{w''}\to X$. Now let
$$T_{u, w'', \epsilon}(\theta)={\mc T}_{w'', 0}\circ \iota_\epsilon\circ \rho_u^{-1} \circ \iota_{u, \epsilon} \circ t_{u, w''}^{-1}(\theta, 1)$$
be a map $T_{u, w'', \epsilon}: S^1\to {\mc T}(X)$. Here $\iota_{u, \epsilon}$ and $\iota_\epsilon$ are bundle maps defined by parallel transports, and $\rho_u: P\to P_u$ is the isomorphism on smooth locus. We require that for any $\theta\in S^1$, we have$$\lim_{\epsilon\to 0} \limsup_{\tau\to \infty} \left(   \limsup_u d_{Hauss} (\phi_{u, w''} ( S_u^\tau\times \{\theta\} ), T_{u, w'', \epsilon}( \theta))\right)=0.$$
On the other hand, the energy of vanishing bubbles converges to 0 hence each vanishing bubble shrink to a map ${\mc T}_j^{new} : P_{u, y_j}' \to {\mc T}(X)$ for $j=1, \ldots, d-1$. More precisely, we define the isomorphism of bundles
$$\gamma_{u, j, \epsilon}: = g_{u, j}\circ \cdots\circ g_{u,1}'\circ g_{u, 1} \circ \iota_{u, \epsilon}^{-1} \circ \rho_u \circ \iota_{\epsilon} : P_{w''}\to P_{u, j}'.$$
Here $g_{u, j}: P_{u, y_{j-1}}\to P_{u, y_j'}$ is the gluing data, and $g_{u, j}': P_{u, y_j'}\to P_{u, y_j}$ is the canonical transport on the bubble $O_{u, j}$ we defined in Section \ref{canonicaltransport}.
We require that for each $j$ and $p\in P_{w''}$,
$$\lim_{\epsilon\to 0} \left( \limsup_{u\to \infty} d\left(    {\mc T}^{new}_j ( \gamma_{u, j, \epsilon}(p)) , {\mc T}_{w'', j} (p)  \right)   \right)=0.$$

\item Suppose that $\lambda_u=\lambda\in \Lambda_{cr}$ for big enough $u$. Then there is no vanishing bubble. And we have equivariant maps ${\mc T}_{u, w''}: P_{w''}\to {\mc T}(X)$. We require first that
$$\lim_{\tau\to\infty}\left(\limsup_{u\to\infty} {\rm diam}_{S^1} \varphi_u(N_{w''}^\tau)\right)=0.$$
Let ${\mc T}_{w''}: P_{w''}\to {\mc T}(X)$ be the chain of gradient segments of the limit ${\mc C}$. We require that for any $p\in P_w$, we have
$$\lim_{\epsilon\to 0}\left(\limsup_{u\to \infty} d\left( {\mc T}_{u, w''}\circ \iota_{u, \epsilon}^{-1}\circ \rho_u\circ \iota_\epsilon^{-1}(p), {\mc T}_{w''}(p) \right) \right)=0.$$
Here $\iota_\epsilon$ is the isomorphism between $P_{w''}$ and $P_{w'',\epsilon}$ given by parallel transport with respect to $A$ and $\iota_{u, \epsilon}$ is the isomorphism between $P_{u, w''}$ and $P_{u, w'', \epsilon}$ given by parallel transport with respect to $A_u$. $\rho_u$ is the isomorphism between $P_{w'', \epsilon}$ and $P_{u, w'', \epsilon}$.
\end{enumerate}

\item {\bf Limit behavior of the energy density.} We have
$$\lim_{\tau\to\infty}\limsup_{u\to\infty} \left\| d_{A_u}\phi_u \right\|_{L^\infty(N_{w''}^\tau)}=0.$$
\end{enumerate}

\section{Compactness}

We will prove the following theorem in this section.
\begin{thm}
    Let $g, h, n$ be nonnegative integers and $\overrightarrow{m}$ be a nonnegative vector such that $(g, h, n, {\overrightarrow m})$ is in the stable domain. Let $K>0$ be any positive real number and let $c\in i{\mb R}$. Let $\{{\mc C}_u\}$ be a sequence of $(c, L)$-STHM's, of genus $g$, with $h$ boundary components and $(n, \overrightarrow{m})$-marked, satisfying $\mc{YMH}_c({\mc C}_u)\leq K$. Then there is a subsequence $\left\{ {\mc C}_{u_j} \right\}$ such that the sequence of isomorphism classes $\left\{ \left[ {\mc C}_{u_j} \right] \right\}$ converges to $[{\mc C}]$ for a $c$-stable twisted holomorphic maps with boundary. 
\end{thm}

The remaining of this section is devoted to its proof. The strategy of the proof is very much in parallel with that of the proof of \cite[Theorem 8.1]{Mundet_Tian_2009}. In the following, we will take subsequences of the original sequence of $(c, L)$-STHM's many but finite times, and still denote the subsequences by the same symbols. And the situations near interior marked points and interior nodes are covered in the proof of \cite[Theorem 8.1]{Mundet_Tian_2009}, so we omit many details related to this.

\subsection{Getting the first limit curve}\label{limitcurve}

Let $\left( C_u, {\bf x}_u\cup {\bf z} \right)$ be the prestable curve underlying ${\mc C}_u$. By Theorem \ref{bubblebound}, the number of bubbles of each $C_u$ is uniformly bounded so we can assume(by taking a subsequence) that all curves $(C_u, {\bf x}_u\cup {\bf z}_u)$ have the same combinatorial type(including the topological type of $C_u$ and the distribution of marked points into different components). 

We want to add new marked points(as few as possible, in some sense) such that the prestable curve becomes stable. Let ${\bf x}_u^0\subset C_u$ be a list of boundary points such that: 1) ${\bf x}_u^0$ is disjoint from boundary nodes and ${\bf x}_u$; 2) every component of $(C_u, {\bf x}_u\cup {\bf x}_u^0\cup {\bf z}_u)$, which has at least a circle boundary component, is stable; 3)  ${\bf x}_u^0$ has as few elements as possible, for each $u$. Then let ${\bf z}_u^0\subset C_u$ be a list of interior points such that: 1) ${\bf z}_u^0$ is disjoint from interior nodes and ${\bf z}_u$; 2) $(C_u, {\bf x}_u\cup {\bf x}_u^0\cup {\bf z}_u\cup {\bf z}_u^0)$ is stable; 3) ${\bf z}_u^0$ has as few elements as possible, for each $u$. Let ${\bf x}_u'={\bf x}_u\cup {\bf x}_u^0$ and ${\bf z}_u'={\bf z}_u\cup {\bf z}_u^0$. The each element in the new sequence $\left( C_u, {\bf x}_u'\cup {\bf z}_u'\right)$ has the same combinatorial type.

Then by taking a subsequence, we can assume that the isomorphism class of $\left( C_u, {\bf x}_u'\cup {\bf z}_u' \right)$ converges to the isomorphism class of a stable curve $\left( C', {\bf x}'\cup {\bf z}' \right)$ in $\ov{\mc M}_{g, h, n, \ora{m}}$ and we can assume that the set ${\bf x}'$(resp. ${\bf z}'$) contains the limit ${\bf x}$(resp. ${\bf z}$) of the sequence ${\bf x}_u$(resp. ${\bf z}_u)$. We call points in ${\bf x}$(resp. ${\bf z}$) {\bf original boundary(resp. interior) marked points}. In the following, the lists of marked points ${\bf x}_u'$ and ${\bf z}_u'$ may increase, and the limit curve $C'$ may change accordingly.

\subsection{Adding tree bubbles, the first step}\label{addtree:1}

We will proceed as in \cite[Section 8.3]{Mundet_Tian_2009}. Take an exhaustion $K_1\subset K_2\subset\cdots$ of the smooth locus of $C'\setminus {\bf x}\cup {\bf z}$ by compact subsets. And for each $K= K_l$ in the sequence, since $\left( C_u, {\bf x}_u'\cup {\bf z}_u' \right)$ converges to $(C', {\bf x}'\cup {\bf z}')$, we can assume that there is a canonical inclusion $$\iota_{K, u}: K \to C_u$$ with image $K_{u}$. We will do the following operations inductively from $l=1$.

\begin{itemize}

\item If $\sup_{K_u} \left| d_{A_u} \varphi_u \right|$ is bounded as $u\to \infty$, then we do nothing. We call this situation that {\bf no tree bubble appears in $K_l$}.

\item If $\sup_{K_u} \left| d_{A_u} \varphi_u \right|=s_u$ is not bounded, then pick for each $u$ a point $z_u^1\in K_u$ such that $\left| d_{A_u}\varphi_u \right|$ attains its maximum at $z_u^1$. By passing to a subsequence, there are two possibilities, according to which we will proceed differently.

{\bf Case A.} $\lim_{u\to\infty} z_u^1$ lies in the interior of $C'$. Then pick an $z_u^2\in C_u$ at distance $s_u^{-1}$ from $z_u^1$. Set ${\bf z}_u''={\bf z}_u'\cup \{z_u^1,z_u^2\}$ and ${\bf x}_u''={\bf x}_u'$.

{\bf Case B.} $\lim_{u\to\infty}z_u^1=x^1\in \partial C'$. Then take a holomorphic local coordinate $f: U_{x^1}\to D(\epsilon)\subset {\mb H}$ such that $f(x^1)=0$. There are two subcases(the division into the following two cases is independent of the choice of the local coordinate): a). $\lim_{u\to\infty}\left(  s^u {\rm Im} z_u^1 \right) =\infty.$ Then take $z_u^2$ in the interior of $C_u$ with distance $s_u^{-1}$ from $z_u^1$. Take ${\bf z}_u''={\bf z}_u'\cup \{z_u^1, z_u^2\}$ and ${\bf x}_u''={\bf x}_u'$; b). $\lim_{u\to\infty}\left( s^u {\rm Im} z_u^1 \right) <\infty$. Take $x_u^1=f^{-1}\left({\rm Re} f(z_u^1) \right)$ and $x_u^2\in \partial C_u$ with distance $s_u^{-1}$ from $x_u^1$. Take ${\bf x}_u''={\bf x}_u'\cup \{x_u^1, x_u^2\}$ and ${\bf z}_u''={\bf z}_u'$.

In any case, we get a new sequence of stable curves $(C_u, {\bf x}_u''\cup {\bf z}_u'')$. By taking a subsequence, we can assume that the sequence converges to a stable curve $(C'', {\bf x}''\cup {\bf z}'')$. In Case A, $C''$ is the union of $C'$ with a sphere bubble $B$ connected at a new interior node $z_0=\lim_{u\to\infty} z_u^1\in K$, the two point set $\{z_u^1, z_u^2\}$ converges to two new marked points on $B$; in Case B-a), $C''$ is the union of $C'$ with a sphere bubble $B$ and a disk bubble $D$, where $D$ is connected to $C'$ through $z_0=\lim_{u\to \infty} z_u^1\in K$ and $B$ is connected to $B$ through some $z_1$, the two point set $\{z_u^1, z_u^2\}$ converges to two new marked points on $B$; in Case B-b), $C''$ is the union of $C'$ with a disk bubble $D$, connected at $z_0=\lim_{u\to\infty} z_u^1\in K$, and the two point set $\{x_u^1, x_u^2\}$ converges to two boundary marked points on $D$.\footnote{There is a special case, i.e., if $z_u^1$ coincides with some newly added marked point subsection \ref{limitcurve}, or converges to some new marked points in \ref{limitcurve}. But we can choose the new marked points wisely to avoid this happen.}

In any case, we can prove as in \cite{Mundet_Tian_2009} that the triple induced on $B$(in Case A and Case B-a) or on $D$(in Case B-b)) has positive energy and hence at least $\epsilon_X$.

We call this situation that {\bf a nontrivial tree bubble appears in $K_l$}.
\end{itemize}

The induction process is, if no tree bubble appears in $K_l$, then we go to $K_{l+1}$; if a nontrivial tree bubble appears in $K_l$, then we replace $(C', {\bf x}_u'\cup {\bf z}_u')$ by $(C'', {\bf x}_u''\cup {\bf z}_u'')$, and modify $K_1, K_2, \ldots$, as follows, and start the process again from the modified $K_1$.

\begin{itemize}

\item Choose a descending sequence $\epsilon_i\to 0$. For any $y\in C'$, let $B(y, \epsilon_i)\subset C'$ be the open ball centered at $y$ of radius $\epsilon_i$.

\item In Case A, replace $K_i$ by $ \left( K_i\cup B  \right) \setminus B(z_0, \epsilon_i)$;

\item in Case B-a), replace $K_i$ by $\left( K_i\cup D\cup B\right) \setminus B(z_1, \epsilon_i)$;

\item in Case B-b), replace $K_i$ by $\left( K_i\cup D \right) \setminus B(z_0, \epsilon_i)$.
\end{itemize}

This induction process must stops at finite times, which means the following. We have added finitely many marked points on $(C_u', {\bf x}_u'\cup {\bf z}_u')$ in the same pattern for all $u$, such that the new sequence of stable marked curves has a subsequence(still denoted by the same symbols) converging to some stable curve $(C'', {\bf x}''\cup {\bf z}'')$, such that for any compact subset $K$ of the smooth locus of $C''\setminus {\bf x}''\cup {\bf z}''$, which is identified with $K_u\subset C_u'$, we have
$$
\limsup_{u\to \infty} \sup_{K_u} \left| d_{A_u} \varphi_u \right|<\infty.
$$
Or in other words, no tree bubble appears in $K_l$ for all $l$.

We replace $C'$ by $C''$, ${\bf x}_u'$ by ${\bf x}_u''$, ${\bf z}_u'$ by ${\bf z}_u''$, ${\bf x}'$ by ${\bf x}''$ and ${\bf z}'$ by ${\bf z}''$. We can still distinguish original marked points from the new marked points.

\subsection{Adding tree bubbles, the second step}\label{tree2}

For bubbling off at new interior nodes $w\in C'$, we can proceed as in \cite[Subsection 8.4]{Mundet_Tian_2009}. For the convenience, we give the basic process.

By definition of the convergence of stable curves, for each $u$, there is a $(C', {\bf x}'\cup {\bf z}')$-admissible complex structure $I_u$ and a set of smoothing parameters $\{\delta_u\}$, and an isomorphism $C_u\simeq C_u'=C'(J_u, \delta_u)$.

\subsubsection{Bubbling off at new interior nodes}

Take any new interior node $z\in C'$ with $\delta_u=\delta_{u, z}\neq 0$ for all $u$. Let $N_{u, z} \simeq \left[ \ln|\delta_{u, z}|-\ln\epsilon, \ln\epsilon \right]\times S^1=:\left[ p_u, q_u \right]\times S^1$ identified with the neck region in $C_u$. Denote by $(t, \theta)$ the standard cylindrical coordinate. For any $\tau>0$, let $N_{u, z}^\tau=\left[p_u+\tau, q_u-\tau \right]\times S^1$. 

If $$\limsup_{\tau\to\infty}\left(\limsup_{u\to\infty} \left\| d_{A_u}\varphi_u \right\|_{L^\infty( {N_{u,z}^\tau}) }\right)<\infty,$$
where the norm is taken with respect to the cylindrical metric, then we do nothing. 

Otherwise, we have
$$\limsup_{\tau\to\infty}\left(\limsup_{u\to\infty} \left\| d_{A_u}\varphi_u \right\|_{L^\infty( {N_{u,z}^\tau}) }\right)=\infty.$$ 
Then we can pick two points $z_u^1, z_u^2\in N_{u, z}$ such that $d\left( z_u^j, \partial N_{u, z} \right) \to\infty$, $s_u=\left| d_{\alpha_u}\varphi_u(z_u^1)\right|\to\infty$, $d(z_u^1, z_u^2)=s_u^{-1}$. And we assume that $s_u$ is equal to the supremum of $\left| d_{A_u}\phi_u \right|$ over the disk $B_1(z_u^1)$. We take ${\bf z}_u''={\bf z}_u'\cup \{z_u^1, z_u^2\}$ and take a subsequence of $(C_u, {\bf x}_u'\cup {\bf z}_u'')$ converging to a stable curve $(C'', {\bf x}'\cup {\bf z}'')$. Then $C''$ is the union of $C'$ with two bubbles, one connecting bubble and one tree bubble(as shown in Figure \ref{figure3}).
\begin{figure}[htbp]
\centering
\includegraphics[scale=0.60]{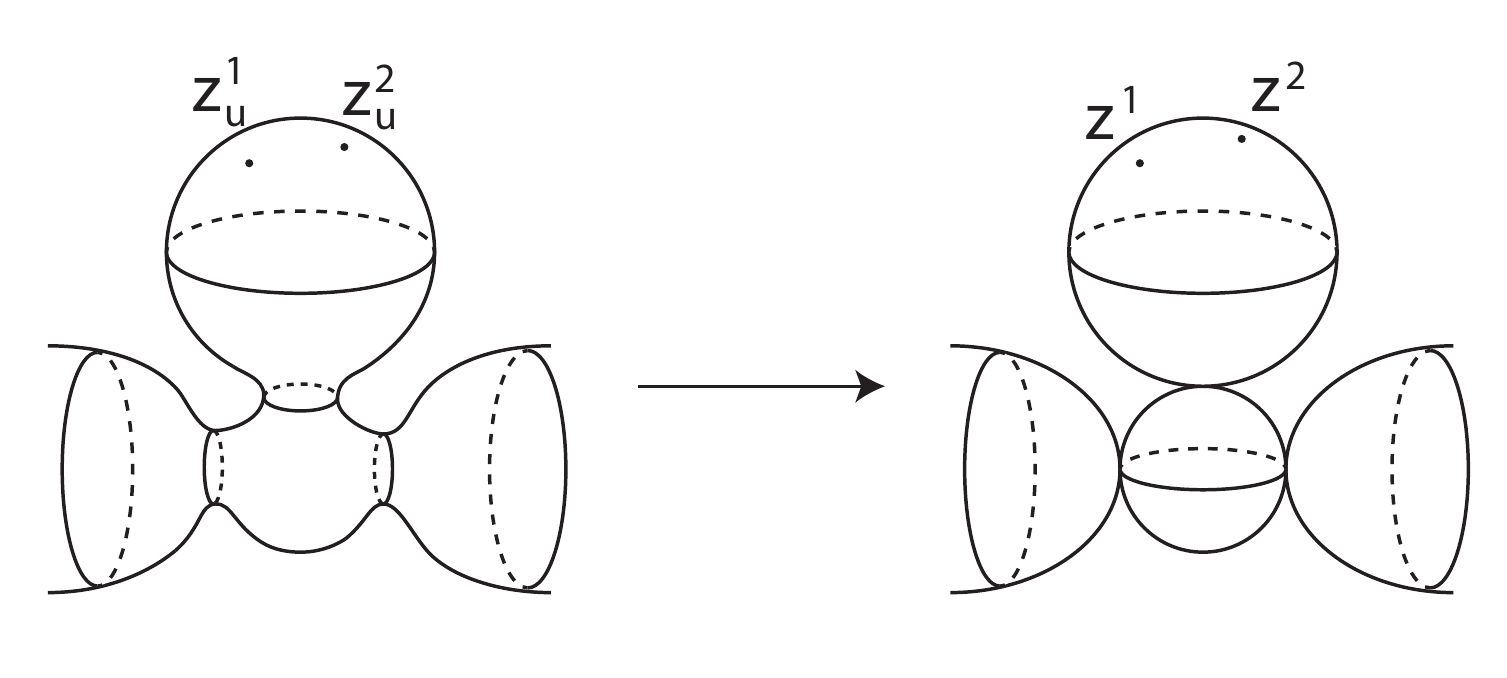}
\caption{Bubbling-off at a new interior node}
\label{figure3}
\end{figure}

Moreover, the limit triple on the tree bubble has positive energy hence $\geq \epsilon_X$. We can replace ${\bf z}_u'$ by ${\bf z}_u''$ and repeat the above process and it must stop at finite times of repeat.

\subsubsection{Bubbling off at old interior nodes and original interior marked points}

If $w\in C'$ is an old interior node. Then we consider the normalization for each $C_u$ near $w$ and apply the same process as before in a neighborhood of each of the two preimages of $w$. The only difference in this case is that instead of having finite cylinders $N_u$ for each $u$, we will have a semiinfinite cylinder $N_{u, w}= [\ln\epsilon, \infty)\times S^1$ and $N_{u, w}^\tau=[\ln\epsilon+\tau, \infty)\times S^1$. Finally, we apply the same technique around each original marked point.

\subsubsection{Bubbling off at new boundary nodes of type 1}

Now we consider boundary nodes. If $w'\in C'$ is a new boundary node of type 1. Then there is a long strip $N_{u, w'} =[\ln|\delta_{u, w'}|-\ln\epsilon, \ln\epsilon]\times [0, \pi]=:[p_u, q_u]\times [0, \pi]$ which can be identified with a subset of $C_u$ for each $u$. Let $N_{u}^\tau:=[p_u+\tau, q_u-\tau]\times [0, \pi]$. Suppose that
$$
\limsup_{\tau\to\infty}\left(\limsup_{u\to\infty} \left\|d_{A_u}\varphi_u\right\|_{L^\infty( N_{u}^\tau) }\right)=\infty
$$
where we also use the standard cylindrical metric. Then we can pick $z_u^1\in N_u$ such that $d(z_u^1, \{p_u, q_u\}\times [0, \pi])\to\infty$, $s_u=\left| d_{A_u}\varphi_u(z_u^1)\right|\to\infty$ and $s_u$ is equal to the supreme of $\left|d_{A_u}\varphi_u \right|$ over the disk centered at $z_u^1$ of radius 1. We distinguish the following cases(by taking subsequences):

\begin{enumerate}

\item $\lim_{u\to\infty} s_u d(z_u^1, \partial N_u)=\infty$. Then take $z_u^2$ lying in the interior of $N_u$ with distance $s_u^{-1}$ to $z_u^1$ and ${\bf z}_u''={\bf z}_u'\cup \{z_u^1 ,z_u^2\}$ and ${\bf x}_u''={\bf x}_u'$;

\item $\lim_{u\to\infty} s_u d(z_u^1, \partial N_u)<\infty$. Suppose $z_u^1=(a_u, \theta_u)$. We can assume that $\theta_u\to 0$($\theta_u\to \pi$ is similar). Then take $x_u^1=(a_u, 0)$ and $x_u^2=(a_u+s_u^{-1}, 0)$; set ${\bf x}_u''={\bf x}_u'\cup \{ x_u^1, x_u^2\}$ and ${\bf z}_u''={\bf z}_u'$.
\end{enumerate}

In any case, a subsequence of the curves $C_u$ with the new set of marked points will converge to some stable curve $C''$. In the first case, $C''$ containes a sphere tree bubble on which the limit of $\{z_u^1, z_u^2\}$ are attached. The situation is shown in Figure \ref{figure4}.

\begin{figure}[htbp]
\centering
\includegraphics[scale=0.60]{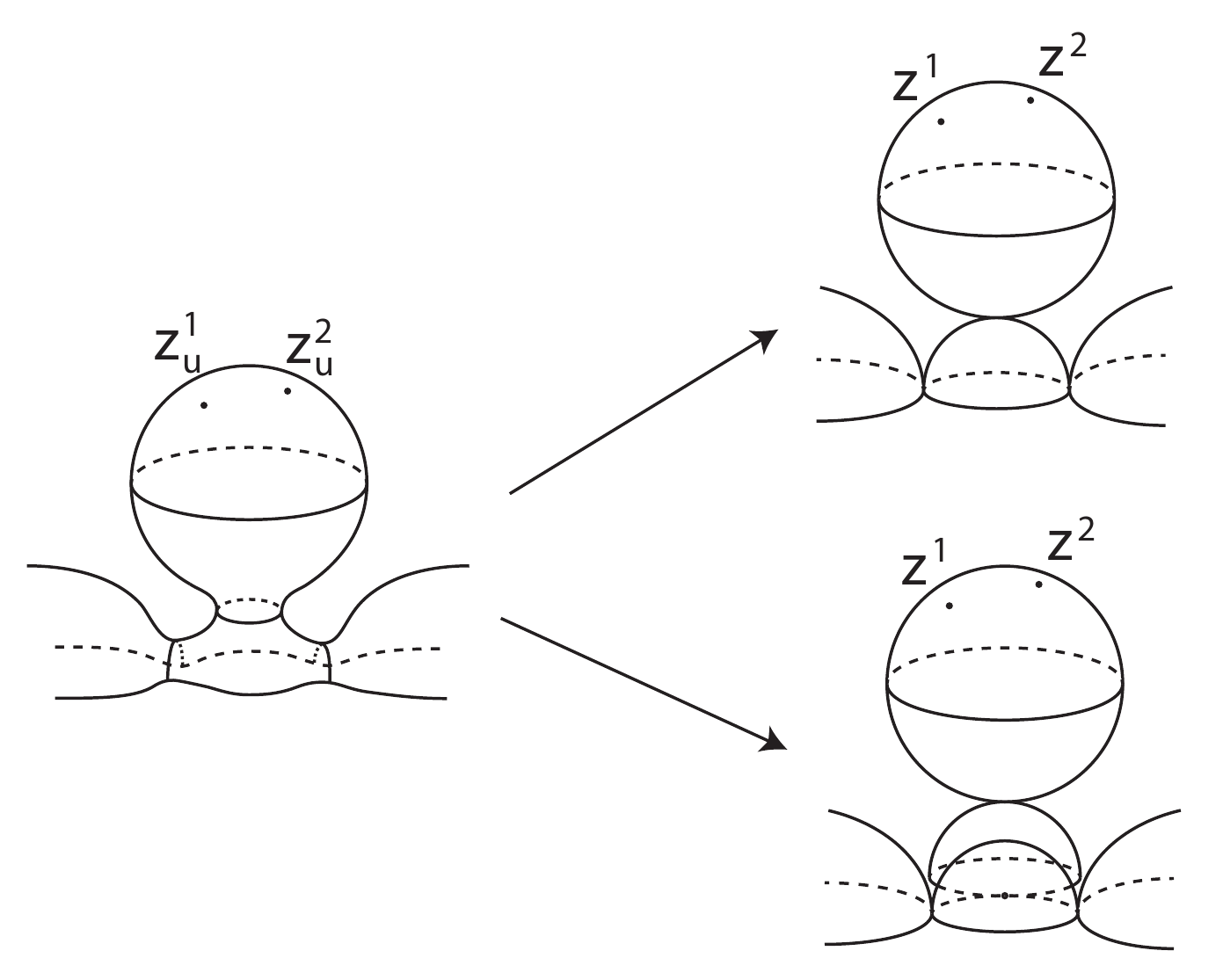}
\caption{Bubbling-off a sphere at a new type-1 boundary node}
\label{figure4}
\end{figure}

In the second case, $C''$ is the union of $C'$ with two disk bubbles, one of which is connecting and one of which is a tree bubble; the tree bubble has two boundary marked points which are limits of $x_u^1$ and $x_u^2$(see Figure \ref{figure5}).

\begin{figure}[htbp]
\centering
\includegraphics[scale=0.60]{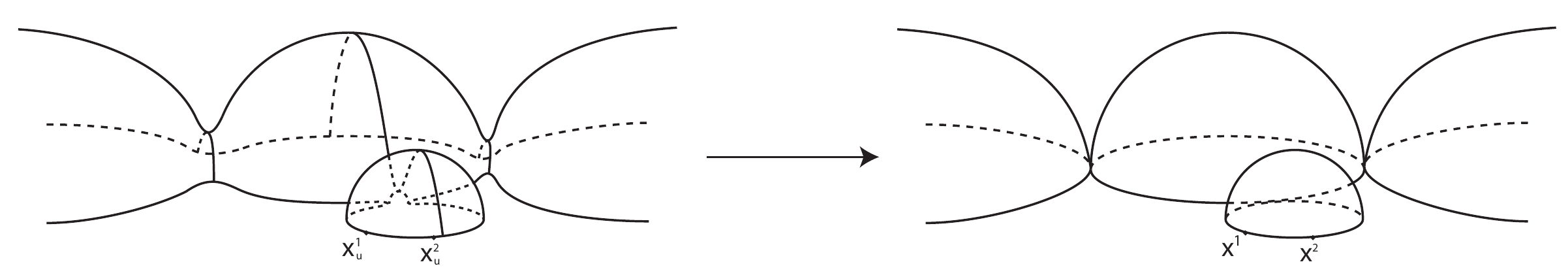}
\caption{Bubbling-off a disk at a new type-1 boundary node}
\label{figure5}
\end{figure}

Each of the tree bubble in the two cases contributes $\geq \epsilon_X$ to the total energy. So if we update the set of marked points as what we did before and repeat the process, then we must stop at finite times.

\subsubsection{Bubbling off at old boundary nodes of type 1 and original boundary marked points}

This case is much similar to the previous one. The only difference is that instead of having a strip of finite length, we will have a semiinfinite strip $[0, \infty)\times [0, \pi]$.

\subsubsection{Bubbling off at new boundary nodes of type 2}

Now consider new boundary nodes of type 2. Let $w''\in C'$ be such a node. Then the ``neck'' of $w''$ in $C_u$ is a cylinder conformal to $N_{u, w''}=[0, r_u]\times S^1$ with $r_u\to \infty$ and $S(u)= \{r_u\}\times S^1$ the boundary circle that shrink to $w''$. For each $\tau>0$, denote $N_u^\tau:= N_{u, w''}^\tau= [\tau, r_u]\times S^1$. Assume that $$
\limsup_{\tau\to\infty} \left(  \limsup_{u\to \infty} \sup_{N_u^\tau} \left| d_{A_u}\varphi_u\right| \right)=\infty.
$$

Then, we can take a sequence of points $z_u^1=(t_u, \theta_u)\in N_u$, such that $t_u\to \infty$ and $s_u= \left| d_{A_u} \varphi_u(z_u)\right| \to \infty$ as $u\to \infty$. By taking a subsequence, we consider the following two cases:
\begin{enumerate} 
\item If $\lim_{u\to \infty} s_u (r_u-t_u)=\infty$, then take another sequence of points $z_u^2\in N_u$ such that $d(z_u, z_u^2)= s_u^{-1}$, set ${\bf z}_u''= {\bf z}_u'\cup \{ z_u^1, z_u^2\}$ and ${\bf x}_u''= {\bf x}_u'$.

\item If $\lim_{u\to \infty} s_u (r_u -t_u)< \infty$, then take another sequence of points $x_u^1= (r_u, \theta_u), x_u^2= (r_u, \theta_u+ s_u^{-1}) \in  S(u)$. Set ${\bf z}_u''= {\bf z}_u'$ and ${\bf x}_u''= {\bf x}_u'\cup \{x_u^1, x_u^2\}$.
\end{enumerate}

Then the new sequence of curves $( C_u, {\bf x}_u'' \cup {\bf z}_u'')$ will subsequentially converge to some stable curve $(C', {\bf x}'' \cup {\bf z}'' )$: in the first case, $C'$ will contain a sphere bubble on which the limits of $z_u^1$ and $z_u^2$ are attached(see Figure \ref{figure6});in the second case, $C'$ will contain a disk bubble on which the limits of $x_u^1 $and $x_u^2$ are attached, and a disk connecting bubble(see Figure \ref{figure7}).
\begin{figure}[htbp]
\centering
\includegraphics[scale=0.50]{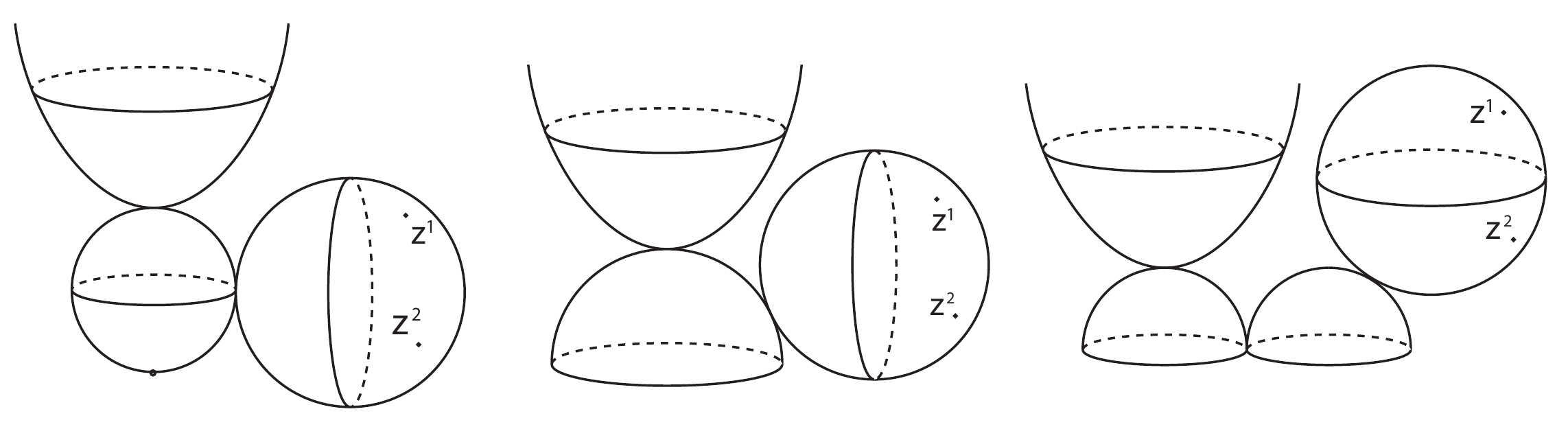}
\caption{Limit curves when bubbling-off a sphere at a new type-2 boundary node, three different cases}
\label{figure6}
\end{figure}
\begin{figure}[htbp]
\centering
\includegraphics[scale=0.50]{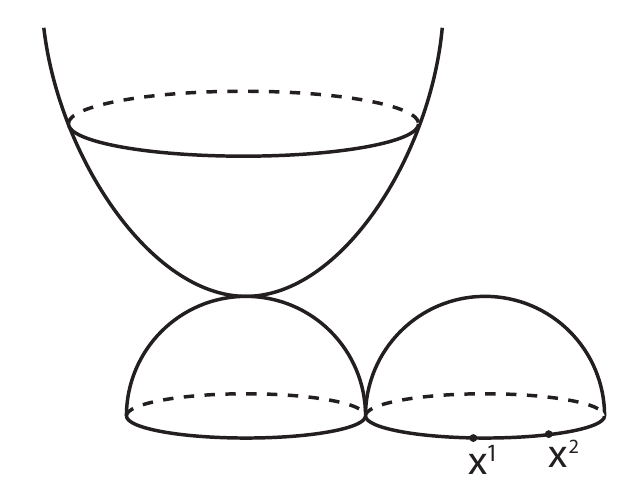}
\caption{Limit curve when bubbling-off a disk bubble at a new type-2 boundary node}
\label{figure7}
\end{figure}

\subsubsection{Bubbling off at old boundary nodes of type 2}

Let $w''\in C'$ be an old boundary node of type 2. Bubbling off at $w$ is the same as the bubbling off at original interior marked points.

\subsubsection{Tree bubbles, summary}

After all the steps in Section \ref{tree2}, we modify the exhausting sequence of compact subsets and repeat the process in Section \ref{addtree:1}. Then all these stop after finite times and we have found every possible tree bubble. This means the: by taking a subsequence and adding new marked points, we get a new sequence of stable curves $\left(C_u, {\bf x}_u\cup {\bf z}_u\right)$, which converges to $\left(C, {\bf x}\cup {\bf z}\right)$. We have
\begin{enumerate}

\item For each compact subset $K\subset C\setminus {\bf x}\cup {\bf z}\cup {\bf w}$, which is identified with a compact subset of each $C_u$, we have
$$
\limsup_{u\to\infty} \left\| d_{A_u} \varphi_u\right\|_{L^\infty(K_u)} <\infty.
$$
Here the norm is taken with respect to the conformal metric on $K$.

\item Let $\{K_l\}_{l\geq 1}$ exhaust $C\setminus {\bf x}\cup {\bf z} \cup {\bf w}$. The complement of $K_l$ in $C_u$ is a union of ``cylindrical-like'' pieces, which is denoted by $N_{u, l}$. Then we have that
$$
\limsup_{l\to\infty} \left( \limsup_{u\to\infty} \left\| d_{A_u} \varphi_u \right\|_{L^\infty(N_{u, l})}\right)<\infty.$$
Here the norm is taken with respect to the cylindrical metric.

\end{enumerate}

\subsection{Conneting bubbles appear}

The matching conditions at various types of nodes are fulfilled only when the energy densities at the neck regions are very small. The bubbling-off of tree bubbles cannot guarentee this condition so we have to find more connecting bubbles.

\subsubsection{At interior nodes}

For each new interior node $w\in C'$, we do the same operation as in \cite[Subsection 8.5]{Mundet_Tian_2009}. Let's just briefly explain what happens. 

Let $N_u$ be the neck associated to the new node $w$. Its conformal length goes to infinity as $u\to \infty$, and we use $d$ to denote the distance function for the cylindrical metric on $N_u$. For $\epsilon$ smaller than the $\epsilon_1$ in Theorem \ref{thm22} and $\epsilon_2$ in Theorem \ref{thm24}, we define an {\bf $\epsilon$-bubbling list} to be a sequence of lists $\left\{  (z_1^u, \ldots, z_r^u)\right\}_u$ which satisfies
\begin{enumerate}
\item $z_j^u\in N_u$;

\item for each $i\neq j$, $d(z_i^u, z_j^u)\to \infty$ and $d(z_i^u, \partial N_u)\to \infty$ as $u\to \infty$;

\item for each $j$, $\liminf_{u\to \infty} \left| d_{A_u} \varphi_u(z_j^u)\right|\geq \epsilon$.
\end{enumerate}
Then we can prove the following lemma:

\begin{lemma}\cite[Lemma 8.2]{Mundet_Tian_2009} There is a constant $\kappa_b>0$ depending only on $X, \omega, I$, such that if $\left\{ (z_1^u, \ldots, z_r^u)   \right\}_u$ is an $\epsilon$-bubbling list, then $r\leq \kappa_b \sup_u \mc{YMH}_c({\mc C}_u)$.
\end{lemma}

Then for each new interior node $w\in C'$, we take an $\epsilon$-bubbling list of maximal length. If the list is empty, then we do nothing. Otherwise, add to the marked point sets ${\bf z}_u'$ the points $z_1^u, \ldots, z_r^u$. Passing to a subsequence, we may assume that the new sequence of stable curves converges to some stable curve $(C'', {\bf x}''\cup {\bf z}'')$, by which we replace $(C', {\bf x}'\cup {\bf z}')$. 

We do a similar process for each old interior node, each old boundary node of type 2, and each original interior marked point of $C'$.

\subsubsection{At type-1 boundary nodes}

Now we consider new boundary nodes of type 1. Let $w\in C'$ be such a node. $N_u:=[p_u, q_u]\times [0, \pi]$ be the corresponding neck in $C_u$. We assume that $N_u$ is contained in a principal component of $(C_u, {\bf x}_u\cup {\bf z}_u)$. Let the volume form on $N_u$ be $f_u dt\wedge d\theta$. Then for any $l>0$, there is a constant $\kappa_l$ independent of $u$ such that
$$\left| \nabla^l f_u(z) \right|\leq \kappa_l e^{-\min\{t-p_u, q_u-t\}}.$$
Take any trivialization of $P_u$ over $N_u$, and write $d_{A_u}=d+\alpha_u$. Since $(\alpha_u, \phi_u)$ satisfies the vortex equation, we have
$$
d\alpha_u =f_u(c-\mu(\phi_u)) dt\wedge d\theta.$$

Let $\epsilon>0$ be smaller than the $\epsilon_5$ in Corollary \ref{cor58}. We define an {\bf $\epsilon$-bubbling list} to be a sequence of lists $\{(y_1^u, \ldots, y_r^u)\}_u$ satisfying
\begin{enumerate}

\item $y_j^u\in N_u$ and $\min\left\{ d(y_j^u, \{p_u\}\times [0, \pi]), d(x_j^u, \{q_u\}\times [0, \pi]) \right\}\to\infty$ as $u\to\infty$;

\item for each $i\neq j$, $d(y_i^u, y_j^u)\to\infty$ as $u\to\infty$;

\item for each $j$ we have $\liminf_{u\to\infty} \left| d_{\alpha_u}\phi_u(y_j^u) \right|\geq\epsilon$.

\end{enumerate}

\begin{lemma} There is some constant $\kappa_d>0$ depending only on $X$, $\omega$ and $I$, such that for any $\epsilon$-bubbling list $\left\{ (y_1^u, \ldots, y_r^u) \right\}$, $r\leq \kappa_b \mc{YMH}_c(C_u)$.
\end{lemma}

\begin{proof} We can claim that there exists $\delta>0$ such that if $(\alpha, \phi): S\to X$ is a twisted holomorphic pair from the unit square $S$ to $X$ with , then 
$$\|\alpha\|_{W^{3, p}(S)}\leq \delta,\ \left\| d_\alpha\phi \right\|_{L^2(S)}\leq \delta \Rightarrow \left\| d_\alpha\phi \right\|_{L^\infty(S)}<\epsilon.$$
If the claim is not true, then there exists a sequence of pairs $(\alpha_j, \phi_j): S\to X$ such that $\left\| \alpha_j \right\|_{W^{3, p}}\to 0$, $\left\| d_{\alpha_j}\phi_j \right\|_{L^2}\to 0$ with $\left\| d_{\alpha_j}\phi_j \right\|_{L^\infty}\geq \epsilon$. By \cite[Lemma 3.4]{Mundet_Tian_2009}, passing to a subsequence, $\phi_j$ converges to some $\phi: S\to X$ in $W^{2,p}$-topology. Hence $\|d\phi\|_{L^2}=0$, i.e. $d\phi\equiv 0$. But $d_{\alpha_j}\phi_j\to d\phi$ in $C^0$-topology, this contradicts with $\left\| d_{\alpha_j}\phi_j \right\|_{L^\infty}\geq \epsilon$.

Now suppose that $\left\{ (y_1^u, \ldots, y_r^u) \right\}_u$ is an $\epsilon$-bubbling list and let $S_u^j\subset N_u$ be a unit square that contains $y_j^u$. It follows from the definition of $\epsilon$-bubbling lists that for $u$ big enough, $S_u^j$ are pairwise disjoint. We then can deduce that $\left\| \alpha_u \right\|_{W^{3, p}(S_u^j)}\to 0$ for each $j$. Then by the previous claim, we must have $\left\| d_{\alpha_u}\phi_u \right\|_{L^2(S_u^j)}\geq \delta$. Hence $\kappa_d=\delta^{-1}$ is qualified for this lemma.
\end{proof}

Then let $\left\{ (y_1^u, \ldots, y_r^u) \right\}_u$ be an $\epsilon$-bubbling list of maximal length. If the list is empty, then we do nothing. If not, suppose $y_j^u=(a_j^u, \theta_j^u)$ and set ${\bf x}_u''={\bf x}_u'\cup \left\{ x_j^u=(a_j^u, 0)|1\leq j\leq r \right\}$. Passing to a subsequence, we may assume that the sequence $(C_u, {\bf x}_u''\cup {\bf z}_u')$ converges to some stable curve $(C'', {\bf x}''\cup {\bf z}')$. The curve $C''$ obtained from $C'$ by adding a sequence of disk connecting bubbles at $w$(see Figure \ref{figure8}).
\begin{figure}[htbp]
\centering
\includegraphics[scale=0.70]{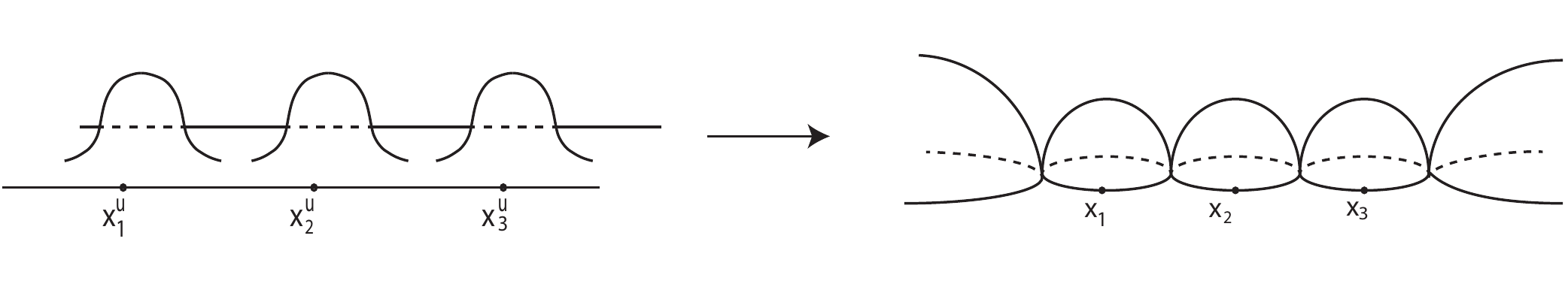}
\caption{An $\epsilon$-bubbling list at a new type-1 boundary node}
\label{figure8}
\end{figure}

We repeat this process for each new boundary node of type 1, and do a similar operation to each old boundary nodeof type 1. Finally, we do a similar operation to each original boundary marked point.

\subsubsection{At type-2 boundary nodes}

Let $w''$ be a new boundary node of type 2 and $N_u= [0, r_u]\times S^1$ be the corresponding neck in $C_u$. Let $(\alpha_u, \phi_u)$ be the corresponding twisted pair of $(A_u, \varphi_u)$ under some local trivialization. We define an {\bf $\epsilon$-bubbling list} to be a sequence of lists $\{(z_1^u, \ldots, z_s^u)\}_u$ such that
\begin{enumerate}
\item $z_j^u=(t_j^u, \theta_j^u)\in N_u$, $t_1^u<t_2^u<\ldots< t_r^u$ and $\lim_{u\to\infty} t_j^u=\infty$;

\item for each $i\neq j$, $t_i^u-t_j^u\to\infty$;

\item for each $j$ we have $\liminf_{u\to\infty} \left| d_{\alpha_u} \phi_u(z_j^u)\right| \geq \epsilon$.
\end{enumerate}

In the same way, we have a uniform upper bound of the number $r$ for an $\epsilon$-bubbling list at any new boundary node of type 2. 

Now let $\{(z_1^u, \ldots, z_r^u)\}_u$ be an $\epsilon$-bubbling list of maximal length. If $r=0$, then we do nothing. Otherwise, we distinguish two cases(by taking a subsequence)
\begin{itemize}
\item $\lim_{u\to\infty} r_u- t_r^u=\infty$. Then we take two points $x_{r+1}^u, x_{r+1}^{u'}\in \{r_u\}\times S^1\subset N_u$ such that $d(x_{r+1}^u, x_{r+1}^{u'})=1$. Then add to the marked point set ${\bf z}_u'$ the points $z_1^u, \ldots, z_r^u$ and add to ${\bf x}_u'$ the points $x_{r+1}^u, x_{r+1}^{u'}$. Then a subsequence of the new sequence of stable curves converges to some stable curve $(C'', {\bf z}''\cup {\bf x}'')$, which looks like Figure \ref{figure9}.
\begin{figure}[htbp]
\centering
\includegraphics[scale=0.80]{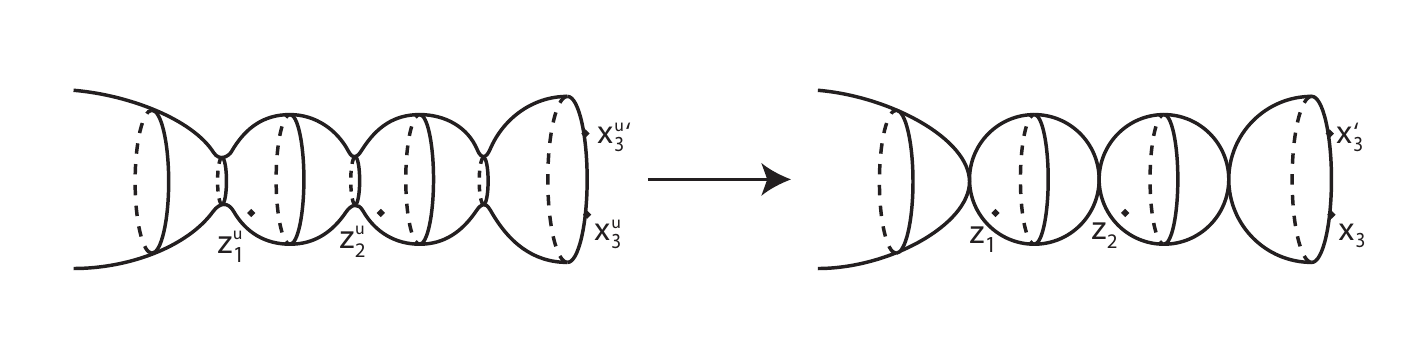}
\caption{An $\epsilon$-bubbling list at a new type-2 boundary node with $r=2$}
\label{figure9}
\end{figure}

\item $\lim_{u\to \infty} r_u-t_r^u<\infty$. Then take $x_r^u= (r_u, \theta_r^u)\in N_u$ and $x_r^{u'}=(r_u, \theta_r^u+1)$. Add to the marked point set ${\bf z}_u'$ the points $z_1^u, \ldots, z_{r-1}^u$ and add to ${\bf x}_u'$ the points $x_r^u, x_r^{u'}$. Then a subsequence of the new sequence of stable curves converges to some stable curve $(C'', {\bf z}''\cup {\bf x}'')$, which looks like Figure \ref{figure10}.
\begin{figure}[htbp]
\centering
\includegraphics[scale=0.80]{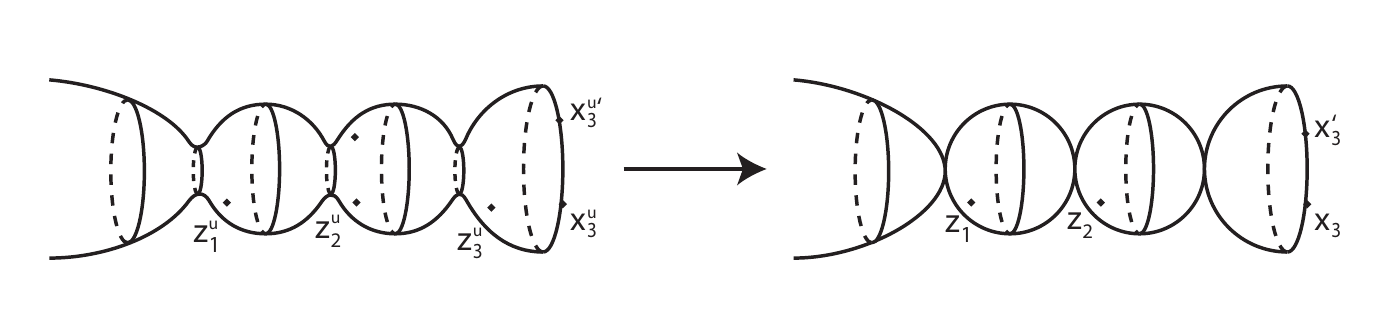}
\caption{An $\epsilon$-bubbling list at a new boundary node with $r=3$}
\label{figure10}
\end{figure}

\end{itemize}

\subsection{Constructing the limit, first part}

After all the above process, we obtain a stable curve $(C', {\bf x}'\cup {\bf z}')$. We have ${\bf x}\subset {\bf x}'$(resp. ${\bf z}\subset {\bf z}'$) which are the limit ${\bf x}_u$(resp. ${\bf z}_u$). We may assume that for any $u$, there is a $(C', {\bf x}\cup {\bf z})$-admissible complex structure $I_u$ and smoothing parameters $\{\delta_u\}$, and biholomorphism $\xi_u': C_u\to C'(I_u, \delta_u)$. Moreover, $I_u$ converges to $I$ on $C'$.

\subsubsection{Constructing the bundle, the connection and the section}

Take a compact set $K\subset C'\setminus {\bf z}$ disjoint from nodes, denote $\iota_{K, u}: K\to C'(I_u, \delta_u)$ the canonical inclusion with image $K_u$. Since we have applied the previous processes, we can assume that
$$
\sup_u\sup_{K_u} \left| F_{A_u} \right|<\infty, \ \ \ \sup_u\sup_{K_u} \left| d_{A_u} \varphi_u \right|<\infty.
$$
Then standard arguments can be applied to deduce that there is an $S^1$-principal bundle $P$ over the smooth locus of $C'\setminus {\bf z}$, a smooth connection $A$ on $P$, a smooth section of $P\times_{S^1} X$ and bundle isomorphisms $\rho_{K, u}: P|_K\to \iota_{K, u}^* P_u|_{K_u}$, such that $\rho_{K, u}|_{K'}=\rho_{K', u}$ for $K'\subset K$ and $\rho_{K, u}^*\iota_{K, u}^* (A_u, \varphi_u)$ converges to $(A, \varphi)$ in $C^\infty(K)$. 

The convergence implies that $\ov\partial_A\varphi=0$, and that on each principal components of $C'$, the pair $(A, \varphi)$ satisfies the equation
$$\iota_{\nu} F_A+\mu(\varphi)=c,$$
where $\nu$ is the canonical volume form on $(C^{st}, {\bf x}^{st}\cup {\bf z}^{st})$. It follows that the curvature of $A$ is bounded(since $\mu$ is bounded) and hence $A$ is a meromorphic connection. On the other hand, the restriction of $A$ to each bubble component of $C'$ is flat. Moreover, $\left\| d_A\varphi \right\|_{L^2}$ is finite.

By the removal of boundary singularity (Theorem \ref{removal}), we can assume that the triple $(P, A, \varphi)$ is extended to each pair $y$ and $y'$, where $\{y, y'\}$ is the preimage of any boundary node $w'$ of type 1.

We have to delete the new marked points, which may violate the stability because then there might be unstable connecting bubbles with zero energy. Let $\{O_b\}$ be the collection of those bubbles in $C'$ and define $C$ to be the quotient $C'/\cup O_b$. The marked point set ${\bf x}\cup {\bf z}$ descends to a list of points in $C$ via the quotient map. Then $(C, {\bf x}\cup {\bf z})$ is a prestable curve. This will be the support of our limit. The triple $(P, A, \varphi)$ descends to a triple over $(C, {\bf x}\cup {\bf z})$, which we denote by the same symbols.

\subsubsection{Constructing the gluing data}

The method of constructing gluing data on \cite[p. 1165-1166]{Mundet_Tian_2009} can be used here to construct gluing data at new interior nodes and new boundary nodes of type 1. For new interior nodes, it is the same as in \cite{Mundet_Tian_2009}. So we only define the gluing data at new boundary nodes of type 1.

Recall the definition of $\gamma_{u, w, \epsilon}(\theta): P_y\to P_y'$, which only requires the convergence of the underlying curve. Let $d$ be a distance function on ${\rm Iso}(P_y, P_{y'})$. By the convergence of the connection, we have that
$$
\lim_{\epsilon\to 0} \left(\sup_{0<\epsilon_1<\epsilon_2<\epsilon} \limsup_{u\to\infty} d(\gamma_{u, w, \epsilon_1}(\theta), \gamma_{u, w, \epsilon_2}(\theta)) \right)=0.
$$
By Stokes' theorem, we also have that for $\theta_1, \theta_2\in [0, \pi]$,
$$
\lim_{\epsilon\to 0} \left(\limsup_{u\to\infty} d(\gamma_{u, w, \epsilon}(\theta_1), \gamma_{u, w, \epsilon}(\theta_2))\right)=0.
$$
Thus there exists a unique isomorphism $g_w: P_y\to P_{y'}$ such that
$$
\lim_{\epsilon\to 0}\left(\limsup_{u\to\infty} d(g_w, \gamma_{u, w, \epsilon}(\theta))\right)=0
$$
for all $\theta\in [0, \pi]$. We define the gluing data at the node $w$ to be $g_w$, and the convergence condition for gluing data at the new node $w$ is automatically satisfied.

\subsection{Constructing the limit, second part}

We construct chains of gradient segments here. The construction of chains of gradient segments at interior nodes is given in \cite[Section 8.7, 8.8]{Mundet_Tian_2009}. Then we only consider the case at type-2 boundary nodes. 

Let $w''$ be such a node of $C$. If $w''$ is a new node, and let $\lambda \in i{\mb R}$ be the residue of $A$ at $w''$, with respect to some trivialization. Then we distinguish two cases:
\begin{itemize}
\item $\lambda\notin \Lambda_{cr}$. Then the limit orbit of $\phi$ at $w''$ is in the fixed point set. Let $N_u=[0, r_u]\times S^1$ be the neck which shrink to $w''$ and $N_u^\tau= [\tau, r_u]\times S^1$. Theorem \ref{thm40} implies that
$$
\lim_{\tau\to \infty} \left(  \limsup_{u\to \infty} {\rm diam}_{S^1}(\varphi_u(N_u^\tau)) \right)=0.
$$
Then we can define the chain of gradient segment to be the degenerate one.

\item $\lambda\in \Lambda_{cr}$. In this case the chain of gradient segments is given essentially by Theorem \ref{thm4.1}, \ref{thm4.2} and \ref{thm26}, which is the same as the construction in \cite[Section 8.7]{Mundet_Tian_2009}.

\end{itemize}

If $w''$ is an old node, then the construction of the chain of gradient segments at $w''$ is just a tautology of the definition of the convergence of $(c, L)$-STHM's near old type-2 boundary nodes.

It is easy to check that one end of the chain of gradient segments lies in $L$, which satisfies our boundary condition in this case.

\subsection{The limit is a $(c, L)$-STHM}

So far we have constructed every ingredient necessary for a $(c, L)$-STHM. However, the definitions of $(c,L)$-STHM and the convergence are so long that we want to check every condition to make sure the limit is indeed a $(c, L)$-STHM and the subsequence does converge to the limit. This and the next subsection also serve as a summary of our construction.

The first 5 conditions of being a $(c, L)$-STHM are automatically satisfied by the limit ${\mc C}$. The stability condition is also established when we construct the underlying curve of ${\mc C}$.

\subsubsection{Critical holonomy at interior marked points}

Suppose $z_u\in {\bf z}_u$ converges to $z\in {\bf z}$. Take $\epsilon$ small and $D(\epsilon)$ to be the closed disk centered at $z$ with radius $\epsilon$, which can be identified canonically to a subset $D(\epsilon)_u$ of $C_u$. For big enough $u$, $z_u$ is in the interior of $D(\epsilon)_u$ and it contains no other marked point. Let $h_{u, \epsilon}$ be the holonomy of $A_u$ along the loop $\partial D(\epsilon)_u$ and $h_\epsilon$ be the holonomy of $A$ along $\partial D(\epsilon)$. Let $h_u$ be the limit holonomy of $A_u$ around $z_u$. Then we have
\begin{align}\label{holonomy}
h_{u,\epsilon} \exp\left( \int_{D(\epsilon)_u} F_{A_u} \right)=h_u.
\end{align}
By the convergence, we know that $h_{u, \epsilon}$ converges to $h_\epsilon$, while since $F_{A_u}$ is uniformly bounded, $\exp(\int_{D(\epsilon)_u} F_{A_u})$ converges to $\exp(\int_{D(\epsilon)} F_A)$. Hence the left hand side of (\ref{holonomy}) converges to the limit holonomy of $A$ around $z$. While since the set of critical holonomies is discrete, the limit holonomy of $A$ around $z$ must be critical and equal to $h_u$ for large $u$.

\subsubsection{Monotonicity of chains of gradient segments}

The monotonicity states that, for an interior node or type-2 boundary node $w\in C^{st}$, the chains of gradient segments connecting the connecting bubbles should be all downward or all upward. In the interiod node case, it has a proof in \cite[Section 8.7]{Mundet_Tian_2009}; in the case of type-2 boundary node, it follows from a similar estimate. We omit the details.

\subsubsection{Stablizer of chains of gradient segments} This condition is also automatically satisfied when the monnotone chains of gradient segments appear.

\subsubsection{Matching conditions}

Among various matching conditions, the only one which is not quite clear is that at boundary nodes of type 1.

If $w'\in C$ is an old boundary node of type 1, then the convergence $\rho_{K, u}^*\iota_{K, u}^*(A_u, \phi_u)$ to $(A, \phi)$ can be extended over the preimages of $w'$ in the normalization of $C$. By passing to a subsequence, we can also assume that the gluing data at $w'$ converges. Hence the matching condition at $w'$ is preserved in the limit.

Now suppose $w'\in C$ is a new boundary node of type 1. Let $y, y'$ be its preimages in the normalization $C'$. Recall the setting and the diagram
$$\begin{CD}
P_y\times [0, \pi] @> \iota_\epsilon>> P_{y, \epsilon} @> \rho_u>> P_{u, y, \epsilon}\\
@V g_{w'}\times r VV @. @VV g_{u, w', \epsilon}V\\
P_{y'}\times [0, \pi] @> \iota_\epsilon' >> P_{y', \epsilon} @> \rho_u >> P_{u, y', \epsilon}.
\end{CD}
$$

The gluing data at $w'$ is defined to be the limit
$$g_{w'}:=\lim_{\epsilon\to 0}\left(\lim_{u\to \infty} \gamma_{u, w', \epsilon}(\theta)\right).
$$

Let $\phi_{y, \epsilon}$(resp. $\phi_{y', \epsilon}$) to be the restriction of $\phi$ to $P_{y, \epsilon}$(resp. $P_{y', \epsilon}$). We have that for any $p\in P_y$ and $\theta\in [0, \pi]$
\begin{align}
\lim_{\epsilon\to 0} d\left(\phi_y(p), \phi_{y, \epsilon}\circ \iota_\epsilon(p, \theta)\right)=0
\end{align}
and that $\lim_{u\to \infty} \phi_{u, y, \epsilon}\circ \rho_u\to \phi_{y, \epsilon}$. Moreover,
$$\lim_{\epsilon\to 0} \limsup_{u\to\infty} d\left( {\rm Im}\phi_{u, y, \epsilon}, {\rm Im}\phi_{u, y', \epsilon} \right)=0.$$
Hence for any $p\in P_y$, $\theta\in [0, \pi]$,
\begin{multline*}
d\left( \phi_y(p), \phi_{y'}(g_w(p)) \right) \leq d \left( \phi_y(p), \phi_{y, \epsilon}\circ \iota_\epsilon(p, \theta)\right)+ d \left( \phi_{y, \epsilon}\circ \iota_\epsilon(p, \theta), \phi_{u, y, \epsilon}\circ \rho_u\circ \iota_\epsilon(p, \theta) \right)\\
+ d \left(\phi_{u, y, \epsilon}\circ \rho_u\circ \iota_\epsilon(p, \theta), \phi_{u, y', \epsilon}\circ g_{u, z, \epsilon}\circ \rho_u\circ \iota_\epsilon(p, \theta) \right)\\
+d \left( \phi_{u, y', \epsilon}\circ g_{u, z, \epsilon}\circ \rho_u\circ \iota_\epsilon(p, \theta), \phi_{y',\epsilon}\circ \rho_u^{-1}\circ g_{u, z, \epsilon}\circ \rho_u\circ \iota_\epsilon(p, \theta) \right)\\
+ d \left(\phi_{y',\epsilon}\circ \rho_u^{-1}\circ g_{u, z, \epsilon}\circ \rho_u\circ \iota_\epsilon(p, \theta), \phi_{y'} \circ \gamma_{u, w, \epsilon}(p, \theta) \right)+ d \left( \phi_{y'}\circ \gamma_{u, w, \epsilon}(p, \theta), \phi_{y'}\circ g_w(p) \right).
\end{multline*}
Apply $\limsup_{u\to \infty}$ to the right hand side, one obtains
\begin{multline*}
d\left( \phi_y(p), \phi_{y'}(g_w(p)) \right)
\leq d \left( \phi_y(p), \phi_{y, \epsilon}\circ \iota_\epsilon (p, \theta)\right)+ \limsup_{u\to\infty} d \left(\phi_{y}\circ \iota_\epsilon(p, \theta), \phi_{u, y', \epsilon}\circ g_{u, z, \epsilon}\circ \rho_u\circ \iota_\epsilon(p, \theta) \right)\\
+ \limsup_{u\to \infty} d \left(\phi_{y',\epsilon}\circ \rho_u^{-1}\circ g_{u, z, \epsilon}\circ \rho_u\circ \iota_\epsilon(p, \theta), \phi_{y'} \circ \gamma_{u, w, \epsilon}(p, \theta) \right)
+ \limsup_{u\to \infty} d \left( \phi_{y'}\circ \gamma_{u, w, \epsilon}(p, \theta), \phi_{y'}\circ g_w(p) \right)\\
\leq d \left( \phi_y(p), \phi_{y, \epsilon}\circ \iota_\epsilon(p, \theta)\right)+ d \left(\phi_{y}\circ \iota_\epsilon(p, \theta), \phi_{y'}(P_{y', \epsilon})\right)+ d\left( \phi_{y'}(P_{y', \epsilon}), \phi_{y'}(P_{y'})\right).
\end{multline*}
Then let $\epsilon\to 0$, we see that $\phi_y=\phi_{y'}\circ g_w$.

\subsection{The sequence does converge to the limit}

We check if the subsequence and the limit $(c, L)$-STHM satisfy the definition in Subsection \ref{defnofconvergence}.
\begin{enumerate}
\item The convergence of the underlying curve is automatic. 

\item The convergence of the connections and sections is also automatic.

\item The convergence near the interior marked points is satisfied since we have found all bubbles near interior marked points.

\item The convergence near the new interior nodes is essentially covered in \cite{Mundet_Tian_2009}.

\item The convergence near the new boundary nodes of type 1.

\begin{itemize}

\item The convergence of the gluing data is automatic.

\item (\ref{914}) is satisfied because we have found all possible bubbles near new boundary nodes of type 1, and (\ref{914}) follows by Corollary \ref{cor58}.

\end{itemize}

\item The convergence near new boundary nodes of type 2 is essentially the same as the convergence near new interior nodes.

\item The convergence near old interior nodes is covered in \cite{Mundet_Tian_2009}.

\item The convergence near old boundary nodes of type 1 is automatic by the construction.

\item The convergence near old boundary nodes of type 2 is essentially the same as the convergence near old interior nodes.

\end{enumerate}

So the conditions of convergence are all satisfied.

	\bibliography{symplectic_ref}
	
	\bibliographystyle{amsplain}

	{\it 
	Guangbo Xu\\
	
	Department of mathematics, Princeton University, Fine Hall, Washington Road, Princeton, NJ 08544, USA\\
	
	Email: guangbox@math.princeton.edu
	
	}

	\end{document}